\documentclass[11pt]{amsart}

\usepackage[svgnames,x11names]{xcolor}
\usepackage{amssymb,amsfonts}
\usepackage[pagebackref,colorlinks,urlcolor={DarkOliveGreen4},
linkcolor={DarkSlateBlue},
citecolor={Chocolate4}]{hyperref}
\usepackage[capitalize]{cleveref}
\usepackage{amssymb,textcomp}
\usepackage{graphicx}
\numberwithin{equation}{section}
\usepackage{enumerate}

\usepackage[utf8]{inputenc}
\usepackage[T1]{fontenc}

\usepackage[mathscr]{eucal}

\usepackage[a4paper, twoside=false, vmargin={2cm,3cm}, includehead]{geometry}

\newtheorem{lemma}{Lemma}[section]
\newtheorem{theorem}[lemma]{Theorem}
\newtheorem*{theorem*}{Theorem}
\newtheorem{corollary}[lemma]{Corollary}

\newtheorem*{question*}{Question}
\newtheorem{proposition}[lemma]{Proposition}
\newtheorem*{proposition*}{Proposition}

\newtheorem{problem}{Problem}
\newtheorem*{problem*}{Problem}

\theoremstyle{definition}
\newtheorem{definition}{Definition}[section]
\newtheorem*{claim*}{Claim}

\newtheorem*{remark}{Remark}
\newtheorem*{remarks}{Remarks}

\theoremstyle{plain}
\newtheorem*{namedthm}{\namedthmname}
\newcounter{namedthm}
\makeatletter
	\newenvironment{named}[2]
	{\def\namedthmname{#1}
	\refstepcounter{namedthm}
	\namedthm[#2]\def\@currentlabel{#1}}
	{\endnamedthm}
\makeatother
\newcommand{\lE}{\mathbb{E}^{\log}}

\newcommand{\C}{{\mathbb C}}
\newcommand{\E}{{\mathbb E}}
\newcommand{\D}{{\mathbb D}}

\newcommand{\N}{{\mathbb N}}
\renewcommand{\P}{{\mathbb P}}
\newcommand{\Q}{{\mathbb Q}}
\newcommand{\R}{{\mathbb R}}
\renewcommand{\S}{\mathbb{S}}
\newcommand{\T}{{\mathbb T}}
\newcommand{\Z}{{\mathbb Z}}
\newcommand{\U}{{\mathbb U}}

\newcommand{\CA}{{\mathcal A}}

\newcommand{\CK}{{\mathcal K}}
\newcommand{\CM}{{\mathcal M}}

\begin{document}

\title{Partition regularity of Pythagorean pairs}

\author{Nikos Frantzikinakis}
\address[Nikos Frantzikinakis]{University of Crete, Department of Mathematics and Applied Mathematics, Heraklion, Greece}
\email{frantzikinakis@gmail.com}

\author{Oleksiy Klurman}
\address[Oleksiy Klurman]{University of Bristol, School of Mathematics, Bristol, UK}
\email{lklurman@gmail.com}

\author{Joel Moreira}
\address[Joel Moreira]{University of Warwick, Department of Mathematics, Coventry, UK}		 \email{ joel.moreira@warwick.ac.uk}

\begin{abstract}
	We  address a core  partition regularity problem in Ramsey theory by proving that every finite coloring of the positive integers contains monochromatic  Pythagorean pairs, i.e., $x,y\in \N$ such that  $x^2\pm y^2=z^2$   for some $z\in \N$.
	We also show  that partitions generated by  level sets of  multiplicative functions taking finitely many values
	always contain Pythagorean triples.
	Our proofs  combine known Gowers uniformity properties of aperiodic multiplicative functions with a  novel and rather flexible approach based on  concentration estimates of multiplicative functions.
\end{abstract}

\thanks{The first author was supported  by the Research Grant-ELIDEK HFRI-FM17-1684.}

\subjclass[2020]{Primary: 05D10; Secondary:11N37,11B30,  37A44.}

\keywords{Partition regularity, Pythagorean triples,  multiplicative functions, concentration inequalities,  Gowers uniformity.}

\date{\today}

\maketitle

\setcounter{tocdepth}{1}
\tableofcontents

\section{Introduction and main results}
\subsection{Introduction}
A fundamental problem in Ramsey theory is to determine which patterns must appear in a single cell for every partition of $\N=\{1,2,\dots\}$ into finitely many cells.
%Those configurations that do are called \emph{partition regular}.
A famous example is provided by an early theorem of Schur \cite{Sc16}, which states that every finite partition of $\N$ has a solution to the equation $x+y=z$ where all  variables $x,y,z$ belong to the same cell.
Equations (and systems of equations) that satisfy this property are called \emph{partition regular}.

In 1933, Rado significantly extended Schur's theorem by characterizing all systems of linear equations that are partition regular \cite{R33}.
Polynomial equations however, have proven to be much more difficult to tackle.
In particular, the following notorious problem of Erd\H os and Graham \cite{Gr07,Gr08} remains unsolved.
\begin{problem*}\label{mainproblem}
	Determine whether the equation $x^2+y^2=z^2$ is partition regular.
	%%\footnote{Equivalently, determine whether the equation $x+y=z$ is partition regular on the squares.}
\end{problem*}
Integer solutions to the equation $x^2+y^2=z^2$ are known as \emph{Pythagorean triples}, so the problem is colloquially referred to as the partition regularity problem for Pythagorean triples.
Graham in \cite{Gr07} places the origin of the problem in the  late 70's and offered \$250 for its  solution, noting that ``There is actually very little data (in either direction) to know which way to guess.''  While this was perhaps true a decade ago,
in the last few years there have been some positive developments.
The case where one allows only partitions of $\N$ into two sets was verified in 2016 with the help of a computer search \cite{HKM16}; this endeavor was hailed as the ``longest mathematical proof'' at the time, occupying 200 terabytes of data \cite{L16}.

Pioneering work in non-linear partition regularity goes back to the famous theorems of Furstenberg \cite{Fust} and S\'ark\"ozy \cite{Sa78-1}, culminating in the influential polynomial Szemer\'edi theorem of Bergelson and Leibman \cite{BL96}.
%The disadvantage of these results  is that  they   apply only  to equations with a shift-invariant solution set.
%In the non-shift invariant case, several related equations are known to be or not to be partition regular:
While these results apply only to shift-invariant configurations, there are now also several non-shift invariant configurations that are known to be or not to be partition regular.
Bergelson showed in \cite{Be87} that the equation $x^2+y=z$ is partition regular, and the equation $x^2+y=z^2$ was shown to be partition regular by the third author in \cite{Mo17}. On the other hand, the equation $x+y=z^2$ was shown   not to  be partition regular by  Csikv\'ari,  Gyarmati,  and S\'ark\"ozy in \cite{CGS12} (however, it is partition regular if we restrict to 2-colorings \cite{GL19,P18}).
Resolving an old conjecture, Khalfalah and Szemer\'edi \cite{KS06} showed that the equation
$x+y=z^2$ is partition regular if we only require $x,y$ to be of the same color and allow any $z\in\mathbb{N}.$
Other partition regularity results  of similar flavor can be found in \cite{Al22,BLM21, BJM17,DL18,DR18,DLMS23, L91}.
Lastly, we remark that in the case of more variables, a
result by Chow, Lindqvist, and Prendiville \cite{CLP21} establishes that the equation $x_1^2+x_2^2+x_3^2+x_4^2=x_5^2$ is partition regular (see also \cite{BP17,Ch22} for  related results).

Despite these developments, even the question of whether in any finite partition of $\N$ there is a Pythagorean triple with two terms in the same cell was still open.
We will say informally  that $(x,y)\in\N^2$ is a \emph{Pythagorean pair} if there exists $z\in\N$ such that either
$$
x^2+y^2=z^2 \quad \text{or} \quad  x^2+z^2=y^2.
$$
An attempt to address the question of whether Pythagorean pairs are partition regular was made  by the  first author and Host in \cite{FH17}, where an approach using Gowers uniformity properties and related  decomposition results of multiplicative functions was proposed.
This approach covered pairs $(x,y)$ satisfying, say, the equations $16x^2+9y^2=z^2$ or $x^2+y^2-xy=z^2$, but missed the case of Pythagorean pairs for reasons that we will explain later on.
Extending  these ideas, Sun in \cite{Su18,Su23}  established partition regularity in $(x,y)$  for the equation $x^2-y^2=z^2$, when $\N$  is replaced by the ring of integers of a larger number field, such as the Gaussian integers. However, the methods used there do not  apply to $\N$.

In the present paper, we develop a general approach to  partition regularity questions of pairs, by combining  the  method of \cite{FH17} together with a new input
related to  concentration estimates of multiplicative functions.
As a consequence, we show (among other things) that
Pythagorean pairs (and related pairs) are partition regular (see \cref{T:PairsPartition}) and density regular (see \cref{T:PairsDensity}).
We also show  that partitions generated by  level sets of  multiplicative functions taking finitely many values
always contain Pythagorean triples (see \cref{T:Triples}). The exact statements are given in the following subsections and our proof strategy and comparison with the previous approach in \cite{FH17} is described in Section~\ref{S:ProofStrategy}.

\subsection{Partition and density regularity of Pythagorean pairs}
Our first goal is to prove partition regularity and density regularity results for Pythagorean pairs, a case covered by taking $a=b=c=1$ in the next two results. Our results also answer the first part of Question~3 from \cite{FH17} and Problem~34 from \cite{Fr16}.
\begin{theorem}\label{T:PairsPartition}
	Let $a,b,c\in \N$ be  squares. Then 	for every finite coloring of $\N$ there exist
	\begin{enumerate}
		\item \label{I:PairsPartition1}	distinct  $x,y\in \N$ with the same color and $z\in \N$ such that $ax^2+by^2=cz^2$.
		
		\item \label{I:PairsPartition2} distinct $y,z\in \N$ with the same color and $x\in \N$ such that $ax^2+by^2=cz^2$.
	\end{enumerate}
\end{theorem}
\begin{remarks}
	$\bullet$	In  \cite[Corollary~2.8]{FH17},  part~\eqref{I:PairsDensity1} was covered under the additional restriction that $a+b$ is also a square, thus missing the case of Pythagorean pairs.
	
	$\bullet$ In fact, Theorem~\ref{T:PairsDensity} implies that all four elements  $x,y$ and $y',z'$ in part \eqref{I:PairsPartition1} and \eqref{I:PairsPartition2} respectively can be taken to be of the same color.
	
	$\bullet$ We can also extend \cite[Theorem~2.7]{FH17}, covering more general  homogeneous equations of the form $p(x,y,z)=ax^2+by^2+cz^2+dxy+exz+fyz=0$ where $a,b,c,d,e,f\in \Z$.
	Our method allows us to   show  that if $e^2-4ac$ and $f^2-4bc$ are non-zero squares, then for every finite coloring of the integers there exist distinct monochromatic $x,y$ and an integer $z$ such that
	$p(x,y,z)=0$.\footnote{Arguing as in Step~2 of \cite[Appendix~C]{FH17}
		we get parametrizations for $x,y$ of the form covered in Section~\ref{SS:generalforms}.} In contrast,   \cite[Theorem~2.7]{FH17}  assumes in addition that $(e+f)^2 -4c(a+b+d)$ is a non-zero square.
	
	%%	$\bullet$ We can also give a result similar to \cite[Theorem~2.7]{FH17} covering more %%general, not necessarily diagonal,   homogeneous quadratic equations in three variables %%$P(x,y,z)=0$.
	%%	Our method allows to   show that if the two forms $P(x,0,z)$ and $P(0,y,z)$ have %%non-zero square discriminants, then for every finite coloring of the integers there exist %%distinct monochromatic $x,y$ and an integer $z$ such that
	%%	$p(x,y,z)=0$. In contrast,   \cite[Theorem~2.7]{FH17}  assumes in addition that the  %%discriminant of the form $P(x,x,z)$ is a non-zero square.
	
	$\bullet$ The assumption that $a,b,c\in \N$ are all squares is not  sufficient for partition regularity
	of 	the equation $ax^2+by^2=cz^2$. For example, the equation $x^2+y^2=4z^2$ is not partition regular, so in this case our result is  optimal, as only pairs  and  not triples can be  partition regular.  See Section~\ref{SS:problems} for  more details and conjectural necessary and sufficient conditions for partition regularity of such equations.
\end{remarks}
We establish a stronger density version of these partition regularity results.
It is clear that the set of odd numbers, which has additive density $1/2$, does not contain integers $x,y$ such that $x^2+y^2=z^2$ for some $z\in \N$, ruling out a potential density version using additive density.
On the other hand, since the equation $x^2+y^2=z^2$ is homogeneous, the set of solutions is invariant under dilations, and using a dilation-invariant notion of density turns out to be more fruitful.

To this end, we recall some standard notions. A \emph{multiplicative F\o lner sequence in $\N$} is a sequence $\Phi=(\Phi_K)_{K=1}^\infty$ of finite subsets of $\N$ asymptotically invariant under dilation, in the sense that
$$\forall x\in\N,\qquad\lim_{K\to\infty}\frac{\big|\Phi_K \cap (x\cdot \Phi_K)\big|}{|\Phi_K|}=1.$$
%Our partition regularity results are a direct consequence of the following density versions, in terms of.
An example of a multiplicative F\o lner sequence is given by  \eqref{E:PhiK}.
The \emph{upper multiplicative density} of a set $\Lambda\subset\N$ with respect to a multiplicative F\o lner sequence $\Phi=(\Phi_K)_{K=1}^\infty$ is the quantity
$$\bar{d}_\Phi(\Lambda):=\limsup_{K\to\infty}\frac{\big|\Phi_K\cap \Lambda\big|}{|\Phi_K|},$$
and we write $d_\Phi(\Lambda)$ if the previous limit exists.
We say that $\Lambda\subset\N$ has {\em positive multiplicative density} (or, more precisely, positive upper Banach density with respect to multiplication)
%%, and write $\bar{d}(\Lambda)>0$,
if $\bar{d}_\Phi(\Lambda)>0$ for some multiplicative F\o lner sequence $\Phi$. A finite coloring of $\N$ always contains a monochromatic cell with positive multiplicative density, thus, the next result strengthens Theorem~\ref{T:PairsPartition}.
\begin{theorem}\label{T:PairsDensity}
	Let $a,b,c\in \N$ be  squares.  Then for every $\Lambda\subset \N$ with positive multiplicative density,  there exist
	\begin{enumerate}
		\item \label{I:PairsDensity1}	  distinct $x,y\in \Lambda$ and $z\in\N$ such that $ax^2+by^2=cz^2$.
		
		\item \label{I:PairsDensity2}   distinct $y,z\in \Lambda$ and $x\in\N$ such that $ax^2+by^2=cz^2$.
	\end{enumerate}
\end{theorem}
\begin{remarks}
	$\bullet$ In fact, we prove the following stronger property:
	If $\bar{d}_\Phi(\Lambda)>0$, then there exist a sub-sequence $\Psi$ of $\Phi$ and distinct $x,y\in \N$ such that  $ax^2+by^2=cz^2$ for some $z\in \N$, and
	$$
	d_\Psi\big((x^{-1}\Lambda)\cap (y^{-1}\Lambda)\big)>0.
	$$
	A similar statement also holds with the roles of $x$ and $z$ reversed.	
	
	$\bullet$ If $a+b\neq c$, it is not true that  every $\Lambda\subset \N$ with positive multiplicative density contains  $x,y,z$ such that $ax^2+by^2=cz^2$.
	To see this when $a=b=c=1$ (the argument is similar whenever $a+b\neq c$),
	let $\Phi$ be any %%tempered
	multiplicative  F\o lner  sequence and  $\alpha$ be an irrational such that the sequence
	$(n^2\alpha)$ is equidistributed $\pmod{1}$ with respect to a subsequence $\Phi'$ of $\Phi$ (such an $\alpha$ and $\Phi'$ exist by the ergodicity of the multiplicative action $T_nx=n^2x$, $n\in \N$, defined on $\T$ with its Haar measure). Let
	$\Lambda:=\{n\in \N\colon \{n^2\alpha \}\in [1/5,2/5)\}$, which has positive upper  density with respect to $\Phi'$. If $x,y,z\in \Lambda$, then  $\{(x^2+y^2)\alpha \}\in [2/5,4/5)$ and  $\{z^2\alpha \}\in [1/5,2/5)$, hence we cannot have $x^2+y^2=z^2$. This example was shown to us  by V.~Bergelson.
\end{remarks}
We remark that the previous results also resolve the first part of Problem 3  in \cite{FH17} and also Problem 6 in \cite{FH17}. The latter implies   that the starting point in S\'ark\"ozy's theorem~\cite{Sa78} (or the  variant in \cite{KS06} dealing with the equation $x+y=n^2$) can be taken to be a square, as the following result shows.
\begin{corollary}
	For every finite coloring of $\N$  there exist
	\begin{enumerate}
		\item \label{I:PairsCorollary1}
		distinct $m,n\in \N$  such that the integers $m^2$ and $m^2+ n^2$ have the same color.
		
		\item \label{I:PairsCorollary2} distinct  $m,n\in \N$  such that the integers $m^2$ and $n^2- m^2$ have the same color.
	\end{enumerate}
\end{corollary}
To prove part~\eqref{I:PairsCorollary1},  let $C_1,\ldots, C_k$  be a coloring of $\N$. Using part~\eqref{I:PairsPartition2} of  Theorem~\ref{T:PairsPartition} for the coloring $C_i':=\{n\in \N\colon n^2\in C_i\}$, $i=1,\ldots, k$, we deduce that there exist $i_0\in \{1,\ldots, k\}$ and $x,z\in C'_{i_0}$ such that
$x^2+y^2=z^2$. Then $x^2,z^2\in C_{i_0}$. Letting $m:=x$ and $n:=y$ we get that $m^2, m^2+n^2\in C_{i_0}$.  The proof of part~\eqref{I:PairsCorollary2} is similar and uses
part~\eqref{I:PairsPartition1} of  Theorem~\ref{T:PairsPartition}.

A coloring $C_1,\ldots, C_k$ of the squares induces a coloring $C'_1,\ldots, C'_k$ of $\N$ in the natural way:
$C_i':=\{n\in \N\colon n^2\in C_i\}$, $i=1,\ldots, k$. Applying Theorem~\ref{T:PairsPartition} for the induced coloring we deduce the following result.
\begin{corollary}
	For every finite coloring of the squares there exist
	\begin{enumerate}
		\item \label{I:PairsCorollary1'}  distinct squares $x,y$ with the same color such that $x+y$ is a square.
		
		\item \label{I:PairsCorollary2'} distinct squares $x,y$ with the same color such that $x-y$ is a square.
	\end{enumerate}
\end{corollary}
\subsection{Pythagorean triples on level sets of multiplicative functions}
Our second objective is to lend support to the hypothesis that Pythagorean triples are partition regular by proving that the level sets of multiplicative functions that take finitely many values
always include Pythagorean triples.
Since the equation $x^2+y^2=z^2$ is homogeneous, one might expect that a presumed counterexample to partition regularity would have ``multiplicative structure'', so \cref{T:Triples} below addresses the most obvious possibilities.
We also remark that Rado's theorem implies that a given linear system of equations is partition regular as soon as it has monochromatic solutions in every coloring realized using a (finitely valued) completely multiplicative function;
%some level set of every multiplicative function that takes finitely many values on the unit circle,
but of course this result does not apply to the Pythagorean equation.

\begin{theorem}\label{T:Triples}
	%%Let $a,b,c\in \N$ be squares,
	Let   $f\colon \N\to \S^1$ be a completely multiplicative function that takes finitely many values. Then  there exist distinct $x,y,z\in \N$  such that
	$$
	x^2+y^2=z^2 \quad \text{and} \quad f(x)=f(y)=f(z)=1.
	$$
\end{theorem}
\begin{remarks}
	%%	$\bullet$	Since $f$ is completely multiplicative and takes  finitely many values,  its %%range   consists  of roots of unity.
	$\bullet$ There is nothing special about the value $1$ in \cref{T:Triples}.
	If $\zeta\in \S^1$ is any other number in  the range of $f$, then since the equation $x^2+y^2=z^2$ is invariant under dilations of the variables $x,y,z$,  we get that there exist distinct $x,y,z\in \N$,   such that
	$$
	x^2+y^2=z^2 \quad \text{and} \quad f(x)=f(y)=f(z)=\zeta.
	$$
	
	$\bullet$ With a bit more effort we can  extend Theorem~\ref{T:Triples} to cover more general  equations of the form
	\begin{equation}\label{E:abcxyz}
		ax^2+by^2=cz^2
	\end{equation}
	where $a,b,c\in \N$ are  squares and we  have either $a=c$, or $b=c$, or $a+b=c$. We outline the additional steps needed to be taken to prove such a result in Section~\ref{SS:abc}. Note that having  one of these three identities satisfied is a necessary condition for the partition regularity of \eqref{E:abcxyz}. For more details and related problems see the discussion in Section~\ref{SS:problems}.
\end{remarks}
Related linear equations $ax+by=cz$ on the level sets of completely multiplicative functions $f:\mathbb{N}\to\{-1,1\}$ have been studied in the works of Br\"udern \cite{Brud08} and more recently by de la Bret\`eche and Granville \cite{BG22}. One consequence of such results \cite[Corollary~2]{BG22}  is that the number of Pythagorean triples $(x,y,z)$ modulo any prime $p\ge 3,$ that is, solutions to $x+y=z$ where $x,y,z\le N<p$ are quadratic residues, is at least $\frac{1}{2}(k'+o_{N\to\infty}(1)) N^2$ where $k'=.005044...$ is a sharp constant.

\subsection{Parametric reformulation of the main results}\label{SS:parametric} To prove our main results it is convenient to  restate them  using  solutions  of \eqref{E:abcxyz}  in parametric form.

Our assumptions give that $a=a_0^2$, $b=b_0^2$, $c=c_0^2$ for some $a_0,b_0,c_0\in \N$. Then a simple computation shows that the following are solutions of $ax^2+by^2=cz^2$:
$$
x=k\, \ell_1\,  (m^2-n^2),\quad
y= k\, \ell_2 \, mn, \quad
z=k \, \ell_3\, (m^2+n^2), \quad  m,n\in\N,
$$
%%\begin{align*}
%%	x&=k\, l_1\,  (m-n) (m+n) ; \\
%%	y&= k\, l_2 \, mn ; \\
%%	z&=k \, l_3\, (m^2+n^2).
%%\end{align*}
where $\ell_1:= a_0bc$, $\ell_2:= 2ab_0c$, $\ell_3:=abc_0$.

So in order to prove Theorem~\ref{T:PairsDensity}   it suffices to establish the following result.
\begin{theorem}\label{T:PairsDensityParametric}
	Suppose that   $\Lambda\subset \N$ satisfies  $\bar{d}_\Phi(\Lambda)>0$
	for some multiplicative F\o lner sequence $\Phi$.
	%%Then
	%%for some  subsequence $\Psi$ of $\Phi$
	%%the following holds:
	Then for every $\ell,\ell'\in \N$ there exist
	%%\footnote{
		%%	{\bf 	I did try to write in detail the argument using $|m^2-n^2|$ in place of %%$(m^2-n^2)$ (as Joel suggested) but this seems to bring more problems than solve. The %%trouble comes when at some point we have  to use the identity
			%%		$$
			%%		|Q(m-n)+1|=\begin{cases}Q|m-n|+1, \quad  &\text{if }\  m\geq n \\
				%%			Q|m-n|-1, \quad &\text{if } m<n  \end{cases}
			%%			$$
			%%		in order to use concentration inequalities.   This is a problem since  when $f\sim \chi$ %%we can no longer  say that $f(Qn-1)$ is close to $1$, since it could be close to $-1$ if %%$\chi(Q-1)=-1$. So I decided to revert back to the original argument that uses only the %%values $m>n$.}}	
	\begin{enumerate}
		\item \label{I:PairsDensity1parametric}	  $m,n\in\N$ with $m>n$ such that
		$\ell\, (m^2-n^2)$ and  $\ell'\, mn$ are distinct and
		$$
		\bar{d}_\Phi\big((\ell\, (m^2-n^2))^{-1}\Lambda\cap (\ell'\, mn)^{-1}\Lambda\big)>0.
		$$
		
		\item \label{I:PairsDensity2parametric} $m,n\in \N$ such that
		$\ell\, (m^2+n^2)$ and  $\ell'\, mn$ are distinct and
		$$
		\bar{d}_\Phi\big((\ell\, (m^2+n^2))^{-1}\Lambda\cap (\ell'\, mn)^{-1}\Lambda \big)>0.
		$$
	\end{enumerate}
\end{theorem}
\begin{remark}
	Since $2(m^2+n^2)=(m+n)^2+(m-n)^2$ and $4mn=(m+n)^2-(m-n)^2$, applying \eqref{I:PairsDensity2parametric} with $2\ell$ in place of $\ell$ and $4\ell'$ in place of $\ell'$, we can add
	%%	$$
	%%	\bar{d}_\Phi\big((\ell\, ((m+n)^2+(m-n)^2))^{-1}\Lambda\cap (\ell'\, (m+n)^2-(m-n)^2 )^{-1}\Lambda %%\big)>0.
	%%	$$
	\begin{itemize}
		\item[(iii)] $m,n\in \N$ such that
		$\ell\, (m^2+n^2)$ and  $\ell'\, (m^2-n^2)$ are distinct and
		$$
		\bar{d}_\Phi\big((\ell\, (m^2+n^2))^{-1}\Lambda\cap (\ell'\, (m^2-n^2))^{-1}\Lambda \big)>0.
		$$
	\end{itemize}
\end{remark}
%%\subsubsection{Parametric reformulation of Theorem~\ref{T:Triples}}
In order to prove Theorem~\ref{T:Triples},  it suffices to establish the following result.
\begin{theorem}\label{T:Triplesparametric}
	%%Let $a,b,c\in \N$ be squares,
	Let  $f\colon \N\to \S^1$ be a completely multiplicative function that takes finitely many values.  Then there exist $k,m,n\in \N$, with $m> n$,  such that the integers $m^2-n^2$, $2mn$, $m^2+n^2$ are distinct and
	\begin{equation}\label{E:ftriples}
		f(k\, (m^2-n^2))=f(k\, 2mn)=f(k\, (m^2+n^2))=1.
	\end{equation}
\end{theorem}

\subsection{Other results}\label{SS:generalforms}
Our methodology is flexible enough to allow us to handle a variety of other
dilation-invariant pairs.  We record a few cases next.
\subsubsection{A question from \cite{DLMS23}}
The next result is related to      \cite[Question~7.1]{DLMS23}. It is only here that we use logarithmic averages
$$
\lE_{m,n\in[N]}:=\frac{1}{(\log{N})^2}\sum_{m,n\in [N]} \frac{1}{mn}
$$
in order to have  access to a result from \cite{Tao15}.
\begin{theorem}\label{T:DLMS}
	Suppose that   $\Lambda\subset \N$ satisfies  $\bar{d}_\Phi(\Lambda)>0$ for some multiplicative F\o lner sequence $\Phi$. Then
	$$
	\liminf_{N\to\infty} \lE_{m,n\in[N]} \, \bar{d}_{\Phi}( (n^2+n)^{-1}\Lambda\cap  (m^2)^{-1}\Lambda)>0.
	$$
\end{theorem}
\begin{remark}
	Our method also implies the   following ergodic version of the previous result, as  posed in \cite{DLMS23}, using Ces\`aro  instead of logarithmic averages:
	If $(T_g)_{g\in\N}$ is a measure-preserving action of $(\N,\times)$ on a probability space $(X,\mu)$ and  $A\subset X$ is measurable with $\mu(A)>0$, then
	$$
	\liminf_{N\to\infty} \lE_{m,n\in[N]} \, \mu(T_{n^2+n}^{-1}A\cap T_{m^2}^{-1}A)>0.
	$$
	This follows from  property \eqref{E:positiveM'} that  we  prove below.
\end{remark}
Our argument also allows us to replace $n^2+n$ and $m^2$ by $n^2+an$ and $m^r$ respectively, where $r\in \N$ and $a$ is a non-zero integer.
%%\footnote{In fact, if we knew that for every aperiodic $f\in \CM$ we have $\lim_{N\to\infty} \lE_{n\in %%[N]}\, f(n)\cdot f(n+1)=0$, then I believe we would get that for every sequence $a\colon \N\to N$ we %%have (for $\mu(A)>0$)
	%%	$$
	%%	\liminf_{N\to\infty} \lE_{m,n\in[N]} \, \mu(T_{n^2+n}^{-1}A\cap T_{a(m)}^{-1}A)>0.
	%%	$$}
The proof of Theorem~\ref{T:DLMS}  follows closely the  argument used to prove
part~\eqref{I:MainPythPairs2} of Theorem~\ref{T:MainPythPairs}. We will  outline this argument  in Section~\ref{SS:DLMS}.
\subsubsection{General linear forms}
We can also prove variants of \cref{T:PairsDensityParametric} that cover  more general patterns of the form
$$
(k\, L_1(m,n)\cdot   L_2(m,n), k \, L_3(m,n)\cdot L_4(m,n)),
$$
where $L_i(m,n)=a_im+b_in$ for some $a_i\in \N$, $b_i\in \Z$, $i=1,2,3,4$, and at least one of the forms, say $L_4(m,n)$, is not a rational multiple of the others.

Suppose we want to show, under the previous assumptions,   that if $\Lambda\subset \Z$ satisfies $\bar{d}_\Phi(\Lambda)>0$ for some multiplicative F\o lner sequence $\Phi$, then there exist  $m,n\in \Z$ such that $L_1(m,n)\cdot   L_2(m,n)$ and $L_3(m,n)\cdot   L_4(m,n)$ are distinct integers and satisfy
$$
\bar{d}_{\Phi}\big((L_1(m,n)\cdot  L_2(m,n))^{-1}\Lambda\cap (L_3(m,n)\cdot  L_4(m,n))^{-1}\Lambda\big)>0.
$$
Without loss of generality we can assume that $b_4\neq 0$. By making the substitution $m\mapsto b_4\, m$ and $n\mapsto n-a_4\, m$ (an operation that preserves our assumptions about the forms $L_i$) we can assume that $a_4=0$. Since the form $L_4$  is not a rational multiple of $L_i$  for $i=1,2,3$, we have $a_i\neq 0$ for $i=1,2,3$. We do another substitution $n\mapsto a_1\, a_2\, a_3 \, n$. We then factor out $a_i$ from the linear form $L_i$ for $i=1,2,3$. We see that it is sufficient  to consider the case where the $L_i$ are integer multiples of forms satisfying $a_1=a_2=a_3=1$ and  $a_4=0$, $b_4\neq 0$.
%%Furthermore, since the forms are pairwise independent we have that $b_1, b_2, b_3$ are %%distinct non-zero integers.
Making a last
substitution $m\mapsto m-b_3 \, n$, we get that it suffices to  prove that
$$
\bar{d}_{\Phi}\big( (\ell\,  (m+an)\cdot  (m+bn))^{-1}\Lambda\cap (\ell' \,  m\, n )^{-1}\Lambda\big)>0
$$
whenever $\ell,\ell'\in \N$ and $a,b\in \Z$.
This case can be covered by repeating the argument used to prove \cref{T:PairsDensityParametric} (which covers the case $a=1,b=-1$)  without any essential change.

\subsubsection{More general expressions and averages}
The methods used to establish part~\eqref{I:PairsDensity2parametric} of Theorem~\ref{T:PairsDensityParametric}, would also allow us to cover
patterns of the form
$$
\Big(k\, (m^2+n^2)^r\prod_{i=1}^lL_i(m,n),\  k\, \prod_{i=1}^{l'}L_i'(m,n)\Big),
$$
where $k\in\N$, $l,l',r\in \Z_+$ are such that $|l|+|l'|>0$,\footnote{The case $l=l'=0$ is covered in \cite[Theorem~1.5]{DLMS23}.} and at least one of the linear forms $L_i, L_i'$ is not a rational multiple of the others.
It should also be possible to cover
variants of \cref{T:MainPythPairs} below in which the averages over squares $\E_{m,n\in [N]}$ are replaced
by averages over discs, i.e., $\E_{m^2+n^2\leq N}$.
However, we do not pursue these directions here.

%%There are a few differences though.  The first is  that a weight
%%similar to $\tilde{w}_\delta(m,n)$ introduced in the proof of  part~\eqref{I:MainPythPairs2} of %%Theorem~\ref{T:MainPythPairs} is only needed when  $l+2r=l'$. The second is that
%%when $l+2r=l'$ the positivity of $\lim_{N\to\infty} \E_{m,n\in [N]} \tilde{w}_\delta(m,n)$
%%is  a bit harder to establish and pased on the fact that if $P,Q\in \Z[t]$ are polynomials with positive %%leading coefficients such that
%%$\deg(P)>\deg(Q)$ then for all large enough $k\in \N$ the equation $P(t)=kQ(t)$ has a positive %%solution
%%(use Bolzano, a negative value for positive $t$ for all large enough $k$ and for any fixed $k$ a %%positive value when $t\to \infty$).
%% And the third is that after some initial maneuvers that are needed in order to assume that one of %%the linear forms is equal to a multiple of $n$, we are eventually led to establishing concentration %%estimates for expressions of the form $f((m+an)^2+(m+bn)^2))$ when $f$ is pretentious. %%Althought the argument used to get similar results when $a=b=0$ is likely to also work in this case, %%this has to verified carefully.

\subsection{Further directions}\label{SS:problems}
Our approach opens the way for  studying several other compelling  partition regularity problems
%%of homogeneous quadratic equations in three variables
that  were previously considered intractable.
We note here some promising directions.

A result of Rado~\cite{R33} implies that if $a,b,c\in \N$, then the linear equation $ax+by=cz$ is partition regular  if and only if
either $a,b,$ or $a+b$ equals $c$, in which case we  say that the triple $(a,b,c)$
{\em   satisfies Rado's condition}.
It follows  that 	 a necessary condition  for the partition regularity of the equation \eqref{E:abcxyz} is that the triple $(a,b,c)$ satisfies Rado's condition.
Perhaps this condition is also sufficient, but very little is known in this direction; in fact, there is no triple $(a,b,c)$ for which the partition regularity  of \eqref{E:abcxyz} is  currently known.
We state a related problem of intermediate difficulty along the lines of  Theorem~\ref{T:Triples}.
\begin{problem}\label{Pr:1}
	Suppose that the triple  $(a,b,c)$ satisfies Rado's condition. Then for any
	completely multiplicative function  $f\colon \N\to \S^1$ taking finitely many values,  there exist distinct $x,y,z\in \N$,   such that
	$$
	ax^2+by^2=cz^2 \quad \text{and} \quad f(x)=f(y)=f(z)=1.
	$$
\end{problem}
Theorem~\ref{T:Triples} solves this problem when $a=b=c=1$ and as we mentioned in the second remark following the theorem, a similar argument  applies to   triples that satisfy Rado's condition and
consist of squares. It would be interesting to solve Problem~\ref{Pr:1}
for some other triples such as  $(1,1,2)$ and $(1,2,1)$. The first one corresponds to the  equation
$$x^2+y^2=2z^2,$$
which was conjectured to be partition regular by  Gyarmati and Ruzsa \cite{GR} and  has parametric solutions of the form
$$
x=k\, (m^2-n^2+2mn), \, y=k\, (m^2-n^2-2mn),\,  z=k\, (m^2+n^2).
$$
The second one corresponds to the equation
$$x^2+2y^2=z^2$$ with parametric solutions of the form
$$
x=k\, (m^2-2n^2), \, y=k\, (2mn), \, z=k\, (m^2+2n^2).
$$
Both parametrizations involve at least two quadratic forms that do not factor into  products of linear forms. This is a  problem for our method, since  a useful variant of
Proposition~\ref{P:aperiodic} is not known in this case, not even if  $f_1,f_2,f_3$ are all equal to  the Liouville function.

Another interesting problem is to relax the conditions on  the coefficients $a,b,c$ in Theorem~\ref{T:PairsPartition}. We mention two representative problems that seem quite challenging.
\begin{problem}\label{Pr:2}
	Show that  	for every finite coloring of $\N$ there exist
	\begin{enumerate}
		\item\label{I:xy2z} distinct  $x,y\in \N$ with the same color and $z\in \N$ such that $x^2+y^2=2z^2$.
		
		\item \label{I:x2yz}  distinct $x,y\in \N$ with the same color and $z\in \N$ such that $x^2+2y^2=z^2$.
	\end{enumerate}
	Show also similar properties with the roles of the variables $y$ and $z$ or $x$ and $z$ reversed.
\end{problem}
\begin{remark}
	More generally, we believe that if for $a,b,c\in \N$  at least one of the integers  $ac,bc, (a+b)c$ is a  square, then
	for every finite coloring of the integers there exist distinct $x,y\in \N$ with the same color and $z\in \N$ such that $ax^2+by^2=cz^2$.  Theorem~\ref{T:PairsPartition} verifies this if both $ac$ and $bc$ are squares.
	We also expect that    if at least one of the integers $bc$, $(c-a)b$ is a  square,  then
	for every finite coloring of the integers there exist distinct $x,z\in \N$ with the same color and $y\in \N$ such that $ax^2+by^2=cz^2$.
	It may also be  that  stronger density regularity results hold,  as in \cref{T:PairsDensity}  and \cref{T:PairsDensityParametric}.
\end{remark}
The broader issue is to find conditions for the polynomials $P,Q\in \Z[m,n]$ such that the following holds: If $\Lambda\subset \N$ satisfies  $\bar{d}_\Phi(\Lambda)>0$
for some multiplicative F\o lner sequence $\Phi$, then  there exist $m,n\in \N$ such that the integers  $P(m,n)$  and $Q(m,n)$ are positive,  distinct, and
$$
\bar{d}_\Phi\big((P(m,n))^{-1}\Lambda\cap (Q(m,n))^{-1}\Lambda\big)>0.
$$
Equivalently, using the terminology from \cite{DLMS23}, the problem is to determine for which polynomials $P,Q\in \Z[m,n]$ we have that  $\{P(m,n)/Q(m,n)\colon m,n\in \N\}$ is a set measurable multiplicative recurrence.

\subsection{Notation} \label{SS:notation}
We let  $\N:=\{1,2,\ldots\}$,  $\Z_+:=\{0,1,2,\ldots \}$, $\R_+:=[0,+\infty)$, $\S^1$ be the unit circle, and $\U$ be  the closed complex unit disk.
With $\P$ we
denote the set of primes and throughout we use the letter $p$ to denote primes.

%5 With $\T$ we denote the one dimensional torus  $\R/\Z$, and we often %identify it   with $[0,1)$.
%% We also often denote elements of $\T$ with  real numbers and we are %%implicitly  assuming that these real numbers are taken  modulo $1$.

For $t\in \R$, we let $e(t):=e^{2\pi i t}$.
For $z\in \C$, with $\Re(z)$, $\Im(z)$  we denote the real  and imaginary parts of $z$ respectively.

For  $N\in\N$, we let $[N]:=\{1,\dots,N\}$. We often denote sequences $a\colon \N\to \U$
by  $(a(n))$, instead of $(a(n))_{n\in\N}$.

If $A$ is a finite non-empty subset of the integers and $a\colon A\to \C$, we let
$$
\E_{n\in A}\, a(n):=\frac{1}{|A|}\sum_{n\in A}\, a(n).
$$

We write $a(n)\ll b(n)$ if for some $C>0$ we have $a(n)\leq C\, b(n)$ for every $n\in \N$.

Throughout this article, the letter $f$ is typically  used for multiplicative functions  and the letter $\chi$ for Dirichlet characters.
%%We always assume that $f(0)=1$.

\subsection*{Acknowledgements}
Part of this research was conducted while the authors were at the
Institute for Advanced Studies at Princeton during parts of the  2022/23 academic year. We are grateful to  the Institute for its hospitality and support and acknowledge the NSF grant DMS-1926686.
For the purpose of open access, the authors have applied a Creative Commons Attribution (CC-BY) licence to any Author Accepted Manuscript version arising from this submission.
 The authors would like to thank James Leng for pointing out an error in a previous version of the manuscript, which led to a modification of the weight sequences used in the proofs.

\section{Roadmap to the proofs}\label{S:ProofStrategy}
This section outlines how we prove  our main results in their parametric reformulation, which is given in Theorems~\ref{T:PairsDensityParametric} and ~\ref{T:Triplesparametric}.

For various facts and notions concerning  multiplicative functions we refer the reader to Section~\ref{SS:multiplicative}.

{\em Throughout, $\ell, \ell'\in \N$ are fixed, and all implicit constants may depend on $\ell, \ell'$.}

\subsection{Reduction of Theorem~\ref{T:PairsDensityParametric} to a positivity property for multiplicative functions}\label{SS:ReductionMultiplicative}
We first use a version of the Furstenberg correspondence principle (see \cite{Be05}) to reformulate the results in an ergodic language.
\begin{theorem}\label{thmmaindynamical}
	Let $\ell, \ell'\in \N$, let $T=(T_n)_{n\in\N}$ be a measure preserving action of $(\N,\times)$ on a probability space $(X,\mu)$,\footnote{Meaning, $T_n\colon X\to X$, $n\in \N$, are invertible measure preserving transformations such that $T_1:=\text{id}$ and $T_{mn}=T_m\circ T_n$ for every $m,n\in\N$}  and let $A\subset X$ be measurable with $\mu(A)>0$.
	Then there exist
	\begin{enumerate}
		\item \label{I:recurrence1}
		$m,n\in\N$  with $m>n$ such that
		$\ell\, (m^2-n^2)$ and  $\ell'\, mn$ are distinct and
		\begin{equation}\label{E:ergodic1}
			\mu(T_{\ell(m^2-n^2)}^{-1}A\cap T_{\ell' mn}^{-1}A)>0.
		\end{equation}
		
		\item \label{I:recurrence2}
		$m,n\in\N$ such that
		$\ell\, (m^2+n^2)$ and  $\ell'\, mn$ are distinct and
		\begin{equation}\label{E:ergodic2}
			\mu(T_{ \ell(m^2+n^2)}^{-1}A\cap T_{\ell' mn}^{-1}A)>0.
		\end{equation}
	\end{enumerate}
	In fact, the set of $m,n\in\N$ for which \eqref{E:ergodic1} and \eqref{E:ergodic2} hold has positive lower density.
\end{theorem}
\begin{remarks}
	$\bullet$	The reduction to the previous  multiple recurrence statement is merely a convenience. It facilitates the purpose of getting a further reduction  to a positivity property for completely multiplicative functions that we describe in \cref{T:MainPythPairs}. Alternatively, one could carry out this last reduction directly, as in \cite[Section~10.2]{FH17}.
	
	$\bullet$
	Using the terminology from \cite{DLMS23}, \cref{thmmaindynamical} can be rephrased as saying that for every $\ell,\ell'\in\N$ both subsets of $\Q^{>0}$
	$$
	\big\{\ell(m^2-n^2)/(\ell'mn)\colon m,n\in\N,\ m>n\big\} \text{ and }\
	\big\{\ell(m^2+n^2)/(\ell'mn)\colon m,n\in\N\big\}
	$$
	are sets of measurable multiplicative recurrence.
\end{remarks}

A function $f\colon \N\to \U$, where $\U$ is the complex unit disk,  is called {\em multiplicative} if
$$
f(mn)=f(m)\cdot f(n)  \quad \text {  whenever  }  (m,n)=1.
$$
It is called {\em completely multiplicative} if the previous equation holds for all $m,n\in\N$.
Let
$$
\CM:=\{f\colon \N\to \S^1 \text{ is a completely multiplicative function}\}.
$$
 Wherever necessary, we  extend multiplicative functions to the non-positive integers in an arbitrary way. %%for example, by letting $f(0):=0$ and %%$f(-n):=f(n)$ for every $n\in \N$.
Throughout, we assume that $\CM$ is equipped with the topology of pointwise convergence.
It easily follows that $\CM$ is a metrizable compact space with this topology.
%%Denote by $\CM$ the set of all completely multiplicative functions $f:\N\to S^1$.
We can identify $\CM$ with the Pontryagin dual of the (discrete) group of positive rational numbers under multiplication. Note that the map
$r/s\mapsto \mu( T_r^{-1}A \cap T_s^{-1}A)$, $r,s,\in \N$, from $(\Q_+,\times)$ to $[0,1]$ is well defined and positive definite. Using a theorem of Bochner-Herglotz, we get that  there exists a finite Borel measure $\sigma$ on $\CM$ such that $\sigma(\{1\})>0$ (in fact, $\sigma(\{1\})\geq \delta^2$, where $\delta=\mu(A)$) and for every $r,s\in\N$,
$$\int_\CM  f(r)\cdot \overline{f(s)}\, d\sigma(f)=\mu(T_r^{-1}A\cap T_s^{-1}A).$$
In particular, we have
%% we can decompose $\sigma=\sigma(\{1\})\delta_1+\sigma'$ for some %%measure\footnote{To be clear, we always mean measures that take values in $[0,1]$.} $\sigma'$ %%and hence
$$\mu(T_{\ell (m^2-n^2)}^{-1}A\cap T_{\ell'mn}^{-1}A)
=
\int_\CM f(\ell (m^2-n^2))\cdot \overline{f(\ell'mn)}\, d\sigma(f)
$$
for every $m,n\in \N$ with $m>n$, and
$$\mu(T_{\ell(m^2+n^2)}^{-1}A\cap T_{\ell'mn}^{-1}A)
=
\int_\CM f(\ell(m^2+n^2))\cdot \overline{f(\ell'mn)}\, d\sigma(f)
$$
for every $m,n\in \N$.
Therefore, Theorem~\ref{thmmaindynamical} follows from the following result.

\begin{theorem}\label{T:MainPythPairs}
	Let $\sigma$ be a  positive bounded measure on $\CM$ such that $\sigma(\{1\})>0$ and
	\begin{equation}\label{E:positivers}
		\int_\CM f(r)\cdot \overline{f(s)}\, d\sigma(f)\geq 0\quad  \text{for every } r,s\in \N.
	\end{equation}
	Then  for every $\ell,\ell'\in \N$
	\begin{enumerate}
		\item \label{I:MainPythPairs1} we have
		\begin{equation}\label{E:sigmapositive1par}
			\lim_{N\to\infty} \E_{m,n\in[N], m>n}\int_{\CM} f(\ell(m^2-n^2))\cdot \overline{f(\ell'mn)}\, d\sigma(f)>0.
		\end{equation}

		\item \label{I:MainPythPairs2} we have
		\begin{equation}\label{E:sigmapositive2par}
			\liminf_{N\to\infty} \E_{m,n\in[N]}\int_{\CM} f(\ell (m^2+n^2))\cdot \overline{f(\ell'mn)}\, d\sigma(f)>0.
		\end{equation}
	\end{enumerate}
\end{theorem}
\begin{remark}
	The limit in \eqref{E:sigmapositive1par} exists by  \cite[Theorem~1.4]{FH16} and the bounded convergence theorem,\footnote{The statement of  \cite[Theorem~1.4]{FH16} does not have the restriction
		$m>n$ in the averaging, but the argument used there also covers this case without essential changes.} however the limit in \eqref{E:sigmapositive2par} may not always  exist.%\footnote{{\bf The previous example I gave (Nikos) was stupid since the conditions were incompatible. To get non-convergence, one needs to use complex valued $f$ and the argument needs more than a few lines to explain. Not worth doing this, so I skip it.}}
	%%To construct an example take $f_0\colon \N\to \S^1$ such that    $f_0\sim 1$ and satisfies %%$f_0(p)=1$ for $p\in4\Z+1$ and $\sum_{p\in \P\cap (4\Z+3)} \Im(f_0(p))/p=\infty$. Then %%$f_0(m^2+n^2)=1$ for every $m,n\in \N$ and the limit 	
	%%	$\lim_{N\to\infty} \E_{m,n\in[N]} f_0(mn)=	\lim_{N\to\infty} (\E_{m\in[N]} f_0(m))^2$ does not %%exist.  Then for $\sigma=\delta_{f_0}$ the limit in \eqref{E:sigmapositive2par} does not exist.
\end{remark}
The reduction up to this point is similar to that in \cite{FH17}. The methods in \cite{FH17} were only able to address a variant of \eqref{I:MainPythPairs1}  in which the expressions under the integral were products of linear factors and were ``pairing up''  when $n=0$ and becoming non-negative.\footnote{For instance, to establish partition regularity of pairs $x,y$ that satisfy the equation $16x^2+9y^2=z^2$, it suffices to study  averages  of  $f(m(m+3n))\cdot \overline{f((m+n)(m-3n))}$  for $f\in \CM$. The key convenient property   these expressions have is that they are  non-negative when $n=0$.}
This positivity property is not shared by the expressions in \eqref{E:sigmapositive1par} (and \eqref{E:sigmapositive2par}), which is the main reason why it was not possible to deal with Pythagorean pairs in \cite{FH17}.
To overcome this obstacle, we do not use a decomposition result that covers all elements of $\CM$ simultaneously (as was the case in  \cite{FH17}), but rather work separately with aperiodic and pretentious multiplicative functions (these notions are defined in \cref{SS:multiplicative}). In particular, coupled with some measurability properties, this allows us to exploit the uniform concentration estimates of Propositions~\ref{P:concentration1} and \ref{P:concentration2}, which are not shared by all elements of $\CM$.
%So although we still use the main result in \cite{FH17} to show that the contribution of the aperiodic multiplicative functions is negligible, we use a substantially different and much more flexible  approach to cover the contribution of the pretentious ones.
We  outline  our approach in the next subsections.

\subsection{Proof plan for part~\eqref{I:MainPythPairs1} %(``easier pairs")
of Theorem~\ref{T:MainPythPairs}}\label{SS:Plain1i}
We prove Theorem~\ref{T:MainPythPairs} by taking an average over the grid
$$
\{(Qm+1,Qn):m,n\in\N\},
$$
where $Q\in\N$ is chosen depending only on $\sigma$.
In view of \eqref{E:positivers} it suffices to prove positivity in \eqref{E:sigmapositive1par} when the average is taken along this subset of  pairs.
With $\ell,\ell'\in \N$ fixed, we introduce the following notation: for $\delta>0$, $f\in\CM$, and $Q,m,n\in\N$, let
\begin{equation}\label{E:AdfQmn}
	A_\delta(f,Q;m,n):=w_{\delta}(m,n)\cdot f\big(\ell\,((Qm+1)^2-(Qn)^2)\big)\cdot \overline{f\big(\ell'\, (Qm+1)Qn\big)},
\end{equation}
where $w_\delta:\N^2\to[0,1]$ is the weight defined in \eqref{E:weight1} of  Lemma~\ref{L:Sdelta} for reasons that will become clear in a moment (at a first reading the reader could just take $w_\delta=1$).  We also remark that the weight $w_\delta$ is supported on the set $\{m,n\in \N\colon m>n\}$, so to compute $A_\delta(f,Q;m,n)$ we only have to compute $f$ on positive integers.
Then  part~\eqref{I:MainPythPairs1} of Theorem~\ref{T:MainPythPairs} follows immediately from the next result, the fact that $0\leq w_\delta(m,n)\leq 1$, and the positivity property \eqref{E:positivers}.
\begin{theorem}\label{T:mainpairs1}
	Let $\sigma$ be a Borel probability measure on $\CM$ such that $\sigma(\{1\})>0$ and \eqref{E:positivers} holds.
	Then there exist $\delta_0>0$ and $Q_0\in\N$ such that
	\begin{equation}\label{E:sigmapositive1}
		\lim_{N\to\infty} \E_{m,n\in[N]}\int_{\CM} A_{\delta_0}(f,Q_0;m,n)\, d\sigma(f)>0.
	\end{equation}
\end{theorem}
%\begin{remark}
%	The values of  $\delta_0>0$ and $Q_0\in\N$  depend on $\sigma$.
	%% but not on  $\ell, \ell'$.
%\end{remark}

% \begin{theorem}\label{thm_mainpairs2}
	%     Let $\sigma$ be a Borel probability measure on $\CM$ such that $\sigma(\{1\})>0$. % and \eqref{E:positivers} holds.
	% 	Then there exists $Q\in\N$ such that
	% 	\begin{equation}\label{E:sigmapositive2.1}
		% 	\liminf_{N\to\infty} \E_{m,n\in[N]}\int_{\CM} L^+(f,Q;m,n)\, d\sigma(f)>0.
		% \end{equation}
	% \end{theorem}

%To pursue this strategy, we will use the  properties of aperiodic and pretentious  multiplicative functions established in  Sections~\ref{SS:FH} and \ref{SS:concentration}.

%From now on we will focus on \cref{thm_mainpairs1}.
To analyse the limit in \eqref{E:sigmapositive1} we use the theory of completely multiplicative functions.
When $f$ is aperiodic, the mean values of  $A_\delta(f,Q;m,n)$ vanish for every $Q$. This is a consequence of the following result, which in turn follows from results in \cite{FH17} (see also \cite{M18} for related work), we shall explain later on how.
\begin{proposition}\label{P:aperiodic1}
	Let $f\colon \N\to \U$ be an aperiodic completely  multiplicative function. Then  for every $\delta>0$ and $Q\in \N$  we have
	%%	$$
	%%	\lim_{N\to\infty} \E_{m,n\in [N]} \, {\bf 1}_{S_{\delta, N}}(m,n)\cdot f(m^2-n^2)\cdot %%\overline{f(2mn)}=0.
	%%	$$
	\begin{equation}\label{E:wdmnA}
		\lim_{N\to\infty} \E_{m,n\in [N]} \,
		A_\delta(f,Q;m,n)=0.
		%%		 w_\delta(m,n)\cdot  f(|(Qm+1)^2-(Qn)^2|)\cdot \overline{f((Qm+1)(Qn))}=0.
	\end{equation}
	Furthermore, for every completely multiplicative function $f\colon \N\to \U$ the previous limit exists.
\end{proposition}
Let
\begin{equation}\label{E:pretentious}
	\CM_p=\{f\colon \N\to \S^1\colon f \, \text{ is a pretentious completely multiplicative function}\};
\end{equation}
we show  in Lemma~\ref{L:Borel} that $\CM_p$ is a Borel subset of $\CM$. It follows from Proposition~\ref{P:aperiodic1} and the bounded convergence theorem, that in order to establish
\eqref{E:sigmapositive1} it suffices to show that there exist $\delta_0>0$ and $Q_0\in \N$ such that
\begin{equation}\label{E:sigmapositive1p}
	\lim_{N\to\infty} \E_{m,n\in[N]}\int_{\CM_p} A_{\delta_0}(f,Q_0;m,n)\, d\sigma(f)>0.
\end{equation}
If $f$ is pretentious, then it ``pretends'' to be a twisted Dirichlet character, and thus exhibits some periodicity. We exploit this periodicity by choosing a highly divisible $Q$ for which the averages  of $A_\delta(f,Q;m,n)$ take a much simpler form.
More precisely, we make use of the following concentration estimate, which is an immediate consequence of \cite[Lemma 2.5]{KMPT21}.
\begin{proposition}\label{P:concentration1}
	Let %$K,N\in \N$ and
	$f\colon \mathbb{N}\to\mathbb{U}$ be a multiplicative function such that $f\sim \chi\cdot n^{it}$ for some $t\in \R$ and Dirichlet character $\chi$ with period $q$ (see \cref{SS:multiplicative} for definitions and notation).  Let also % $\Phi_K$ be as in \eqref{E:PhiK} and suppose that
	$K\in\N$ be large enough so that $q$ divides all elements of the set
 \begin{equation}\label{E:PhiK}
	\Phi_K	:=\Big\{\prod_{p\leq K}p^{a_p}\colon K< a_p\leq 2K\Big\}.
\end{equation}
 Then
	$$
	\limsup_{N\to\infty} \max_{Q\in \Phi_K} \E_{n\in [N]}\big|f(Qn+1)- (Qn)^{it}\cdot \exp\big(F_N(f,K)\big)\big|\ll\D(f,\chi\cdot n^{it}; K,\infty)+K^{-1/2},
	$$
	where the implicit constant is absolute and
	\begin{equation}\label{E:FNfQdef}
		F_N(f,K):=\sum_{K< p\leq N} \frac{1}{p}\,\big(f(p)\cdot \overline{\chi(p)}\cdot p^{-it} -1\big).
	\end{equation}
\end{proposition}
\begin{remarks}
	$\bullet$ It is important for our argument that the implicit constant  is independent of $K$ and the quantity $F_N(f,K)$
	does not depend on $Q$ as long as $Q\in \Phi_K$ and $q\mid Q$.
	
  $\bullet$ We will also need the following variant from \cite[Lemma 2.5]{KMPT21}:
  For any fixed $Q\in \N$ such that $q\prod_{p\leq K} p\mid Q$ we have
$$  \limsup_{N\to\infty} \E_{n\in [N]}\big|f(Qn+1)- (Qn)^{it}\cdot \exp\big(F_N(f,K)\big)\big|\ll\D(f,\chi\cdot n^{it}; K,\infty)+K^{-1/2}.
$$

%%  $\max_{Q\in \Phi_K}$ can be replaced by
%%	$\sup_{Q\in\Psi_K}$ where $\Psi_K:=\{ q\cdot \prod_{p\leq %%K}p^{a_i}\colon a_1,\ldots, a_K\in \N\}$.
	
	$\bullet$ If $f\sim \chi\cdot n^{it}$, then  the sequence $A(N):=\sum_{1< p\leq N} \frac{1}{p}\,\big|1-f(p)\cdot \overline{\chi(p)}\cdot p^{-it}\big|$, $N\in \N$,  is slowly varying, in the sense that for a fixed pretentious $f$ we have for every $c\in (0,1)$ that $\lim_{N\to\infty}\sup_{n\in [N^c,N]}|A(n)-A(N)|=0$.\footnote{If $a_p:=1-f(p)\cdot \overline{\chi(p)}\cdot p^{-it}$, $p\in \P$, we note that
		$\sup_{n\in [N^c,N]}|A(n)-A(N)| \leq (B_N\cdot C_N)^{1/2},$ where $B_N:=\sum_{p\in[N^c,N]}\frac{|a_p|^2}{p}$, $C_N:= \sum_{p\in[N^c,N]}\frac{1}{p}$, $N\in \N$.
		The sequence $C_N$ is bounded and $\lim_{N\to\infty}B_N=0$ because
		$\sum_{p\in \P}\frac{|a_p|^2}{p}<+\infty$.}
	Keeping this in mind, if we use  partial summation on the interval $[N^c,N]$ and then  let $c\to 0^+$, we deduce that the main estimate  of \cref{P:concentration1} still holds if we replace $\E_{n\in[N]}$ with $\lE_{n\in[N]}$.
\end{remarks}
In order to establish \eqref{E:sigmapositive1p} we divide the integral into  two parts. The first is supported on multiplicative functions other than the Archimedean characters $(n^{it})_{n\in\N}$, $t\in \R$,  in which case we show using \cref{P:concentration1} that for a highly divisible $Q_0$ the contribution is essentially non-negative. The second is
supported on Archimedean characters.
We show that this part   is positive using our assumption $\sigma(\{1\})>0$ and by taking $\delta_0$ small enough so that the weight $w_{\delta_0}$  neutralizes the effect of the Archimedean characters  that are different from $1$.  To carry out the first part, the key idea is to
average over ``multiplicatively large'' values of $Q$.
More precisely, for each $K\in\N$ let $\Phi_K$ be the set described in \eqref{E:PhiK}.
The sequence $(\Phi_K)$ is a multiplicative F\o lner sequence with the property that, for every $q\in\N$, as soon as $K$ is large enough, every $Q\in\Phi_K$ is divisible by $q$.
It also has the property that for every $Q\in\Phi_K$ and a prime $p\in\P$, we have $p|Q$ if and only if $p\leq K$.
Let also
\begin{equation}\label{E:CA}
	\CA:=\{(n^{it})_{n\in\N}\colon t\in \R\}.
\end{equation}
Note that $\CA$ is  a Borel subset of $\CM$ since it is a countable union of compact sets (we caution the reader that $\CA$ is not closed with the topology of pointwise convergence, in fact, it is dense in $\CM$).
The most important step in establishing property \eqref{E:sigmapositive1p} is the following fact.
\begin{lemma}\label{L:mainvanishing1}
	Let $f\in\CM_p\setminus \CA$,  $\delta>0$, $\ell,\ell'\in \N$, and $\Phi_K$ be as in \eqref{E:PhiK}.
	Then
	$$	\lim_{K\to\infty}\E_{Q\in\Phi_K}\lim_{N\to\infty}\E_{m,n\in[N]}\, A_\delta(f,Q;m,n)=0.
	$$
	(Note that the inner limit exists by Proposition~\ref{P:aperiodic1}.)
\end{lemma}
Roughly, to prove \cref{L:mainvanishing1} we use the concentration estimate of \cref{P:concentration1} to deduce that  for $Q\in \Phi_K$ the average $\E_{m,n\in[N]}\, A_\delta(f,Q;m,n)$ is asymptotically equal to $C_{\ell,\ell'}(K)\cdot \overline{f(Q)}\cdot Q^{it}$ for some $C_{\ell,\ell'}(K)\in \U$   and $t\in \R$. Since $f\not \in \CA$, by \cref{L:Fol0}  the  average of the last expression, taken  over $Q\in \Phi_K$, converges to $0$ as $K\to \infty$.

Using  the previous result, the fact that the limit $\lim_{N\to\infty}\E_{m,n\in[N]}\, A_\delta(f,Q;m,n)$ exists
(by Proposition~\ref{P:aperiodic1}),
and applying the bounded convergence theorem twice, we deduce the following vanishing property.
\begin{corollary}\label{C:mainvanishing1}
	Let $(\Phi_K)$ and $\CA$  be defined by  \eqref{E:PhiK} and \eqref{E:CA} respectively.  Let also $\sigma$ be a Borel probability measure on $\CM_p$. Then	for every $\delta>0$ we have
	$$	\lim_{K\to\infty}\E_{Q\in\Phi_K}\lim_{N\to\infty}\E_{m,n\in[N]}\,\int_{\CM_p\setminus\CA} A_\delta(f,Q;m,n)\, d\sigma(f)=0.
	$$
\end{corollary}
We are left to study the part of the integral supported on $\CA$.
For such functions the limits
$\lim_{N\to\infty}\E_{m,n\in[N]}A_\delta(f,Q;m,n)$ do not depend on $Q$, and so the previous argument will not help. It is the presence of the  weight  $w_\delta$ that will allow us to prove the following positivity property.
\begin{lemma}\label{L:mainpositive1}
	Let $\sigma$ be a Borel probability measure on $\CM$ such that $\sigma(\{1\})>0$ and $\CA$ be as in \eqref{E:CA}.
	Then there exist $\delta_0,\rho_0>0$, depending only on $\sigma$,  such that
	\begin{equation}\label{E:mainpositive1}
		\liminf_{N\to\infty}\inf_{Q\in \N}\Re\Big( \E_{m,n\in[N]}\int_{\CA} A_{\delta_0}(f,Q;m,n)\, d\sigma(f)\Big)\geq\rho_0.
	\end{equation}
	\begin{remark}
		The weight $w_\delta(m,n)$ is introduced to force positivity in this case, since for some choices of $\ell, \ell'$ and measures $\sigma$, the unweighted expressions  have negative real parts. However,  rather miraculously, if $\ell=1$ and $\ell'=2$ (which is the case to  consider for Pythagorean pairs), we get positivity even in the unweighted case, and
		a somewhat simpler argument applies. We do not pursue this approach here though because it lacks generality.
	\end{remark}
\end{lemma}
Finally,  we will see how the previous results allow us to reach our goal, which is to
prove  Theorem~\ref{T:mainpairs1}, thus completing the proof of part~\eqref{I:MainPythPairs1} of Theorems~\ref{T:PairsDensity}
and \ref{T:MainPythPairs}.
\begin{proof}[Proof of Theorem~\ref{T:mainpairs1} assuming Proposition~\ref{P:aperiodic1}, Corollary~\ref{C:mainvanishing1},  and Lemma~\ref{L:mainpositive1}.]
	By combining Corollary~\ref{C:mainvanishing1}  and Lemma~\ref{L:mainpositive1}, we deduce
	that there exist $\delta_0, \rho_0>0$, depending only on $\sigma$, such that
	$$
	\liminf_{K\to\infty}\E_{Q\in\Phi_K}	\lim_{N\to\infty} \E_{m,n\in[N]}\int_{\CM_p} A_{\delta_0}(f,Q;m,n)\, d\sigma(f)\geq  \rho_0.
	$$
	(There is no need to take the real part on this expression since it is real.)
	From this we immediately deduce that \eqref{E:sigmapositive1p}  holds for some $Q_0\in \N$.   As we also explained before,  this fact, together with  Proposition~\ref{P:aperiodic1},  imply
	\eqref{E:sigmapositive1} via the bounded convergence theorem, completing the proof.
\end{proof}

To establish Theorem~\ref{T:mainpairs1}, it remains  to prove Proposition~\ref{P:aperiodic1},  Lemma~\ref{L:mainvanishing1} (Corollary~\ref{C:mainvanishing1} is an immediate consequence), and Lemma~\ref{L:mainpositive1}. We do this in Section~\ref{S:PythPairs1}.

\subsection{Proof plan for part~\eqref{I:MainPythPairs2} of Theorem~\ref{T:MainPythPairs}} \label{SS:Plan1ii}
The general strategy is similar to that used to prove part~\eqref{I:MainPythPairs1} of Theorem~\ref{T:MainPythPairs}, but there are two major differences.
The first is the required concentration estimate, which is given in Proposition~\ref{P:concentration2} below.
Unlike \cref{P:concentration1}, this result is new and of independent interest, and its proof occupies a considerable portion of the argument.
The second difference is that the limit in \eqref{E:sigmapositive2par} may not exist, which
causes additional technical problems.

%From now on we will focus on \cref{thm_mainpairs1}.

Arguing as before, we get that  part~\eqref{I:MainPythPairs2} of Theorem~\ref{T:MainPythPairs}  follows from the following positivity property.
\begin{theorem}\label{T:mainpairs2}
	Let $\sigma$ be a Borel probability measure on $\CM$ such that $\sigma(\{1\})>0$ and \eqref{E:positivers} holds.
	Then there exists $\delta>0, c\in \R$ such that
 \begin{equation}\label{E:sigmapositive2}
		\liminf_{N\to\infty} \E_{m,n\in[N]}\, \tilde{w}_{\delta,c}(m,n)\cdot \int_{\CM}f\big(\ell \big(m^2+n^2\big)\big)\cdot \overline{f\big(\ell'\, m n\big)}\, d\sigma(f)>0,
	\end{equation}
 where $\tilde{w}_{\delta,c}(m,n)$ is the weight defined in \eqref{E:weight2} of Lemma~\ref{L:Sdelta}.
%%	\begin{equation}\label{E:sigmapositive2}
%%		\liminf_{N\to\infty} \E_{m,n\in[N]}\int_{\CM} %%B_{\delta_0}(f,1;m,n)\, d\sigma(f)>0.
%%	\end{equation}
\end{theorem}
%%\begin{remark}
%	Initially, we  only  establish a slight variant of this, where $1$ is replaced by $Q_N$, where
%	$Q_N$ varies with $N$, but   belongs to a finite set for every $N\in \N$.  We  explain at the end of this subsection how this seemingly weaker  positivity property implies \eqref{E:sigmapositive2}.
%\end{remark}
Again, to analyse the limit in \eqref{E:sigmapositive2}, we use the theory of completely multiplicative functions. We introduce the following notation: for $\delta>0, c\in \R$, $f\in\CM$, and $Q,m,n\in\N$, let
\begin{equation}\label{E:BdfQmn}
	B_{\delta,c}(f,Q;m,n):=\tilde{w}_{\delta,c}(m,n)\cdot f\big(\ell \big((Qm+1)^2+(Qn)^2\big)\big)\cdot \overline{f\big(\ell'\, (Qm+1) (Qn)\big)}.
\end{equation}
If $f$ is aperiodic, we have the following result, which we will deduce from the results in \cite{FH17}.
\begin{proposition}\label{P:aperiodic2}
	Let $f\colon \N\to \U$ be an aperiodic multiplicative function. Then  for every $\delta>0, c\in \R$ and $Q\in \N$  we have
	%%	$$
	%%	\lim_{N\to\infty} \E_{m,n\in [N]} \, {\bf 1}_{S_{\delta, N}}(m,n)\cdot f(m^2-n^2)\cdot %%\overline{f(2mn)}=0.
	%%	$$
	\begin{equation}\label{E:wdmnB}
		\lim_{N\to\infty} \E_{m,n\in [N]} \,
		B_{\delta,c}(f,Q;m,n)=0.
		%%		 w_\delta(m,n)\cdot  f(|(Qm+1)^2-(Qn)^2|)\cdot \overline{f((Qm+1)(Qn))}=0.
	\end{equation}
\end{proposition}
\begin{remark}
	It follows that \eqref{E:wdmnB} also holds even  if  $Q$ depends on $N$, but its values are taken from a finite subset of $\N$.
\end{remark}
%Unlike  the case  of the averages \eqref{E:wdmnA}, the averages \eqref{E:wdmnB} do not always converge. This will cause us technical difficulties that will  be explained later.

If $f$ is pretentious we will crucially use the following concentration estimate (which is a direct consequence of a more general result proved in Section~\ref{S:concentration2}) to analyse the average \eqref{E:sigmapositive2}.
It features a version of the pretentious distance that only considers primes\footnote{The reason we only need primes $\equiv1\bmod4$ is that these are the primes that split in the splitting field of $m^2+n^2$. In a subsequent work we extended these techniques to obtain concentration estimates to general binary quadratic forms.} congruent to $1$ mod $4$:
$$
	\D_1(f,\chi\cdot n^{it}; K,\infty)^2:= \sum_{\substack{K< p, \\ p\equiv   1 \! \! \! \pmod{4}}} \frac{1}{p}\, (1-\Re(f(p)\cdot \overline{\chi(p)} \cdot p^{-it})).
	$$
\begin{proposition}\label{P:concentration2}
	Let %$K,N\in \N$ and
	$f\colon \mathbb{N}\to\mathbb{U}$ be a multiplicative function such that $f\sim \chi\cdot n^{it}$ for some $t\in \R$ and Dirichlet character $\chi$ with period $q$. Let  also  $\Phi_K$ be as in \eqref{E:PhiK} and suppose that
	$K$ is large enough  so that, say, $\D_1(f,\chi\cdot n^{it}; K,\infty)\leq 1$ and $q$ divides all elements of $\Phi_K$. Then
	\begin{multline*}
		\limsup_{N\to\infty}  \max_{Q\in \Phi_K}	\E_{m,n\in [N]}\, \big|f\big((Qm+1)^2+(Qn)^2\big)- Q^{2it} \cdot (m^2+n^2)^{it}\cdot   \exp\big(G_N(f,K)\big)\big|\ll  \\
		\D_1(f,\chi\cdot n^{it}; K,\infty)+K^{-1/2},
	\end{multline*}
	where the implicit constant is absolute and
	\begin{equation}\label{E:GNfQdef}
		G_N(f,K):=2\sum_{\substack{ K< p\leq N,\\ p\equiv   1 \! \! \! \pmod{4}}}\, \frac{1}{p}\,(f(p)\cdot \overline{\chi(p)}\cdot p^{-it} -1).
	\end{equation}
	
\end{proposition}
\begin{remarks}
	%%$\bullet$	An equivalent way to state the result is that   if  $Q\in \N$  is such that  $ %%\prod_{p\leq P_0}p\mid Q$ and use as
	%%	$G_N(f,P_0):=2\sum_{\substack{  p\leq N, \, p\nmid Q,\\ p\equiv   1 \! \! \! \pmod{4}}}\, %%\frac{1}{p}\,(f(p)\cdot \overline{\chi(p)}\cdot n^{-it} -1).$
	$\bullet$ 	It is important for our argument that the  implicit constant does not depend on $K$ and  that  $\exp\big(G_N(f,K)\big)$ is the same for all $Q\in \Phi_K$ that are divisible by $q$. It is also important for our applications that we get some uniformity over the  $Q\in \Phi_K$.
	
%%	$\bullet$  Using Proposition~\ref{P:concentration2quanti} %%below, we can also replace the  $\max_{Q\in \Phi_K}$ with
%%	$\sup_{Q\in\Psi_K}$, where $\Psi_K:=\{ q\cdot \prod_{p\leq %%K}p^{a_i}\colon a_1,\ldots, a_K\in \N\}$.
	
	$\bullet$ For the future applications in mind, we prove a somewhat more general and quantitatively more explicit variant, see Proposition~\ref{P:concentration2quanti} below.
\end{remarks}
As in the proof of  Theorem~\ref{T:mainpairs1},
in order to prove Theorem~\ref{T:mainpairs2}  we split the integral into two parts, one that  is supported on Archimedean characters and the other on its complement. To handle the second part, we use the following  result, which is proved using \cref{P:concentration2} and can be compared to Corollary~\ref{C:mainvanishing1}. Again, taking multiplicative averages over the variable $Q$ is a key maneuver, but the non-convergence of  the averages $\E_{m,n\in [N]} \, B_{\delta,c}(f,Q;m,n)$ causes considerable technical difficulties in our proofs.
%%The increased technicallity in the statement is caused by the fact that the averages %%$\E_{m,n\in [N]} \,
%%B_\delta(f,Q;m,n)$ do not always converge.
\begin{proposition}\label{P:vanishing2}
	Let $(\Phi_K)$, $\CA$, $B_{\delta,c}(f,Q; m,n)$  be defined by  \eqref{E:PhiK},  \eqref{E:CA}, \eqref{E:BdfQmn}, respectively, and $\delta>0, c\in \R$.  Let also $\sigma$ be a Borel probability measure on $\CM_p$. Then	
	%%for every $\varepsilon>0$
	%%there exists $T_0>0$, depending only on $\varepsilon$ and $\sigma$,  such that
	%%  for every $T\geq T_0$ we have
	$$
	\lim_{K\to\infty} 	\limsup_{N\to \infty}\Big|\E_{Q\in \Phi_{K}}\,  \E_{m,n\in[N]}	\int_{\CM_p \setminus \CA} \, B_{\delta,c}(f,Q; m,n)\, d\sigma(f)\Big|=0.
	$$	
\end{proposition}
\begin{remark}
	Unlike the case of \cref{C:mainvanishing1}, we cannot pass the limit over $N$ inside the average over $Q$. This will cause some minor problems in our later analysis, which   we will overcome by using the positivity property \eqref{E:positivers} of the measure $\sigma$ (this is why this positivity property is used in the statement of  Theorem~\ref{T:mainpairs2} but not in  Theorem~\ref{T:mainpairs1}).
\end{remark}

We are left to study the   contribution of the set  $\CA$ of Archimedean characters in which case the presence of the weight $\tilde{w}_{\delta,c}$ allows us  to establish positivity by taking $\delta$ small enough and choosing $c$ appropriately.
\begin{lemma}\label{L:mainpositive2}
	Let $\sigma$ be a bounded Borel  measure on $\CM_p$ such that $\sigma(\{1\})>0$ and $\CA$ be as in \eqref{E:CA}.
	Then there exist $\delta_0,\rho_0>0, c_0\geq 0$, depending only on $\sigma$,  such that
	$$
	\liminf_{N\to\infty}\inf_{Q\in \N}	\Re\Big( \E_{m,n\in[N]}\int_{\CA} B_{\delta_0,c_0}(f,Q;m,n)\, d\sigma(f)\Big)\geq\rho_0.
	$$
\end{lemma}
We conclude this section by noting how the previous results allow us to reach our goal, which is to
prove  Theorem~\ref{T:mainpairs2},  thus completing the proof of part~\eqref{I:MainPythPairs2} of Theorems~\ref{T:PairsDensity}
and \ref{T:MainPythPairs}.
\begin{proof}[Proof of Theorem~\ref{T:mainpairs2} assuming  Proposition~\ref{P:aperiodic2}, Proposition~\ref{P:vanishing2},  and Lemma~\ref{L:mainpositive2}]
	We start by combining Proposition~\ref{P:vanishing2}  and Lemma~\ref{L:mainpositive2}.
	We deduce
	that there exist $\delta_0, \rho_0>0, c_0\geq 0$, depending only on $\sigma$, such that
	$$
	\liminf_{K\to\infty}	\liminf_{N\to\infty} \E_{Q\in\Phi_K} \Re\Big(\E_{m,n\in[N]}\int_{\CM_p} B_{\delta_0,c_0}(f,Q;m,n)\, d\sigma(f)\Big)\geq  \rho_0.
	$$
	In this case,
	it is a little bit tricky   to deduce   that  \eqref{E:sigmapositive2} holds. We do it as follows.  The last estimate implies that there exist $K_0\in \N$ and $Q_N\in \Phi_{K_0}$, $N\in \N$,  such that
	$$
	\liminf_{N\to\infty}  \Re\Big(\E_{m,n\in[N]}\int_{\CM_p} B_{\delta_0,c_0}(f,Q_N;m,n)\, d\sigma(f)\Big)\geq  \rho_0/2.
	$$
	Note that since $Q_N$ belongs to a finite set, Proposition~\ref{P:aperiodic2} implies that
	in the last expression we can replace $\CM_p$ with $\CM$. Hence,
	\begin{equation}\label{E:rho}
	\liminf_{N\to\infty}  \E_{m,n\in[N]}\int_{\CM} B_{\delta_0,c_0}(f,Q_N;m,n)\, d\sigma(f)\geq  \rho_0/2.
	\end{equation}
	(The real part is no longer needed since the last expression is known to be real by \eqref{E:positivers}.)

Recall the definition of 	$\tilde{w}_{\delta,c}$ in \eqref{E:weight2}.
Using  the uniform continuity of $F_\delta$ and that  $L(m,n):=\log\frac{(m^2+n^2)}{mn}$ satisfies  $|L(Qm+1,Qn)-L(m,n)|\leq C/m$ for some $C>0$ and all $m,n\in\N$,  it is easy to verify  that  for every $\delta>0, c\in \R$
 $$
 \lim_{N\to\infty} \E_{m,n\in [N]}\sup_{Q\in \N}|\tilde{w}_{\delta,c}(Qm+1,Qn)-\tilde{w}_{\delta,c}(m,n)|=0.
 $$
We deduce that if  in the definition of $B_{\delta_0,c_0}(f,Q_N;m,n)$ given in \eqref{E:BdfQmn} we  replace $\tilde{w}_{\delta,c}(m,n)$ with $\tilde{w}_{\delta,c}(Q_Nm+1,Q_Nn)$,  the limit on  left side of \eqref{E:rho} remains unchanged. Keeping this in mind,
	and since  $Q_N$ takes values in a finite set with upper bound, say $Q_0$, and by the positivity property \eqref{E:positivers}, we have
$$
\tilde{w}_{\delta,c}(m,n)\cdot \int_{\CM}f\big(\ell \big(m^2+n^2\big)\big)\cdot \overline{f\big(\ell'\, m n\big)}\, d\sigma(f)\geq 0
$$
for every $m,n\in\N$,   we deduce that
	$$
	\liminf_{N\to\infty} \E_{m,n\in[N]}\, \tilde{w}_{\delta,c}(m,n)\cdot \int_{\CM}f\big(\ell \big(m^2+n^2\big)\big)\cdot \overline{f\big(\ell'\, m n\big)}\, d\sigma(f)\geq
	\rho_0/(2Q_0^2).
	$$
	This establishes \eqref{E:sigmapositive2} and ends the proof.
\end{proof}

In order to establish Theorem~\ref{T:mainpairs2}, it remains  to prove Proposition~\ref{P:aperiodic2},  Proposition~\ref{P:vanishing2}, and Lemma~\ref{L:mainpositive2}. We do this in Section~\ref{S:PythPairs2}, after having  established Proposition~\ref{P:concentration2} in Section~\ref{S:concentration2}, which is crucially used in  the proof of Proposition~\ref{P:vanishing2}.

\subsection{Proof plan of Theorem~\ref{T:Triplesparametric}}\label{SS:Plan2}
For notational convenience, when we write $\E^*_{k\in \N}$  in the following statements, we mean the limit $\lim_{K\to\infty} \E_{k\in \Phi_K}$, where $(\Phi_K)$ is an arbitrary multiplicative F\o lner sequence, chosen so that all the limits in the following statements  exist. Since our setting will always involve a countable collection of limits, such a F\o lner sequence always exists and can be taken as a subsequence of any given
multiplicative F\o lner sequence.

Our argument is divided into two parts. In the first part we   reduce the problem to a positivity property of pretentious multiplicative functions and in the second part we verify this positivity property.
To carry out the first part, we note that to prove \cref{T:Triplesparametric}, it is only necessary to establish the subsequent averaged version.
\begin{theorem}\label{T:triplesrestated}
	Suppose that the completely multiplicative function  $f\colon \N\to \S^1$ takes finitely many values and $F:={\bf 1}_{\{1\}}$.
	Then
	\begin{equation}\label{E:Ffkmn}
		\liminf_{N\to \infty} \E_{m,n\in [N],m>n}\, \E^*_{k\in \N} \, F(f(k\, (m^2-n^2)))\cdot F(f(k\, 2 mn))
		\cdot  F(f(k\,  (m^2+n^2)))>0.
	\end{equation}
	%%	where $\E^*_{k\in \N}:=\lim_{K\to\infty} \E_{n\in \Phi_K}$ and $(\Phi_K)$ is any %%multiplicative F\o lner  sequence, taken so that all previous limits exist.
\end{theorem}
\begin{remark}
	The ``multiplicative average'' $\E^*_{k\in \N}$ is needed in our analysis  to ``clear out'' some unwanted terms.
\end{remark}
We write  $f=gh$, where $g$ has aperiodicity properties and $h$ is pretentious (see Lemma~\ref{L:decomposition} for the exact statement). Since $f$ is finite-valued, it follows that
$g$ takes values in $d$-roots of unity for some $d\in \N$, hence we have
$$
F\circ g={\bf 1}_{g=1}=\E_{0\leq j <d}\, g^j.
$$
We use the previous facts to analyse the average in \eqref{E:Ffkmn}.
The aperiodic part is covered by the next result, which is a direct consequence of  \cite[Theorem~9.7]{FH17}.
\begin{proposition}\label{P:aperiodic}
	Let $f_1,f_2,f_3\colon \N\to \U$ be  completely multiplicative functions and suppose that either $f_1$ or $f_2$ is aperiodic. Then
	$$
	\lim_{N\to\infty} \E_{m,n\in [N],m>n} \, f_1(m^2-n^2)\cdot f_2(mn)\cdot f_3(m^2+n^2)=0.
	$$
	%%	\begin{equation}\label{E:wdmnB}
		%%		\lim_{N\to\infty} \E_{m,n\in [N]} \, f_1(|(Qm+1)^2-(Qn)^2|)\cdot %%f_2((Qm+1)(Qn))\cdot f_3((Qm+1)^2+(Qn)^2)=0.
		%%	\end{equation}
\end{proposition}
Combining the above and some technical maneuvering,  we  get the following reduction, which completes the first part needed to  prove Theorem~\ref{T:triplesrestated}.
\begin{proposition}\label{P:reduction}
	%%	Let $\ell_1,\ell_2,\ell_3\in \N$.
	Suppose that for every  finite-valued  completely multiplicative function  $h\colon \N\to \S^1$, with $h\sim 1$, and modified Dirichlet character $\tilde{\chi}\colon \N\to \S^1$ (see \cref{SS:multiplicative} for the definitions), we have
	$$
	\liminf_{N\to \infty} \E_{m,n\in [N],m>n}\, \E^*_{k\in \N}\ A(k\, (m^2-n^2))\cdot A(k\, 2 mn)
	\cdot  A(k\,(m^2+n^2))>0,
	$$
	where
	$$
	A(n):=F(h(n))\cdot F(\tilde{\chi}(n)), \quad n\in \N, \quad F:={\bf 1}_{\{1\}}.
	$$
	Then  for every finite-valued completely multiplicative function $f\colon \N\to \S^1$  we have
	$$
	\liminf_{N\to \infty} \E_{m,n\in [N],m>n}\, \E^*_{k\in \N} \, F(f(k\, (m^2-n^2)))\cdot F(f(k\, 2 mn))
	\cdot  F(f(k\,  (m^2+n^2)))>0.
	$$
\end{proposition}
Therefore, it remains to verify the assumption of this result.  For this purpose,
we will make crucial use of the following concentration estimates, which easily follow  from Propositions~\ref{P:concentration1} and \ref{P:concentration2}, as we will see later.
\begin{corollary}\label{C:concentration2}
	Let $f\colon \mathbb{N}\to\mathbb{U}$ be a finite-valued multiplicative function such that $f\sim \chi $ for some Dirichlet character $\chi$  with period $q$. Then for every $\varepsilon>0$ there exists $Q_0=Q_0(f,\varepsilon)\in \N$ such that the following holds: \begin{enumerate}
		\item \label{I:1}  For all $Q\in \N$  such that  $Q_0\mid Q$ 	 we have
		$$
		\limsup_{N\to\infty}	\E_{n\in [N]}\big|f(Qn+1)-  1  \big|\ll \varepsilon,
		$$
		where the implicit constant is absolute.

		\item \label{I:2} For all $Q\in \N$  such that  $Q_0\mid Q$  we have
		%%$$
		%%	\limsup_{N\to\infty}	\big|\E_{m,n\in [N]}\, f\big((Qm+a)^2+(Qn+b)^2\big)-  %%\tilde{\chi(a^2+b^2)  \big|\ll \varepsilon
			%% $$
			$$
			\limsup_{N\to\infty}	\E_{m,n\in [N]}|f\big((Qm+1)^2+(Qn)^2\big)-  1  \big|\ll \varepsilon,
			$$
			where the implicit constant is absolute.
		\end{enumerate}
	\end{corollary}
	Finally, using the previous concentration estimates  and the key maneuver of taking multiplicative averages over $Q\in \N$,  which was also a crucial element in the proof of Theorem~\ref{T:MainPythPairs},  we verify the assumptions of Proposition~\ref{P:reduction}.
	\begin{proposition}\label{P:pretentiousfinite}
		Let   $f\colon \N\to \S^1$ be a finite-valued pretentious multiplicative function and $\tilde{\chi}\colon \N\to \S^1$ be a modified Dirichlet character.  Then
		$$
		\liminf_{N\to \infty} \E_{m,n\in [N],m>n}\, \E^*_{k\in \N}\ A(k\, (m^2-n^2))\cdot A(k\, 2mn)
		\cdot  A(k\, (m^2+n^2))>0,
		$$
		where
		$$
		A(n):=F(f(n))\cdot F(\tilde{\chi}(n)), \quad n\in \N, \quad  F:={\bf 1}_{\{1\}}.
		$$
	\end{proposition}
	Thus, to  prove Theorem~\ref{T:triplesrestated} it remains to verify Propositions~\ref{P:reduction} and \ref{P:pretentiousfinite}. We do this in Sections~\ref{S:triplesreduction} and \ref{S:Triples} (the other results mentioned in this subsection are needed in the proofs of these two results and will also be verified).

	\section{Background and preparation}
	\subsection{Some elementary facts} We will use the following elementary property.
	\begin{lemma}\label{L:lN}
		Let $a\colon \Z\to \U$ be   an even sequence and $l_1,l_2\in \Z$, not both of them $0$. Suppose that for some $\varepsilon>0$ and for some sequence $L_N\colon \N\to \U$ we have
		$$
		\limsup_{N\to\infty} \E_{n\in [N]} |a(n) -L_N|\leq \varepsilon.
		$$
		Then
		$$
		\limsup_{N\to\infty} \E_{m,n\in [N]} |a(l_1m+l_2n)-L_{lN}|\leq 2 l\varepsilon
		$$
		where $l:=|l_1|+|l_2|$.
	\end{lemma}
	\begin{proof}
	We can assume that $\ell_2\neq 0$. 	We have
		\begin{equation}\label{E:l1l2}
			\E_{m,n\in [N]} |a(l_1m+l_2n)-L_{lN}|\leq
			\frac{1}{N^2}\sum_{|k|\leq l N} w_N(k)\, |a(k) -L_{lN}|,
		\end{equation}
		where for $k\in \Z$ we let
		$$
		w_N(k):=|\{(m,n)\in [N]^2\colon l_1m+l_2n=k\}|.
		$$
		For every $k\in\Z$ and $m\in [N]$  there exists at most one $n\in [N]$ for which
		$l_1m+l_2n=k$, hence  $|w_N(k)|\leq N$ for every  $k\in \Z$. Since $a$ is even, we deduce that
		the  right hand side in \eqref{E:l1l2} is bounded by
		$$
		2\,  l \cdot \E_{k\in [lN]}
		|a(k) -L_{lN}|.
		$$
		The asserted estimate now follows from this and our assumption. 	
	\end{proof}
	The next well-known property of multiplicative functions will also be used several times.
	\begin{lemma}\label{L:Fol0}
		%%Let $p_1,p_2,\ldots,$ be the sequence of primes in increasing order.
		Let  $(\Phi_K)$ be a multiplicative  F\o lner  sequence.
		%%$j_K:=\max\{j\in \N\colon p_j\leq K\}$ and
		%%$$
		%%\Phi_K:=\{ p_1^{i_1}\cdots p_{j_K}^{i_K}\colon K\leq i_1,\ldots, i_K\leq 2K\}.
		%%$$
		If $f\colon \N\to \U$ is a completely multiplicative function
		and $f\neq 1$, then
		$$
		\lim_{K\to\infty}\E_{n\in \Phi_K}\, f(n)=0.
		$$
	\end{lemma}
	\begin{proof}
		Since $f\neq 1$ there exists $p\in \mathbb{P}$ such that $f(p)\neq 1$. 	By the definition of $\Phi_K$ we have
		$$
		\lim_{K\to\infty}\frac{|\Phi_K\cap (p\cdot \Phi_K)|}{|\Phi_K|}=1.
		$$
		From this and the fact that $f(pn)=f(p)\cdot f(n)$ we get
		$$
		\E_{n\in \Phi_K}\, f(n)= \E_{n\in p\cdot\Phi_K}\, f(n) +o_{K\to\infty}(1)
		= f(p)\cdot \E_{n\in \Phi_K}\, f(n) +o_{K\to\infty}(1).
		$$
		Since $f(p)\neq 1$, we deduce that $\E_{n\in \Phi_K}\, f(n)=o_{K\to\infty}(1)$.
	\end{proof}

	%%\subsection{Vanishing of unweighted multilinear averages}\label{SS:FH}
	%%In the proof of Theorem~\ref{T:Triples} we will use the following  vanishing property of %%multilinear averages of aperiodic multiplicative functions from \cite[Theorem~9.7]{FH17}.
	%%\begin{theorem}
	%%	\label{T:FH}
	%%	For $s\in \N$
	%%	let the $L_1,\ldots, L_s$ be non-trivial linear forms in two variables with integer %%coefficients. If $s\geq 2$, suppose that $L_1, L_j$ are linearly independent for $j=2,\ldots, %%s$.  Furthermore, let $g\colon \N\to \S^1$  be  an arbitrary completely multiplicative %%function,  let $f_1\colon \N\to \S^1$ be an  aperiodic completely multiplicative function, %%and suppose that both multiplicative functions are extended to  even functions on $\Z$.
	%%	If $s\geq 2$, let also $f_2,\ldots, f_s\colon \Z\to \C$  be arbitrary bounded even functions.
	%%	Then
	%%	$$
	%%	\lim_{N\to \infty} \,\E_{ m,n\in [N]} \, {\bf 1}_{K_N}(m,n) \,
	%%	g(m^2+n^2)\,
	%%	\prod_{j=1}^sf_j(L_j(m,n))=0,
	%%	$$
	%%	where   $K_N$, $N\in \N$, are arbitrary  convex subsets of $[-N,N]^2$.
	%%\end{theorem}
	
	\subsection{Some useful weights}
	In the proof of Theorems~\ref{T:PairsPartition} and \ref{T:PairsDensity} we will utilize weighted averages.  The weights are employed to ensure that
	the averages  $\E_{m,n\in [N]}\, A_\delta(f,Q,m,n)$ and
	$\E_{m,n\in [N]}\, B_{\delta,c}(f,Q,m,n)$, where $A_\delta, B_{\delta,c}$ are as in \eqref{E:AdfQmn}, \eqref{E:BdfQmn} respectively,   have a positive real part if $f$ is an Archimedean character and   $\delta$ is sufficiently small.
	
%	We will now define these weights. If $\delta \in (0,1/2)$, we consider the circular arc with center %$1$ given by
%	$$
%	I_\delta:=\{e(\phi)\colon \phi \in (-\delta,\delta)\}.
%	$$
	\begin{lemma}\label{L:Sdelta}
		For every $\delta\in (0,1/2)$,    let $F_\delta \colon \R\to [0,1]$ be the continuous  trapezoid function that is equal to $1$ on $[-\delta/2, \delta/2]$, equal to  $0$ outside $[-\delta,\delta]$, and linear on the remaining two intervals. Let also $\ell,\ell'\in \N$,
		\begin{equation}\label{E:weight1}
			w_\delta(m,n):=F_\delta \Big(\log\frac{\ell (m^2-n^2)}{\ell'mn}\Big) \cdot {\bf 1}_{m> n}, \quad m,n\in\N,
		\end{equation}
		and for $c\in \R$ let
		\begin{equation}\label{E:weight2}
			\tilde{w}_{\delta,c}(m,n):=F_\delta \Big(\log\frac{\ell (m^2+n^2)}{\ell'mn}-c\Big), \quad m,n\in\N.
		\end{equation}
		%%	and we let $w_\delta(n,n):=0$ for $n\in\N$.
		Then for every $c\geq  \log\frac{2\ell}{\ell'}$ we have
				$$
		\lim_{N\to\infty} \E_{m,n\in [N]}\,  w_\delta(m,n)>0 \ \text{ and }\ 	\lim_{N\to\infty} \E_{m,n\in [N]}\,  \tilde{w}_{\delta,c}(m,n)>0.
		$$
		In particular, if $2\ell\leq \ell'$ we can take $c=0$.
	%Moreover, if $c \leq  \log\frac{2\ell}{\ell'}+M$, the lower bound can be chosen to depend only on $\delta$ and $M$.
 	\end{lemma}
	\begin{remark}
		We opted for a continuous function for $F_\delta$ instead of an indicator function, to make it easier to prove Propositions~\ref{P:aperiodic1'} and \ref{P:aperiodic2'} later on.
	\end{remark}
	\begin{proof}
		We first  cover the weight in \eqref{E:weight1}.
		%%$$
		%%A=\{(u,v)\in [0,1]\times [0,1]\colon u>v>0\}.
		%%$$
		Note that the limit we want to evaluate is equal to
		$$
		\lim_{N\to\infty} \E_{m,n\in [N]}\,
		F_\delta \Big(\log\frac{\ell((m/N)^2-(n/N)^2)}{\ell'(m/N)\cdot (n/N)}\Big)
		\cdot {\bf 1}_{m/N> n/N}.
		$$
		Let
		$\tilde{F}_\delta\colon [0,1]\times [0,1]\to [0,1]$ be given by
		$$
		\tilde{F}_\delta(x,y):=F_\delta\Big(\log \frac{\ell (x^2-y^2)}{\ell'xy}\Big) \cdot {\bf 1}_{x> y}, \quad x,y\in (0,1].
		$$
		Then $\tilde{F}_\delta$ is Riemann integrable on $[0,1]\times [0,1]$ as it is bounded and continuous except for a set of Lebesgue measure $0$.  Hence, the limit we aim to compute exists and is equal to the  Riemann integral
		$$
		\int_0^1\int_0^1 \tilde{F}_\delta(x,y)\, dx\, dy.
		$$
		It remains to show that this integral is positive, and since  $\tilde{F}_\delta$ is non-negative, it suffices to show that $\tilde{F}_\delta$
 does not vanish almost everywhere.

  To verify the non-vanishing property, note that
		if $x= ay$ where
		$a:=\frac{\ell'+\sqrt{(\ell')^2+4\ell^2}}{2\ell}>1$,
		then $x>y$ and  $\ell(x^2-y^2)= \ell'xy$. Hence, 	$\tilde{F}_\delta(x,y)=F_\delta(0)=1$ on the line $x= ay$. Since $\tilde F_\delta$ is continuous in the region $x>y$, this proves that it stays bounded away from zero in a neighborhood of the line $x=ay$ in that region, and hence it does not vanish almost everywhere. This completes the proof for the weight \eqref{E:weight1}.

		The argument   for the second weight \eqref{E:weight2} is very similar, so we only  summarize it.
		Let
		$\tilde{F}_{\delta,c}\colon [0,1]\times [0,1]\to [0,1]$ be given by
		$$
		\tilde{F}_{\delta,c}(x,y):=F_\delta\Big(\log \frac{\ell(x^2+y^2)}{ \ell'xy}-c \Big)\cdot {\bf 1}_{(0,1]\times (0,1]}(x,y).
		$$
		Then the limit we want to evaluate exists and is equal to the  Riemann integral
		$$
		\int_0^1\int_0^1 \tilde{F}_{\delta,c}(x,y)\, dx\, dy.
		$$
		The integral is positive because  $\tilde{F}_{\delta,c}$ is non-negative and does not vanish almost everywhere.  It remains to verify the non-vanishing property for $c\geq  \log\frac{2\ell}{\ell'}$.
Let	  $b:=\ell'/\ell\cdot e^c\geq 2$  and  $a:=\frac{b+\sqrt{b^2-4}}{2}$.
  If $x,y\in [0,1]$ are such that $x= ay$,
		then %%$x\geq y$ and
		  $\ell(x^2+y^2)=e^c\, \ell'xy $, which implies
	 	$\tilde{F}_{\delta,c}(x,y)=F_\delta(0)=1$.
	 	By continuity  $\tilde{F}_{\delta,c}$
	  is bounded away from zero on a neighborhood of the line $x=ay$, hence it does not vanish almost everywhere.
	\end{proof}

	\subsection{Multiplicative functions}\label{SS:multiplicative}
	We record here some basic notions and facts about multiplicative functions that will be used throughout the article.
	\subsubsection{Dirichlet characters} A {\em Dirichlet character} $\chi$  is a periodic completely multiplicative function, and is often thought of as a multiplicative function on $\Z_m$ for some $m\in \N$. In this case,
	$\chi$  takes the value $0$ on integers that are not coprime to $m$, and
	takes values on $\phi(m)$-roots of unity on all other integers, where $\phi$ is the Euler totient function.
	If $\chi$ is a Dirichlet character, we define the {\em modified Dirichlet character}  $\tilde{\chi}\colon \N\to \S^1$    to be the completely multiplicative function satisfying
	%%$\tilde{\chi}(p)=\chi(p)$ if $\chi(p)\neq 0$ and $\tilde{\chi}(p)=1$ if $\chi(p)=0$, $p\in \P$.
	$$
	\tilde{\chi}(p):=
	\begin{cases}
		\chi(p), &\quad \text{if } \chi(p)\neq 0\\
		1, &\quad \text{if } \chi(p)=0.
	\end{cases}
	$$
	We note in passing that  the level sets of modified Dirichlet characters $\tilde{\chi}$, which can be seen as finite colorings of $\N$, are precisely the colorings that appear in Rado's theorem when showing that certain systems of linear equations are not partition regular.
 In particular, a system of linear equations is partition regular if and only if it has a monochromatic solution in any coloring realized by a modified Dirichlet character.%connected to partition regularity of linear equations as follows: a system of linear equations is partition regular, for some choices of $\chi$, turn out to be the colorings of $\N$ that serve as examples for partition non-regularity results of linear equations on the integers.
 %We note that level sets of modified Dirichlet characters $\tilde{\chi}$, for some choices of $\chi$,  turn out to be the sets that serve as examples for partition non-regularity results of linear equations on the integers.
	
	\subsubsection{Distance between  multiplicative functions}\label{SS:distance}
	Following Granville and Soundararajan~\cite{GS07,GS23},  in this and the next subsection, we define a distance and a related notion of pretentiousness  between multiplicative functions.
	If $f,g\colon \N\to \U$ are multiplicative functions and $x,y\in \R_+$ with $x<y$ we let
	\begin{equation}\label{E:Dfgxy}
		\D(f,g; x,y)^2:= \sum_{x< p \leq y} \frac{1}{p}\, (1-\Re(f(p)\cdot \overline{g(p)})).
	\end{equation}
	We also let
	\begin{equation}\label{E:Dfg}
		\D(f,g)^2:=\sum_{p\in \mathbb{P}} \frac{1}{p}\, (1-\Re (f(p)\cdot \overline{g(p)})).
	\end{equation}
	Note that if $|f|=|g|=1$, then
	$$
	\D(f,g)^2=\frac{1}{2}\cdot \sum_{p\in \mathbb{P}} \frac{1}{p}\, |f(p)-g(p)|^2.
	$$
	It can be shown (see \cite{GS08} or \cite[Section~2.1.1]{GS23}) that $\D$ satisfies the triangle inequality
	$$
	\D(f, g) \leq \D(f, h) + \D(h, g)
	$$
	for all  $f,g,h\colon \P\to \U$.
	Also, for all  $f_1, f_2, g_1, g_2\colon \P\to \U$,  we have (see
	\cite[Lemma~3.1]{GS07})
	\begin{equation}\label{E:Df1f2}
		\D(f_1f_2, g_1g_2) \leq \D(f_1, g_1) + \D(f_2, g_2).
	\end{equation}

	\subsubsection{Pretentious   multiplicative functions}\label{SS:pretentious}
	If $f,g\colon \N\to \U$ are multiplicative functions, 	we say that $f$ {\em pretends to be} $g$,
	and write $f\sim g$, if $\D(f,g)<+\infty$.
	It follows from \eqref{E:Df1f2} that if $f_1\sim g_1$ and $f_2\sim g_2$, then $f_1f_2\sim g_1g_2$.
	We say that $f$ is {\em pretentious,} if $f\sim  \chi \cdot n^{it}$ for some $t\in \R$ and Dirichlet character $\chi$, in which case
	%%  $t$ is uniquely determined and
	$$
	\sum_{p\in \mathbb{P}} \frac{1}{p}\, (1-\Re (f(p)\cdot \overline{\chi(p)}\cdot p^{-it}))<+\infty.
	$$
	The value of  $t$ is uniquely determined; this follows
	from \eqref{E:Df1f2} and the fact that $n^{it}\not\sim \chi$
	for every non-zero $t\in \R$  and  Dirichlet character $\chi$ (see for example \cite[Corollary~11.4]{GS23}  or \cite[Proposition~7]{GS08}).
	
	%%Also the Dirichlet character $\chi$ is uniquely determined as long as we assume that %%it is primitive.
	Although real valued or finite-valued multiplicative functions always have a mean value,  we caution the reader that this is not the case  for general multiplicative functions with values on the unit circle. For example, we have
	$$
	\E_{n\in [N]} \, n^{it}=N^{it}/(1+it)+o_N(1),
	$$
	so  we have non-convergent means when $t\neq 0$. But even multiplicative functions satisfying  $f\sim 1$ can have non-convergent means. In particular, if
	$f\sim 1$ is a completely multiplicative function,  then it is known
	(see for example \cite[Theorems~6.2]{E79})  that there exists $c\neq 0$ such that
	$$
	\E_{n\in[N]}\, f(n)= c\cdot e(A(N)) +o_N(1),
	$$
	where $A(N):=	\sum_{p\leq N} \frac{1}{p}\, \Im (f(p))$, $N\in \N$.
	Hence, we have non-convergent means when, for example,
	$$	
	\sum_{p\in \mathbb{P}} \frac{1}{p}\, \Im (f(p))=+\infty,
	$$
	which is the case  if 	
	$f(p):=e(1/ \log\log{p})$, $p\in \P$.  This oscillatory behavior of the mean values of some complex-valued multiplicative functions has to be taken into account and will cause problems in the proofs of some of our main results.
	
	Finally, we record an
	observation that will only be used in the proof of Theorem~\ref{T:Triples}.
	\begin{lemma}\label{L:convergesabsolutely}
		Let $f\colon \N\to \U$ be a pretentious finite-valued multiplicative function. Then $f\sim\chi$ for some Dirichlet character $\chi$ and
		\begin{equation}\label{E:convergesabsolutely}
			\sum_{p\in \P} \frac{1}{p}|1-f(p)\cdot \overline{\chi(p)}|<+\infty.
		\end{equation}
	\end{lemma}
	\begin{remark}
		It can be shown using \eqref{E:convergesabsolutely}
		that  finite-valued  pretentious multiplicative functions always have  convergent means.
	\end{remark}
	\begin{proof}
		Since $f$ is pretentious we have $f\sim \chi \cdot n^{it}$ for some $t\in \R$ and Dirichlet character $\chi$. Then
		$\D(n^{it}, g)<+\infty$ where $g:=f\cdot \overline{\chi}$ is a finite-valued multiplicative function.
  In particular, there exists $d\in\N$ for which $g^d$ is the constant $1$, so from \eqref{E:Df1f2} it follows that $\D(n^{idt}, 1)<+\infty$ , which is turn implies that
		$t=0$ (hence, $f\sim \chi$) and
		$$
		\sum_{p\in \P\colon f(p)\cdot \overline{\chi(p)}\neq 1}\frac{1}{p}<+\infty.
		$$
		Hence,
		$$
		\sum_{p\in \P}\frac{1}{p} |\Im (f(p)\cdot \overline{\chi(p)})|<+\infty.
		$$
		If we combine this with $\D(f,\chi)<+\infty$, we deduce that
		\eqref{E:convergesabsolutely} holds.
	\end{proof}

	\subsubsection{Aperiodic  multiplicative functions} \label{SS:aperiodic}
	We say that a multiplicative function $f\colon \N\to \U$ is {\em aperiodic} if  for every $a,b\in\N$,
	$$\lim_{N\to\infty}\frac1N\sum_{n=1}^N\, f(an+b)=0.$$
	The following well known result  of Daboussi-Delange \cite[Corollary~1]{DD82}  states that a multiplicative function is aperiodic if and only if it is non-pretentious.
	\begin{lemma}\label{lemmamultiplicativestructure}
		Let $f\in\CM$.
		Then either $f\sim \chi\cdot n^{it}$ for some Dirichlet character $\chi$ and $t\in \R$,  or $f$ is aperiodic.
	\end{lemma}
	In our arguments we typically distinguish two cases. One where a multiplicative function is aperiodic, then we show that the expressions we are interested in vanish. The complementary one where the multiplicative function is pretentious is  treated using concentration estimates.

	\subsection{Some Borel measurability results}
	Recall  that $\CM$ is equipped with the topology of pointwise convergence. In the proof of Theorem~\ref{T:MainPythPairs} we require certain Borel measurability properties of subsets of $\CM$ and related maps.
	The second property proved below will only be used in the proof of part~\eqref{I:MainPythPairs2}  of   Theorem~\ref{T:MainPythPairs}.
	
	Recall that 	if $f$ is pretentious, then  there exist a  unique $t=t_f\in \R$ and a  Dirichlet character $\chi$ such that $f\sim \chi \cdot n^{it}$.% It is known that the value of $t$ is unique and  we denote this $t$ by $t_f$.
	\begin{lemma}\label{L:Borel}
		\begin{enumerate}
			\item\label{I:Borel1}	The set $\CM_p$ of pretentious completely multiplicative functions is Borel.
			
			\item \label{I:Borel2} The map $f\mapsto t_f$ from $\CM_p$ to $\R$ is Borel measurable.
			
			%%	\item \label{I:Borel3} 	For every $n\in \N$, the map $f\mapsto \chi_f(n)$   from %%$\CM_p$ to  $\C$  is Borel measurable.
			
			%%		\item \label{I:Borel4}
			%%	For every $T\in \R_+$ the following set is Borel
			%%	$$
			%%	\CM_{p,T}:=\{f\in \CM_p\colon f\sim n^{it} \cdot  \chi \text{ for  some } t\in [-T,T]
			%%\text{ and some Dirichlet character } \chi \}.
			%%	$$
		\end{enumerate}
	\end{lemma}
	\begin{proof}
		We prove \eqref{I:Borel1}.
		For $a,b\in \N$ we  let $M_{a,b}$ be the set of $f\in \CM$ such that
		$$
		\limsup_{N\to\infty} |\E_{n\in [N]} \, f(an+b)|>0.
		$$
		Clearly $M_{a,b}$ is a Borel subset of $\CM$. By Lemma~\ref{lemmamultiplicativestructure}  we have $\CM_p=\bigcup_{a,b\in \N}M_{a,b}$ and the result follows.
		
		We prove \eqref{I:Borel2}.  By  \cite[Theorem~14.12]{Ke12}, it suffices to show that the graph
		$$
		\Gamma:=\{(f,t_f)\in \CM_p\times \R \}
		$$
		is a Borel subset of $\CM_p\times \R$.
		If $\chi_k$, $k\in\N$, is an enumeration of all Dirichlet characters, and
		$$
		\Gamma_k:=\{(f,t_f)\in \CM_p\times \R\colon f\sim \chi_k\cdot n^{it_f}\},
		$$
		then
		$$
		\Gamma=\bigcup_{k\in\N} \Gamma_k.
		$$
		Hence, it suffices to show that for every $k\in\N$ the set $\Gamma_k$ is Borel.
		Note that
		$$
		\Gamma_k:=\{(f,t)\in\CM_p\times \R\colon \D(f,\chi_k\cdot n^{it})<\infty\}.
		$$
		Since for $k\in \N$   the map $(f,t)\mapsto \D(f,\chi_k\cdot n^{it})$ is clearly Borel, the set $\Gamma_k$ is Borel.
		This completes the proof.
	\end{proof}

	\section{Type I Pythagorean pairs}\label{S:PythPairs1}
	As explained in Section~\ref{SS:Plain1i},  in order to  complete the proof of  Theorem~\ref{T:mainpairs1} (and thus of  part~\eqref{I:MainPythPairs1} of Theorem~\ref{T:MainPythPairs}) it remains  to prove  Proposition~\ref{P:aperiodic1}, Lemma~\ref{L:mainvanishing1}, and Lemma~\ref{L:mainpositive1}.  We  do this in this section.
	
	We start with Proposition~\ref{P:aperiodic1}, which we state here in an equivalent form.
	\begin{proposition}\label{P:aperiodic1'}
		Let $f\colon \N\to \U$ be an aperiodic completely multiplicative function, let $\ell,\ell',Q\in\N$ and $\delta>0$. Then, with $w_\delta:\N^2\to[0,1]$ described by \eqref{E:weight1},  we have
		\begin{equation}\label{E:wdmnA'}
			\lim_{N\to\infty} \E_{m,n\in [N]} \, w_\delta(m,n)\cdot  f(\ell((Qm+1)^2-(Qn)^2))\cdot \overline{f(\ell'(Qm+1)(Qn))}=0.
		\end{equation}
		Furthermore, the limit in \eqref{E:wdmnA'} exists for all multiplicative functions $f\colon \N\to \U$.
	\end{proposition}
	\begin{proof}
		Recall that
		$$
			w_\delta(m,n):=F_\delta \Big(\log\frac{\ell (m^2-n^2)}{\ell'mn}\Big) \cdot {\bf 1}_{m> n}, \quad m,n\in\N,
		$$
		where  $F_\delta \colon \R\to [0,1]$ is the continuous function defined  in Lemma~\ref{L:Sdelta}.
	Using that  $\lim_{n\to\infty}(\log(Qn+1)-\log(Qn))=0$, we get that the limit in \eqref{E:wdmnA'}  remains unchanged if we replace  $w_\delta(m,n)$ with  $w_\delta(Qm+1,Qn)$.
	Moreover,  if $L(m,n):=\log\frac{\ell (m^2-n^2)}{\ell'mn}$ for $m,n\in \N$ with $m>n$,
	one easily verifies that for every $Q\in \N$
	$$
	\lim_{T\to\infty}\limsup_{N\to\infty} \E_{m,n\in[N]}  \, {\bf 1}_{|L(Qm+1,Qn)|>T}=0.
	$$
	It follows from the above that in order to verify  \eqref{E:wdmnA'} it suffices to show that for all large enough  $T>0$ we have
	\begin{equation}\label{E:wdmnA''}
		\lim_{N\to\infty} \E_{m,n\in [N]} \, w_{\delta,T}(Qm+1,Qn)\cdot  f( (Qm+1)^2-(Qn)^2)\cdot \overline{f((Qm+1)(Qn))}=0.
\end{equation}
	where  $w_{\delta,T}(m,n):=	 F_{\delta,T}(L(m,n))\cdot  {\bf 1}_{m> n}$ and
$F_{\delta,T}$ is the $(2T)$-periodic extension of $F_\delta\cdot {\bf 1}_{[-T,T]}$.
		Since $F_{\delta,T}$ can be approximated uniformly by trigonometric
  polynomials, using linearity we deduce that it suffices to verify \eqref{E:wdmnA''} with
  $w_{\delta,T}(m,n)$ replaced by $(m^2-n^2)^{it}\cdot (mn)^{-it}\cdot {\bf 1}_{m>n}$ for arbitrary $t\in \R$.
		Hence, in order to establish \eqref{E:wdmnA'} it suffices to show that for every $t\in \R$ we have
		\begin{equation}\label{E:fkQ0}
			\lim_{N\to\infty} \E_{m,n\in [N]} \,   {\bf 1}_{Qm+1> Qn} \cdot f_t((Qm+1)^2-(Qn)^2)\cdot \overline{f_t((Qm+1)(Qn))}=0,
		\end{equation}
		where $f_t(n):=f(n)\cdot n^{it}$, $n\in \N$.
  Note that since the indicator function of an arithmetic progression is a linear combination of Dirichlet characters, in order to establish \eqref{E:fkQ0}, it suffices to show that  for all Dirichlet characters $\chi,\chi'$ we have
  $$
			\lim_{N\to\infty} \E_{m,n\in [N]} \,   {\bf 1}_{m> n} \cdot \chi(m)\cdot \chi'(n)\cdot f_t(m^2-n^2)\cdot \overline{f_t(mn)}=0,
		$$
  or, equivalently, that
  \begin{equation}\label{E:fkQ0'}
			\lim_{N\to\infty} \E_{m,n\in [N]} \,   {\bf 1}_{m> n} \cdot f_t(m^2-n^2)\cdot (\overline{f_t}\cdot\chi) (m)\cdot  (\overline{f_t}\cdot\chi') (n)=0.
		\end{equation}
  Since $f$ is aperiodic, so is $f_t$. Since $f$ is aperiodic, so is $\overline{f_t}\cdot \chi$ (and $\overline{f_t}\cdot \chi'$).
		Combining \cite[Theorem~2.5]{FH17}
		%%, we have %%$\lim_{N\to\infty}\norm{f_k}_{U^3[N]}=0$ for %%every $k\in \Z$. If we combine this with
		and \cite[Lemma~9.6]{FH17},  we deduce
		that 		\eqref{E:fkQ0'} holds, completing the proof.
		
		Finally, to prove  convergence for all multiplicative functions, we argue as before, using the fact that convergence in the case
		$w_\delta=1$  follows from \cite[Theorem~1.4]{FH16}. We note that although \cite[Theorem~1.4]{FH16} only covers the  case without the weight  ${\bf 1}_{m> n}$,  exactly the same argument can be used to cover this weighted variant.
	\end{proof}
	Next we restate and prove Lemma~\ref{L:mainvanishing1}. Recall that $A_\delta$, $\CM_p$, and $\CA$ were defined in \eqref{E:AdfQmn}, \eqref{E:pretentious}, and \eqref{E:CA} respectively.

 \begin{named}{\cref{L:mainvanishing1}}{}
	Let $f\in\CM_p\setminus \CA$,  $\delta>0$, $\ell,\ell'\in \N$, and let $\big(\Phi_K\big)_{K\in\N}$ be the F\o lner sequence described in \eqref{E:PhiK}.
		Then
		\begin{equation}\label{E:lemmaaveragezero'}
			\lim_{K\to\infty}\E_{Q\in\Phi_K}\lim_{N\to\infty}\E_{m,n\in[N]}\, A_\delta(f,Q;m,n)=0.
		\end{equation}
	\end{named}
	% \begin{lemma}\label{L:mainvanishing1'}
	% 	Let $f\in\CM_p\setminus \CA$,  $\delta>0$, $\ell,\ell'\in \N$, and let $\Phi_K$ be the F\o lner set described in \eqref{E:PhiK}.
	% 	Then
	% 	\begin{equation}\label{E:lemmaaveragezero'}
	% 		\lim_{K\to\infty}\E_{Q\in\Phi_K}\lim_{N\to\infty}\E_{m,n\in[N]}\, A_\delta(f,Q;m,n)=0.
	% 	\end{equation}
	%\end{lemma}
	\begin{proof}
		Let $\delta>0$ and  $f\in\CM_p\setminus \CA$. Then for some $t\in \R$ and Dirichlet character $\chi$ we have
		\begin{equation}\label{E:fng}
			f(n)=n^{it} \cdot g(n), \quad \text{where } g\sim \chi, \, g\neq 1.
		\end{equation}
		For reasons that will become clear later, for $\delta>0$ and $Q\in \N$, let
		\begin{equation}\label{E:LtildeL}
			\tilde L_\delta(f,Q):=f(Q)\cdot Q^{-it}\cdot\lim_{N\to\infty} \E_{m,n\in[N]}\, A_{\delta}(f,Q;m,n).
		\end{equation}
		Note that the limit in the definition of $\tilde L_\delta(f,Q)$ exists by the second part of \cref{P:aperiodic1'}.
  The idea to prove \eqref{E:lemmaaveragezero'} is to show that $\tilde L_\delta(f,Q)$  does not depend strongly on $Q$ (it  depends only on the prime factors of $Q$), so that, as a function of $Q$ it is orthogonal to any non-trivial completely multiplicative function with respect to multiplicative averages.
		Since the left hand side  of \eqref{E:lemmaaveragezero'} is the correlation between $\tilde L(f,Q)$ and the completely multiplicative function $Q\mapsto f(Q)\cdot Q^{-it}$, which is non-trivial by \eqref{E:fng}, the conclusion will  follow.
		
		Fix $\varepsilon>0$ and take $K_0=K_0(\varepsilon,f)$ so that
		$$\sum_{p\geq K_0}\frac1p (1-\Re(f(p)\cdot  \overline{\chi(p)}\cdot p^{-it}))+K_0^{-1/2}\leq\varepsilon.$$
		Using Proposition~\ref{P:concentration1} (and noting that the function $K\mapsto \D(f,\chi\cdot n^{it}; K,N)$ is decreasing for any fixed $f$ and $N$), it follows that for every $N>K>K_0$ and $Q\in\Phi_K$,
		\begin{equation}\label{eqconcentrationinequality1}
			\E_{n\in [N]}|f(Qn+1)-(Qn)^{it}\cdot \exp\big(F_N(f,K)\big)|\ll\varepsilon.
		\end{equation}
		%%Note that
		%%\begin{equation}\label{E:Qm-n}
		%%	|Q(m-n)+1|=\begin{cases}Q|m-n|+1, \quad  &\text{if }\  m\geq n \\
			%%		Q|m-n|-1, \quad &\text{if } m<n  \end{cases}.
		%%\end{equation}
		Using  this identity and  Lemma~\ref{L:lN} with $a(n):=f(Qn+1)\cdot (Qn)^{-it}$ and $l_1=1$, $l_2=-1$, it follows that
		\begin{equation}\label{eqconcentrationinequality2}
			\limsup_{N\to\infty}\E_{m,n\in[N],m>n}\Big|f(Q(m-n)+1)-\big(Q(m-n)\big)^{it}\exp\big(F_{2N}(f,K)\big)\Big|\ll\varepsilon.
		\end{equation}
		Using  \eqref{eqconcentrationinequality1} and Lemma~\ref{L:lN} with $a(n):=f(Qn+1)\cdot (Qn)^{-it}$ and $l_1=l_2=1$, it follows that
		\begin{equation}\label{eqconcentrationinequality3}
			\limsup_{N\to\infty}\E_{m,n\in[N]}\Big|f(Q(m+n)+1)-\big(Q(m+n)\big)^{it}\exp\big(F_{2N}(f,K)\big)\Big|\ll\varepsilon.
		\end{equation}
		Combining  \eqref{eqconcentrationinequality1}, \eqref{eqconcentrationinequality2}, \eqref{eqconcentrationinequality3}, and since all terms involved are $1$-bounded,  we deduce that for every $K>K_0$ and $Q\in\Phi_K$,
		\begin{multline*}
			\limsup_{N\to\infty}\E_{m,n\in[N], m>n}\Big|f\big((Qm+1)^2-(Qn)^2\big)\cdot  \overline{f(Qm+1)}
			\\
			-Q^{it}\cdot (m^2-n^2)^{it} \cdot m^{-it} \cdot \exp(2F_{2N}(f,K)) \cdot \overline{\exp(F_N(f,K))}\big)\Big|\ll
   \varepsilon.
		\end{multline*}
		Multiplying by $c_{\ell,\ell'}\cdot w_\delta(m,n)\cdot  \overline{f(Qn)}\cdot Q^{-it}\cdot f(Q)=c_{\ell,\ell'}\cdot w_\delta(m,n)\cdot \overline{f(n)}\cdot Q^{-it}$,
		where $c_{\ell,\ell'}:=f(\ell)\cdot \overline{f(\ell')}$,  we deduce that
		\begin{multline*}
			\limsup_{N\to\infty}\E_{m,n\in[N],m>n}\Big|A_\delta(f,Q;m,n)\cdot Q^{-it}\cdot f(Q)\\
			-c_{\ell,\ell'}\cdot w_\delta(m,n)\cdot (m^2-n^2)^{it} \cdot m^{-it} \cdot  \overline{f(n)}\cdot \exp(2F_{2N}(f,K)) \cdot \overline{\exp(F_N(f,K))}\Big|\ll\varepsilon.
		\end{multline*}
		This implies that, for every $K>K_0$
		\begin{multline*}
			\limsup_{N\to\infty}\sup_{Q\in \Phi_K}\Big|\tilde L_\delta(f,Q)
			-\\ c_{\ell,\ell'}\cdot \E_{m,n\in[N]}\, w_\delta(m,n)\cdot (m^2-n^2)^{it}\cdot m^{-it}\cdot \overline{f(n)}\cdot \exp(2F_{2N}(f,K)) \cdot \overline{\exp(F_N(f,K))}\Big|\ll\varepsilon.
		\end{multline*}
		Since the second term does not depend on $Q$, we conclude that for every $K>K_0$ and $Q,Q'\in\Phi_K$,
		$\big|\tilde L_\delta(f,Q)-\tilde L_\delta(f,Q')\big|\ll\varepsilon$.
		We can choose $\varepsilon$ arbitrarily small by sending $K\to\infty$, so it follows that
		$$\lim_{K\to\infty}\max_{Q,Q'\in\Phi_K}\big|\tilde L_\delta(f,Q)-\tilde L_\delta(f,Q')\big|=0.$$
		For $K\in \N$, let $Q_K$ be any element of $\Phi_K$.  From the last identity and    \eqref{E:LtildeL} it follows that
		$$
		\lim_{K\to\infty}\E_{Q\in\Phi_K} \lim_{N\to\infty}\E_{m,n\in[N]}\, A_{\delta}(f,Q;m,n)=
		\lim_{K\to\infty} \tilde L_\delta(f,Q_K) \cdot \E_{Q\in\Phi_K}\overline{f(Q)}\cdot Q^{it}.
		$$
		By  \eqref{E:fng} we have that
		$Q\mapsto f(Q)\cdot Q^{-it}$ is a non-trivial multiplicative function,  hence the last limit is zero by Lemma~\ref{L:Fol0}. This establishes
		\eqref{E:lemmaaveragezero'} and completes the proof.
	\end{proof}
	
	Lastly, we restate and prove Lemma~\ref{L:mainpositive1}.
	
 \begin{named}{\cref{L:mainpositive1}}{}
	Let $\sigma$ be a Borel probability measure on $\CM_p$ such that $\sigma(\{1\})>0$ and let $\CA$ be as in \eqref{E:CA}.
		Then there exist $\delta_0,\rho_0>0$, depending only on $\sigma$,  such that
		\begin{equation}\label{E:mainpositive1'}
			\liminf_{N\to\infty}\inf_{Q\in \N}\Re\Big( \E_{m,n\in[N]}\int_{\CA} A_{\delta_0}(f,Q;m,n)\, d\sigma(f)\Big)\geq\rho_0.
		\end{equation}
	\end{named}

 % \begin{lemma}\label{L:mainpositive1'}
		
	% \end{lemma}
	\begin{proof}
		Let $a:=\sigma(\{1\})>0$ and  for $\delta>0$ let
		$$
		\mu_{\delta}:=\lim_{N\to\infty}\E_{m,n\in [N]} \, w_{\delta}(m,n).
		$$
		Note that by 	Lemma~\ref{L:Sdelta} we have $\mu_{\delta}>0$.
		For $T\in \R_+$ we  consider the sets
		$$
		\CA_{T}:=\{(n^{it})_{n\in\N} \colon  t\in [-T,T] \}.
		$$
		These sets are 	closed and as a consequence Borel.
		Since
		$\CA_{T}$ increases to $\CA$ as $T\to \infty$,
		and the Borel measure $\sigma$ is finite,   there exists $T_0=T_0(\sigma)>0$ such that
		\begin{equation}\label{E:CAT0}
			\sigma(\CA\setminus \CA_{T_0})\leq \frac{a}{4}.
		\end{equation}
		Note also that since $\lim_{n\to \infty} \sup_{Q\in \N}|\log(Qn+1)-\log(Qn)|=0$,
		we have
		\begin{multline*}
			\lim_{N\to\infty}\sup_{f\in \CA_{T_0}, Q\in \N}\E_{m,n\in [N],m>n}\big| f(\ell ((Qm+1)^2-(Qn)^2))\cdot \overline{f(\ell'(Qm+1)(Qn))}-\\
			f(\ell (m^2-n^2))\cdot \overline{f(\ell' mn)}\big|=0,
		\end{multline*}
		and   by the definition of  $w_\delta$ given in Lemma~\ref{L:Sdelta}, we have
		$$
		\lim_{\delta\to 0^+} 	\limsup_{N\to\infty}\sup_{f\in \CA_{T_0}} \big|\E_{m,n\in [N]} \, w_\delta(m,n)\cdot  f(\ell (m^2-n^2))\cdot \overline{f(\ell'mn)} -
		\E_{m,n\in [N]} \, w_\delta(m,n) \big|=0.
		$$
		We deduce  from  the last two identities that if $\delta_0$ is small enough (depending only on  $T_0$ and hence only on $\sigma$), then for every $Q\in \N$ we have
		$$
		\liminf_{N\to\infty}\inf_{Q\in \N} \Re\Big(\E_{m,n\in[N]}\int_{\CA_{T_0}} A_{\delta_0}(f,Q;m,n)\, d\sigma(f)\Big)\geq  \frac{\sigma(\CA_{T_0})\cdot \mu_{\delta_0}}{2} 	\geq \frac{ a\cdot \mu_{\delta_0}}{2},
		$$
		where  we
		used that $1\in \CA_{T_0}$, hence  $\sigma(\CA_{T_0})\geq \sigma(\{1\})=a$.
		On the other hand, using \eqref{E:CAT0} and the triangle inequality, we get
		$$
		\limsup_{N\to\infty} \sup_{Q\in \N}\Big|\E_{m,n\in[N]}\int_{\CA\setminus \CA_{T_0}} A_{\delta_0}(f,Q;m,n)\, d\sigma(f)\Big|\leq \frac{a\cdot \mu_{\delta_0}}{4}.
		$$
		Combining the last two estimates we deduce that \eqref{E:mainpositive1'} holds with $\rho_0:=\frac{a\cdot \mu_{\delta_0}}{4}$.
	\end{proof}

	\section{Nonlinear concentration estimates}\label{S:concentration2}
	Our goal is to prove the concentration estimate of  Proposition~\ref{P:concentration2}, which is a crucial ingredient in the proof of part~\eqref{I:MainPythPairs2} of Theorem~\ref{T:PairsDensity} and in the proof of Theorem~\ref{T:Triples}.
	In fact, we will prove a more general and quantitatively more explicit statement with further applications in mind.
	
	Let $f,g\colon \N\to \U$ be multiplicative functions and let $\chi$ be a Dirichlet character and $t\in \R.$ For every $K_0\in \N$  we let
	\begin{equation}\label{E:defGNfK0}
		G_N(f,K_0):=2\sum_{\substack{ K_0< p\leq N,\\ p\equiv   1 \! \! \! \pmod{4}}}\, \frac{1}{p}\,(f(p)\cdot \overline{\chi(p)}\cdot p^{-it} -1)
	\end{equation}
	and
	\begin{equation}\label{E:defD1}
		\D_1(f,g; x,y)^2:= \sum_{\substack{x< p\leq y, \\ p\equiv   1 \! \! \! \pmod{4}}} \frac{1}{p}\, (1-\Re(f(p)\cdot \overline{g(p)})).
	\end{equation}
	\begin{proposition}\label{P:concentration2quanti}
		Let $K_0,N\in \N$ and  $f\colon \mathbb{N}\to\mathbb{U}$ be a multiplicative function. Let also $t\in \R$,  $\chi$ be a Dirichlet character with period $q,$   $Q= \prod_{p\leq K_0}p^{a_p}$ for some  $a_p\in \N$,  and suppose that $q\mid Q$.
		If   $N$ is  large enough, depending only on  $Q$,  then for all  $a,b\in \Z$ with  $-Q\leq a,b\leq Q$ and $(a^2+b^2,Q)=1$ we have
		\begin{multline}\label{E:generalf}
			\E_{m,n\in [N]}\, \big|f\big((Qm+a)^2+(Qn+b)^2\big)- \chi(a^2+b^2)\cdot ((Qm+a)^2+(Qn+b)^2)^{it} \cdot \exp\big(G_N(f,K_0)\big)\big|\ll \\ 	(\D_1+\D_1^2)(f,\chi\cdot n^{it}; K_0,\sqrt{N})+
			Q^2\cdot\mathbb{D}_1(f,\chi\cdot n^{it};N,3Q^2N^2)+ Q\cdot \D_1(f,\chi\cdot n^{it}; \sqrt{N},N)+K_0^{-1/2},
		\end{multline}
		where $G_N,\D_1$ are as in   \eqref{E:defGNfK0}, \eqref{E:defD1},  and
		the implicit constant is absolute.
	\end{proposition}
	\begin{remarks}
		$\bullet$	Note that  if $f\sim \chi\cdot n^{it}$,  we have $\lim_{N\to \infty} \D_1(f,\chi\cdot n^{it};N,3Q^2N^2)=0$ and $\lim_{N\to\infty} \D_1(f,\chi\cdot n^{it}; \sqrt{N},N)=0$.
		Hence, renaming $K_0$ as $K$, taking  the max over all $Q\in \Phi_K$,  and then letting $N\to \infty$ in \eqref{E:generalf} gives the estimate in Proposition~\ref{P:concentration2}.
		
		$\bullet$	The averaging over both variables $m,n\in\mathbb{N}$ is crucial for our argument and allows  us to overcome issues with large primes. In fact, by slightly modifying the  example of \cite[Lemma~2.1]{KlurComp}, one can construct completely multiplicative functions (both pretentious and aperiodic) $f:\mathbb{N}\to \{-1,1\},$ such that for every $a\in \Z_+$ the averages
		$$
  \E_{n\in [N]}\, f((Qn+a)^2+1)
  $$
		behave rather ``erratically'' and a similar concentration estimate fails. On the other hand, in \cite{Te24} Ter\"{a}v\"{a}inen proved a version of the concentration estimates for values of $f(P(Qn+a))$, where  $P\in\mathbb{Z}[x]$ is  arbitrary and the  multiplicative functions $f$ satisfies  $f(p)=p^{it}\chi(p)$ for $p>N$ (with a somewhat more particular choice of $Q$).  In our setting, however, we cannot afford to make such assumptions on $f$.
	\end{remarks}
	The proof is carried out in  several steps, covering progressively more general settings.
	Throughout  the argument we write $p\, \mid\mid\, n$ if $p\mid n$ but $p^2 \nmid n$.
	\subsection{Preparatory counting arguments}
	The following lemma will be used multiple times subsequently.
	\begin{lemma}\label{L:wNPQ}
		For $Q,N\in \N$, $a,b\in \Z$,  and  primes $p, q$ such that   $p,q\equiv 1
		\pmod{4}$  and  $(pq, Q)=1$, let
		%%	\begin{equation}\label{E:wNQp}
			%%		w_{N,Q}(p):=	\frac{1}{N^2}\, \sum_{\substack{m,n\in [N],\\   p\, %%\mid\mid\,  (Qm+1)^2+(Qn)^2} } 1
			%%	\end{equation}
		%%	and
		\begin{equation}\label{E:wNQpq}
			w_{N,Q}(p,q):=\frac{1}{N^2}\, \sum_{\substack{m,n\in [N],\\   p,q\, \mid\mid\,  (Qm+a)^2+(Qn+b)^2} } 1.
		\end{equation}
		Then
		\begin{equation}\label{E:wNPQ1}
			w_{N,Q}(p,p)=\frac{2}{p}\Big(1-\frac{1}{p}\Big)^2 + O\Big(\frac{1}{N}\Big),
		\end{equation}
		and  if $p\neq q$ we have  %%({\bf maybe have to assume $p^2q^2\leq N$ here})
		\begin{equation}\label{E:wNPQ2}
			w_{N,Q}(p,q)=\frac{4}{pq}\Big(1-\frac{1}{p}\Big)^2 \Big(1-\frac{1}{q}\Big)^2 +O\Big(\frac{1}{N}\Big),
		\end{equation}
  where the implicit constants are absolute.
	\end{lemma}
	\begin{remark}
		We deduce the approximate identity
  $$	w_{N,Q}(p,q)=	w_{N,Q}(p,p)\cdot 	w_{N,Q}(q,q)+O\Big(\frac{1}{N}\Big),
  $$
  which is crucial for the proof of the  concentration estimates.
		On the other hand,
		 because of the $O\big(\frac{1}{N}\big)$ errors, these approximate identities  will only be useful to us for sums
		that contain $o(N)$ terms.
	\end{remark}
	\begin{proof}
		Throughout the discussion, we use $\epsilon$  to designate a number in $\{0,1,2,3,4\}$ .
		
		We first establish  \eqref{E:wNPQ1}.
		%%Without loss of generality we can assume that $p\mid N$, if not the remaining part %%of the interval $[N]$ gives as an error bounded by $1/N^{1-a}$.
		Let $p$ satisfy the assumptions. Note first that if  $p\mid Qn+b$ and $p\mid (Qm+a)^2+(Qn+b)^2$, then also $p^2\mid (Qm+a)^2+(Qn+b)^2$, hence  we get no contribution to the sum \eqref{E:wNQpq} in this case. So we can assume that   $p\nmid Qn+b$. Since  $p\equiv 1\pmod 4$, the number   $-1$ is a quadratic residue $\! \! \! \mod{p}$,  and we have exactly two solutions  $m\pmod {p}$ to the congruence
		\begin{equation}\label{E:congruence}
			(Qm+a)^2+(Qn+b)^2\equiv 0\pmod {p}.
		\end{equation}
		Hence, for those $n\in [N]$ we have   $2[N/p]+\epsilon$ solutions in the variable $m\in [N]$ to \eqref{E:congruence}.
		Since there are $N-[N/p]+\epsilon$ integers $n\in [N]$ with $p\nmid Qn+b$ (we used that $(p,Q)=1$ here), we get  a  total of
		$$
		2[N/p]\, (N-[N/p]) +O(N)=2N^2/p-2N^2/p^2+O(N)
		$$ solutions of $m,n\in [N]$ to the congruence \eqref{E:congruence}. Similarly, we get that if
		$p\nmid Qn+b$, then the number of solutions $m,n\in [N]$ to the  congruence  $(Qm+a)^2+(Qn+b)^2\equiv 0\pmod {p^2}$ is
		$$
		2[N/p^2]\, (N-[N/p])+O(N)=2N^2/p^2-2N^2/p^3+O(N).
		$$
		(We used that $-1$ is also a quadratic residue $\! \! \! \mod{p^2}$.)
		These solutions should be subtracted from the previous solutions of \eqref{E:congruence} in order to count the number of solutions of $m,n\in [N]$ for which $p\, \mid\mid \, (Qm+a)^2+(Qn+b)^2$. We deduce that
		\begin{equation}\label{E:pndivn}
			\frac{1}{N^2}\, \sum_{\substack{m,n\in [N],\\   p\, \mid\mid\,  (Qm+a)^2+(Qn+b)^2} } 1=\frac{2}{p}-\frac{4}{p^2}+\frac{2}{p^3}+O\Big(\frac{1}{N}\Big)=\frac{2}{p}\Big(1-\frac{1}{p}\Big)^2 + O\Big(\frac{1}{N}\Big),
		\end{equation}
		which proves  \eqref{E:wNPQ1}.
		
		Next,  we establish \eqref{E:wNPQ2}. 	Let $p,q $ satisfy the assumptions.
		As explained in the previous case, those $n\in [N]$ for which $p\mid Qn+b$ or $q\mid Qn+b$ do not contribute  to the sum \eqref{E:wNQpq} defining $w_{N,Q}(p,q)$, hence we can assume that
		$(pq,Qn+b)=1$.
		Let
		$$
		A_{r,s}:=\frac{1}{N^2}\, \sum_{\substack{m,n\in [N],\\   r,s\mid  (Qm+a)^2+(Qn+b)^2,\, (rs,Qn+b)=1 }} 1
		$$
		and note that
		\begin{equation}\label{E:wNQA}
			w_{N,Q}(p,q)=A_{p,q}-A_{p^2,q}-A_{p,q^2}+A_{p^2,q^2}.
		\end{equation}
		
		We first compute $A_{p,q}$.
		Since  $p\equiv q\equiv 1\pmod 4$, the number  $-1$ is a quadratic residue $\! \! \! \mod{p}$ and $\! \! \! \mod{q}$,  and we get  by the Chinese remainder theorem, that for each $n\in [N]$ with $p,q\nmid Qn+b$ we have $4$ solutions $m\pmod {pq}$ to the congruence
		\begin{equation}\label{E:congruencepq}
			(Qm+a)^2+(Qn+b)^2\equiv 0\pmod {pq}.\footnote{If $pq>N$ these may translate to no solutions in $m\in [N]$, but this is also going to be reflected in our computation below since  in this case $4[N/(pq)]+\epsilon=\epsilon$ could very well be $0$.  }
		\end{equation}
		We deduce that  for  each $n\in [N]$ with $(pq,Qn+b)=1$  we have $4[N/(pq)]+\epsilon$ solutions in the variable $m  \in [N]$ to the congruence \eqref{E:congruencepq}.
		Since  the number of $n\in [N]$ for which  $(pq,Qn+b)=1$ is $N-[N/p]-[N/q]+[N/pq]$,
		we get that the total number of solutions to the congruence \eqref{E:congruencepq} with $m,n\in[N]$ and $(pq,n)=1$ is
		\begin{multline}\label{E:pqndivn}
			4 [N/(pq)] \, (N-[N/p]-[N/q]+[N/(pq)])+ O(N)=\\ N^2 \cdot (4/(pq))\cdot (1-1/p-1/q+1/(pq)) +O(N).
		\end{multline}
		Hence,
		%%$$
		%%\frac{1}{N^2}\, \sum_{\substack{m,n\in [N],\\   p,q\mid  (Qm+1)^2+(Qn)^2,\, (pq,n)=1 }} %%1=
		%%4\, \Big(\frac{1}{pq}-\frac{1}{p^2q}-\frac{1}{pq^2}+\frac{1}{p^2q^2}\Big) %%+O\Big(\frac{1}{N}\Big).
		%%$$
		$$
		A_{p,q}:=
		%%\frac{1}{N^2}\, \sum_{\substack{m,n\in [N],\\   p,q\mid  (Qm+1)^2+(Qn)^2,\, (pq,n)=1 }} 1=
		\frac{4}{pq}\Big(1-\frac{1}{p}\Big) \Big(1-\frac{1}{q}\Big) +O\Big(\frac{1}{N}\Big).
		$$
		Similarly, using that  $-1$ is also a quadratic residue $\! \! \! \mod{p^k}$  and  $\! \! \! \mod{q^k}$ for $k=1,2$,  we find that
		$$
		A_{p^2,q}=
		\frac{4}{p^2q}\Big(1-\frac{1}{p}\Big) \Big(1-\frac{1}{q}\Big) +O\Big(\frac{1}{N}\Big),
		$$
		and
		$$
		A_{p,q^2}=
		\frac{4}{pq^2}\Big(1-\frac{1}{p}\Big) \Big(1-\frac{1}{q}\Big) +O\Big(\frac{1}{N}\Big).
		$$
		Also,
		$$
		A_{p^2,q^2}=
		\frac{4}{p^2q^2}\Big(1-\frac{1}{p}\Big) \Big(1-\frac{1}{q}\Big) +O\Big(\frac{1}{N}\Big).
		$$
		Using the last four identities and  \eqref{E:wNQA}, we deduce that \eqref{E:wNPQ2} holds. This completes the proof.
	\end{proof}
	We will also need to give upper bounds for
	$w_{N,Q}(p,q)$ when $p,q$ are not necessarily primes, and also give upper bounds that do not involve the error terms $O(1/N)$ that cause  us problems in some cases (this is only relevant  when $pq\geq N$). The next lemma is crucial for us and gives an upper bound that is good enough for our purposes.
	\begin{lemma}\label{L:wNQl}
		For $l, Q,N \in \N$ and  $a,b\in \Z$ with  $-Q\leq a,b\leq Q$,   let
		$$
		w_{N,Q}(l):=\frac{1}{N^2} \, \sum_{\substack{m,n\in [N],\\   l\mid  (Qm+a)^2+(Qn+b)^2} } 1.
		$$
		If  $l$ is a sum of two squares, then
		\begin{equation}\label{E:wNQl}
			w_{N,Q}(l)\ll  \frac{Q^2}{l},
		\end{equation}
 % \textcolor{red}{[[Should $\leq$ be replaced by $\ll$, where the absolute constant is absolute? Also in the next displayed formula.]]}
		where the implicit constant is  absolute.
  In particular, if $w_{N,Q}(p,q)$ is as in \eqref{E:wNQpq}, taking $l=p$ and $l=pq$ where $p, q$ are distinct primes of the form
		$1\pmod{4}$,   we get
		\begin{equation}\label{E:wNQppq}
			w_{N,Q}(p,p)\ll  \frac{Q^2}{p}, \quad 	 	w_{N,Q}(p,q)\ll  \frac{Q^2}{pq}.
		\end{equation}
	\end{lemma}
\begin{remark}
	These estimates will allow us to show later  that  the contribution of the $m,n\in[N]$ for which
	 $(Qm+a)^2+(Qn+b)^2$ have large prime divisors (say $\geq \sqrt{N}$)
	 is negligible for our purposes. In contrast, we could not have done the same for the $n\in[N]$ for which $n^2+1$ have  large prime divisors.
\end{remark}
	\begin{proof}
		Recall that an integer is a sum of two squares if and only if in its  factorization as a product of primes, all prime factors congruent to $3 \pmod{4}$ occur with even multiplicity.
		It follows that
		if  $l$ is a sum of two squares and $l\mid (Qm+a)^2+(Qn+b)^2$, then the ratio
		$\big((Qm+a)^2+(Qn+b)^2\big)/l$ is also a sum of two squares. We deduce from this and our assumption $|a|,|b|\leq Q$  that if $l$ is a sum of two squares, then
		$$
		w_{N,Q}(l)\leq \frac{1}{N^2}\, \sum_{k\leq 3Q^2N^2/l}r_2(k)\ll \frac{Q^2}{l},
		$$
		where $r_2(k)$ denotes the number of representations of $k$ as a sum of two squares,  and to get the second estimate we used the well-known fact
		$\sum_{k\le n}r_2(k)\ll n$.  This completes the proof.
	\end{proof}
	\subsection{Concentration estimate for additive functions}
	We start with a  concentration estimate for additive functions that will eventually get  lifted to a concentration estimate for multiplicative functions.
	\begin{definition}
		We say that  $h\colon \mathbb{N}\to\mathbb{C}$ is {\em  additive}, if   it satisfies
		$h(mn)=h(m) +h(n)$ whenever $(m,n)=1$.
	\end{definition}
	\begin{lemma}[Tur\'an-Kubilius inequality for sums of squares]\label{L:TKadditive}	
		Let $K_0,N\in \N$, $a,b\in\Z$ with $-Q\leq a,b\leq Q$,   and  $h\colon \mathbb{N}\to\mathbb{C}$ be an additive  function that is bounded by $1$  on primes and
		such that
		\begin{enumerate}
			\item \label{I:hp1}  $h(p)=0$ for all primes $p\leq K_0$ and $p>N$;
			
			\item \label{I:hp2} $h(p)=0$  for all  primes $p\equiv 3\pmod 4$;
			
			\item \label{I:hp3} $h(p^k)=0$  for all   primes $p$ and  $k\geq 2$.
		\end{enumerate}
		Let also  $Q= \prod_{p\leq K_0}p^{a_p}$ for some $a_p\in \N$.
		%% Let also  $Q\in \N$  be such that  $ \prod_{p\leq P_0}p^{P_0}\mid Q$.
		Then  for all  large enough $N$, depending only on $K_0$, we have
		\begin{equation}\label{E:varianceh}
			\E_{m,n\in [N]}\, \big|h\big((Qm+a)^2+(Qn+b)^2\big)-   H_N(h,K_0) \big|^2\ll   \D^2(h; K_0,\sqrt{N})+Q^2\cdot \D^2(h; \sqrt{N},N)+K_0^{-1},
		\end{equation}
		where the implicit constant is absolute,
		\begin{equation}\label{E:HNhP0}
			H_N(h,K_0):=2\, \sum_{K_0< p\leq N}\, \frac{h(p)}{p}
		\end{equation}
		and
		$$
		\D^2(h; K_0,N):=\sum_{K_0< p\leq N} \frac{|h(p)|^2}{p}.
		$$
	\end{lemma}
	\begin{proof}
		%%	We start with a reduction. We claim that it suffices to prove the result when in %%addition $h$ is real valued and non-negative. Indeed, if $h$ is real valued we then %%apply the result to  the additive functions  $h^\pm$ defined on prime powers  by
		%%	$$
		%%	h^\pm(p^k)=\max\{\pm h(p^k), 0\}.
		%%	$$
		%%	These are also additive functions that satisfy properties \eqref{I:hp1}-\eqref{I:hp3} %%and use that $h=h^+-h_-$, $H_N(h,P_0)=H_N(h^+,P_0)-H_N(h^-,P_0)$, and %%$\D^2(h; P_0,N)\leq 2 (\D^2(h^+;P_0,N)+\D^2(h^-;P_0,N))$. Similarly, when $h$ %%takes complex values, we apply the result for its real and complex valued part. %%Henceforth, we assume that
		%%\begin{equation}\label{E:hpositive}
		%%	h \text{ is real valued and non-negative.} \footnote{We will only use this in our %%estimation of the variance of $h_1$, defined below. I picked up this trick from the %%proof of the linear Tur\'an-Kubilius in the book of Tenenbaum, it saves at least half %%a page of argument.}
		%%\end{equation}
		We consider the additive functions $h_1,h_2$, which are the restrictions  of $h$ to the primes $K_0<p\le \sqrt{N}$ and $\sqrt{N}<p\le N$.\footnote{If we worked with $h$ only,  we would run into trouble establishing \eqref{E:variance1} below, since a  non-acceptable term of the form $O(\sum_{p,q\leq N} N^{-1})$ would appear in our estimates. For $h_1$ this term becomes $O(\sum_{p,q\leq \sqrt{N}} N^{-1})=O((\log{N})^{-2})$, which is acceptable. We could have also worked with the restriction to  the intevral $[K_0,N^a]$ for any $a\leq 1/2$.  In the case of linear concentration estimates this splitting is not needed since the error that appears in this case  is  $O(\sum_{pq \leq N} N^{-1})=O(\log\log{N/\log{N}})$. } More precisely,
		$$
		h_1(p^k):= \begin{cases}h(p), \quad  &\text{if }\  k=1 \text{ and }   K_0<p\leq \sqrt{N}\\
			0, \quad &\text{otherwise}  \end{cases}
		$$
		and
		$$
		h_2(p^k):=\begin{cases}h(p), \quad  &\text{if }\  k=1 \text{ and } \sqrt{N}<p\leq  N\\
			0, \quad &\text{otherwise}  \end{cases}.
		$$
		We also  define
		\begin{equation}\label{E:HiN}
			H_{i,N}(h_i,K_0):=2\sum_{K_0< p\leq N}\, \frac{h_i(p)}{p}, \quad i=1,2,
		\end{equation}
		and the  technical variant
		%%	$$
		%%	H'_{1,N}(h_i,P_0):=\sum_{P_0\leq p\leq N }\, \Big(2-\frac{1}{p}\Big)\,
		%%	\frac{h_1(p)}{p}.
		%%	$$
		\begin{equation}\label{E:H'iN}
			H'_{1,N}(h_i,Q,K_0):=\sum_{K_0< p\leq N }\, w_{N,Q}(p)\cdot h_1(p),
		\end{equation}
		where
		\begin{equation}\label{E:wNQp}
			w_{N,Q}(p):=	\frac{1}{N^2}\, \sum_{\substack{m,n\in [N],\\   p\, \mid\mid\,  (Qm+a)^2+(Qn+b)^2} } 1.
		\end{equation}
		(Note that $w_{N,Q}(p)=w_{N,Q}(p,p)$ where $w_{N,Q}(p,q)$ is as in \eqref{E:wNQpq}.)
		The reason for introducing this variant is because it gives the mean value of $h_1$ along sums of squares. Indeed,
		using properties \eqref{I:hp1}-\eqref{I:hp3}, we have
		\begin{multline}\label{E:meanhi}
			\E_{m,n\in [N]}\, h_1((Qm+a)^2+(Qn+b)^2)=\E_{m,n\in [N]}\, \sum_{p\, \mid\mid\,  (Qm+a)^2+(Qn+b)^2}h_1(p)\\=
			\frac{1}{N^2} \sum_{K_0<  p \leq N} h_1(p) \,  \sum_{\substack{m,n\in [N],\\   p\, \mid\mid\,  (Qm+a)^2+(Qn+b)^2} } 1
			=
			H'_{1,N}(h_1,Q,K_0).
		\end{multline}

		Using \eqref{E:wNPQ1} of  Lemma~\ref{L:wNPQ} and that $h_1(p)=0$ for $p> \sqrt{N}$ and $h_1(p)$ is bounded by $1$,  we get
		\begin{equation}\label{E:H1H1'}
			|H_{1,N}(h_1,K_0)-H'_{1,N}(h_1,Q,K_0)|\ll
			\sum_{K_0< p\leq \sqrt{N}} \frac{1}{p^2}+\frac{1}{\sqrt{N}}
			\leq 	\frac{1}{K_0} +\frac{1}{\sqrt{N}}.
		\end{equation}
		Hence, in order to prove \eqref{E:varianceh},
		it suffices to estimate
		\begin{equation}\label{E:h1H1}
			\E_{m,n\in [N]}\big|h_1((Qm+a)^2+(Qn+b)^2)-	H'_{1,N}(h_1,Q,K_0)\big|^2
		\end{equation}
		and
		\begin{equation}\label{E:h2H2}
			\E_{m,n\in [N]}\big|h_2((Qm+a)^2+(Qn+b)^2)\big|^2+	|H_{2,N}(h_1,K_0)|^2.
		\end{equation}

		We first deal with the expression \eqref{E:h1H1}.
		Using  \eqref{E:meanhi} and expanding the square below we get
		\begin{multline}\label{E:variance}
			\E_{m,n\in [N]} \, \big|h_1((Qm+a)^2+(Qn+b)^2)- H'_{1,N}(h_1,Q,K_0)\big|^2=\\
			\E_{m,n\in [N]}\, \big|h_1((Qm+a)^2+(Qn+b)^2)\big|^2 - |H'_{1,N}(h_1,Q,K_0)|^2.
		\end{multline}
		To estimate this expression, first note that since $h_1$ is additive and $h_1(p^k)=0$ for $k\geq 2$, we have
		\begin{equation}\label{E:h1Qsquare}
			\E_{m,n\in [N]}\, \big|h_1((Qm+a)^2+(Qn+b)^2)\big|^2=
			\E_{m,n\in [N]}\,  \Big|\sum_{p\,  \mid \mid \,(Qm+a)^2+(Qn+b)^2} h_1(p)\Big|^2.
		\end{equation}
		Expanding the square, using the  fact that $h_1(p)=0$ unless $K_0<p\leq \sqrt{N}$, and the definition of $w_{N,Q}(p,q)$ given in \eqref{E:wNQpq}, we get that the right hand side is equal to
		$$
		\sum_{K_0<p\leq \sqrt{N}} |h_1(p)|^2\cdot w_{N,Q}(p,p)+
		\sum_{K_0<p,q\leq \sqrt{N}, \,  p\neq q} h_1(p)\cdot  \overline{h_1(q)}\cdot w_{N,Q}(p,q).
		$$
		Using  equation \eqref{E:wNPQ1} of Lemma~\ref{L:wNPQ} we get that the first term is
		at most
		$$
		2\cdot \sum_{K_0<p\leq \sqrt{N}}\frac{|h_1(p)|^2}{p}+ O(N^{-1/2}).
		$$
		Using equations \eqref{E:wNPQ1} and \eqref{E:wNPQ2} of Lemma~\ref{L:wNPQ} we get that the second term is equal to (we crucially use the bound $p,q\leq \sqrt{N}$ here and the prime number theorem)
		\begin{multline*}
			\sum_{K_0<p,q\leq \sqrt{N}, \,  p\neq q} h_1(p)\cdot \overline{h_1(q)}\cdot w_{N,Q}(p,p)\cdot w_{N,Q}(q,q) +O((\log{N})^{-2})\leq \\(H'_{1,N}(h_1,Q,K_0))^2+O((\log{N})^{-2}),
			%%O\Big(\frac{1}{(\log{N})^2}\Big)
		\end{multline*}
		where to get the last estimate we  added to the sum the contribution of the diagonal terms $p=q$  (which is non-negative), used \eqref{E:H'iN}, and the fact that $h_1(p)=0$ for $p>\sqrt{N}$.
		Combining \eqref{E:variance} with the previous estimates, we are led to the bound
		\begin{multline}\label{E:variance1}
			\E_{m,n\in [N]} \, \big|h_1((Qm+a)^2+(Qn+b)^2)- H'_{1,N}(h_1,Q,K_0)\big|^2\ll \\
			\D^2(h_1; K_0, \sqrt{N})+O((\log{N})^{-2}).
		\end{multline}
		
		Next  we estimate  the expression \eqref{E:h2H2}. Since $h_2$ is additive and satisfies properties \eqref{I:hp1}-\eqref{I:hp3}, we get using \eqref{E:h1Qsquare} (with $h_2$ in place of $h_1$)   and expanding the square
		$$
		\E_{m,n\in [N]} \, \big|h_2((Qm+a)^2+(Qn+b)^2)\big|^2= 	 \sum_{\sqrt{N}<   p, q \leq N} h_2(p)\, \overline{h_2(q)} \, w_{N,Q}(p,q).
		$$
		Since $h_2(p)\neq 0$ only when $p\equiv 1 \! \!  \pmod{4}$,
		using  \eqref{E:wNQppq} of  Lemma~\ref{L:wNQl}, we get that the right hand side is bounded  by
		\begin{multline*}
			\ll 	Q^2\cdot \Big(\sum_{\sqrt{N}< p,q\leq N} \frac{|h_2(p)| \, |h_2(q)|}{pq}+ \sum_{\sqrt{N}< p\leq N} \frac{|h_2(p)|^2}{p}\Big)= \\
			Q^2\cdot \Big(\Big(\sum_{\sqrt{N}< p\leq N} \frac{|h_2(p)|}{p}\Big)^2+\D^2(h_2; \sqrt{N},N) \Big)\leq \\ Q^2 \cdot \Big(\sum_{\sqrt{N}< p\leq N} \frac{|h_2(p)|^2}{p}\cdot
			\sum_{\sqrt{N}< p\leq N} \frac{1}{p}+\D^2(h_2; \sqrt{N},N) \Big)\ll Q^2\cdot \D^2(h_2; \sqrt{N},N),
		\end{multline*}
		where we crucially used the estimate
		$$
		\sum_{\sqrt{N}< p\leq N} \frac{1}{p}\ll 1.
		$$
		Similarly we find  $$
		(H_{2,N}(h_2,K_0))^2=4\, \Big(\sum_{\sqrt{N}< p\leq N} \frac{|h_2(p)|}{p}\Big)^2\ll  \D^2(h_2; \sqrt{N},N).
		$$
		Combining the previous estimates we get the following bound for the expression in  \eqref{E:h2H2}
		\begin{equation}\label{E:h2H2bound}
			\E_{m,n\in [N]}\big(h_2((Qm+a)^2+(Qn+b)^2)\big)^2+	(H_{2,N}(h_1,K_0))^2\ll  Q^2\cdot \D^2(h_2; \sqrt{N},N).
		\end{equation}
		
		Combining the bounds  \eqref{E:H1H1'}, \eqref{E:variance1}, \eqref{E:h2H2bound}, we get the asserted bound \eqref{E:varianceh}, completing the proof.
	\end{proof}

	\subsection{Concentration estimates for multiplicative functions}
	Next we use Lemma~\ref{L:TKadditive} to get a variant that deals with multiplicative functions.
	\begin{lemma}\label{L:TKmultiplicative}	
		Let $K_0,N\in \N$,  $a,b\in \Z$ with  $-Q\leq a,b\leq Q$, and  $f\colon \mathbb{N}\to\mathbb{U}$ be a multiplicative function that satisfies
		\begin{enumerate}
			\item \label{I:fp1}  $f(p)=1$ for all primes $p\leq K_0$ and $p> N$;
			
			\item \label{I:fp2} $f(p)=1$  for all  primes $p\equiv 3\pmod 4$;
			
			\item \label{I:fp3} $f(p^k)=1$  for all   primes $p$ and  $k\geq 2$.
		\end{enumerate}
		Let also  $Q= \prod_{p\leq K_0}p^{a_p}$ with  $a_p\in \N$.
		%% Let also  $Q\in \N$  be such that  $ \prod_{p\leq P_0}p^{P_0}\mid Q$.
		If $N$ is  large enough, depending only on $K_0$, then
		\begin{multline}\label{E:mainestimate'}
			\E_{m,n\in [N]}\, \big|f\big((Qm+a)^2+(Qn+b)^2\big)-   \exp\big(G_N(f,K_0)\big)\big|\ll \\
			(\D+\D^2)(f,1; K_0,\sqrt{N})+ Q\cdot \D(f,1; \sqrt{N},N)+K_0^{-\frac{1}{2}},
		\end{multline}
		where the implicit constant is absolute  and
		\begin{equation}\label{E:GNfP01}
			G_N(f,K_0):=2\sum_{ K_0< p\leq  N}\, \frac{1}{p}\,(f(p) -1).
		\end{equation}
	\end{lemma}
	\begin{proof}
		Let $h:\mathbb{N}\to \mathbb{C}$ be the additive function given on prime powers by
		$$
		h(p^k):=f(p^k)-1.
		$$
		We note that  due to our assumptions on $f,$ properties \eqref{I:hp1}-\eqref{I:hp3} of Lemma~\ref{L:TKadditive} are satisfied for $h/2$, which is bounded by $1$ on primes.
		
		Using that $z=e^{z-1}+O(|z-1|^2)$ for $|z|\leq 1$ and property \eqref{I:hp3}, we have
		\begin{align*}
			f(m^2+n^2)=\prod_{p^k\, \mid \mid \, m^2+n^2}f(p^k)=  \prod_{p\, \mid \mid \, m^2+n^2}\big(\exp(h(p))+O(|h(p)|^2)\big).
		\end{align*}
		Applying the estimate $|\prod_{i\leq k}z_i-\prod_{i\leq k}w_i|\leq \sum_{i\leq k}|z_i-w_i|$, we deduce that for all $m,n\in\N$ we have
		\begin{align*}
			f(m^2+n^2)=\exp(h(m^2+n^2))+O\Big(\sum_{p \, \mid \mid\,  m^2+n^2}|h(p)|^2\Big).
		\end{align*}
		Using this and since $G_N(f,K_0)=H_N(h, K_0)$, where $H_N(h,K_0)$ is given by \eqref{E:HNhP0}, we get
		\begin{multline}\label{equ13}
			\E_{m,n\in [N]}\, \big|f\big((Qm+a)^2+(Qn+b)^2\big)-   \exp\big(G_N(f,K_0)\big)\big|
			\ll \\  \E_{m,n\in [N]}|\exp\big(h((Qm+a)^2+(Qn+b)^2)\big)-\exp(H_N(h, K_0))|+\\
			\E_{m,n\in [N]}\sum_{p\, \mid \mid\,  (Qm+a)^2+(Qn+b)^2}|h(p)|^2.
		\end{multline}
		Next we  use the inequality $|e^{z_1}-e^{z_2}|\le |z_1-z_2|$, which is  valid for $\Re z_1,\Re z_2\le 0$, to bound the last expression by
		\begin{align}\label{equ14}
			\E_{m,n\in [N]}|h((Qm+a)^2+(Qn+b)^2)-H_N(h,K_0)|+\E_{m,n\in [N]}\sum_{p \, \mid \mid \, (Qm+a)^2+(Qn+b)^2}|h(p)|^2.
		\end{align}
		To bound the first term we use Lemma~\ref{L:TKadditive}. It gives that for all large enough $N$, depending on $K_0$ only, we have
		\begin{multline}\label{E:estimate1}
			\E_{m,n\in [N]}\, \big|h\big((Qm+a)^2+(Qn+b)^2\big)-   H_N(h,K_0) \big|\ll\\
			\D(h; K_0,\sqrt{N})+Q\cdot \D(h; \sqrt{N},N)+K_0^{-\frac{1}{2}}
			\ll \D(f,1; K_0,\sqrt{N})+ Q\cdot \D(f,1; \sqrt{N},N)+K_0^{-\frac{1}{2}},
		\end{multline}
		where  to get the last bound we used that  $|h(p)|^2\leq 2-2\,\Re(f(p))$, which holds since $|f(p)|\leq 1$.
		To bound the second term in \eqref{equ14}, we note  that
		using properties \eqref{I:hp1}-\eqref{I:hp3} of Lemma~\ref{L:TKadditive}, we have
		\begin{multline}\label{E:estimate2}
			\E_{m,n\in [N]}\sum_{p\, \mid \mid\,  (Qm+a)^2+(Qn+b)^2}|h(p)|^2=
			\sum_{ K_0< p \leq N}\ |h(p)|^2\, w_{N,Q}(p)\ll \\
			\\	 \sum_{ K_0< p \leq N}\frac{|h(p)|^2}{p}+O((\log{N})^{-1}) \ll 	\mathbb{D}^2(f, 1;K_0,N)+O((\log{N})^{-1}) ,
		\end{multline}
		where $w_{N,Q}(p)$ is as in \eqref{E:wNQp} and we used equation  \eqref{E:wNPQ1} of Lemma~\ref{L:wNPQ}  and the prime number theorem to get the first bound.
		Combining \eqref{equ13}-\eqref{E:estimate2} we get the asserted bound. 	
	\end{proof}
	We use the previous result to deduce the following improved version.
	\begin{lemma}\label{L:TKmultiplicative2}	
		Let $K_0,N\in \N$ and  $f\colon \mathbb{N}\to\mathbb{U}$ be a multiplicative function such that   $f(p)= 1$ for all primes $p> N$ with $p\equiv 1 \pmod{4}$. 	Let also  $Q= \prod_{p\leq K_0}p^{a_p}$ with  $a_p\in \N$.
		If  $N$ is   large enough, depending only on  $K_0$,
		then for all  $a,b\in \Z$ with  $-Q\leq a,b\leq Q$ and $(a^2+b^2,Q)=1$ we have
		\begin{multline}\label{E:mainestimate}
			\E_{m,n\in [N]}\, \big|f\big((Qm+a)^2+(Qn+b)^2\big)- \exp\big(G_N(f,K_0)\big)\big|\ll \\	(\D_1+\D_1^2)(f,1; K_0,\sqrt{N})+ Q\cdot \D_1(f,1; \sqrt{N},N)+K_0^{-\frac{1}{2}},
		\end{multline}
		where the implicit constant is absolute and
		\begin{equation}\label{E:GNfP01'}
			G_N(f,K_0):=2\sum_{ \substack{K_0< p\leq  N; \\ p\equiv 1 \! \! \! \pmod{4}}}\, \frac{1}{p}\,(f(p) -1),\end{equation}
		$$
		\D_1(f,1; x,y)^2:= \sum_{\substack{x< p\leq y;\\ p\equiv   1 \! \! \! \pmod{4}}} \frac{1}{p}\, (1-\Re(f(p)))
		$$
  for $x<y$.
	\end{lemma}
	\begin{proof}		
		We first  define the multiplicative function $\tilde{f}\colon \N\to \U$ on prime powers as follows
		$$
		\tilde{f}(p^k):= \begin{cases}f(p^k), \quad  &\text{if }\   p> K_0\\
			1, \quad &\text{otherwise}  \end{cases}.
		$$
		Since $p\leq K_0$ implies  $p\mid Q$ and $(a^2+b^2,Q)=1$,
		%%\footnote{This is one of the two places in the proof of %%Proposition~\ref{P:concentration2quanti} where we use %%this assumption.}
		we get that   $p\nmid (Qm+a)^2+(Qn+b)^2$ for every $ p\leq K_0$, hence
		$$
		f((Qm+a)^2+(Qn+b)^2)=\tilde{f}((Qm+a)^2+(Qn+b)^2) \quad \text{for every } m,n\in\N.
		$$
		Note also that $G_N(\tilde{f},K_0)=G_N(f,K_0)$ and $\D_1(\tilde{f},1; K_0,N)=\D_1(f,1;K_0,N)$.
		It follows that in order to establish \eqref{E:mainestimate}, it is  enough to show that for all large enough $N$, depending only on $K_0$, we have
		\begin{multline}\label{E:mainestimate2}
			\E_{m,n\in [N]}|\tilde{f}((Qm+a)^2+(Qn+b)^2)- \exp(G_N(\tilde{f},Q))|\ll  \\	(\D_1+\D_1^2)(\tilde{f},1; K_0,\sqrt{N})+ Q\cdot \D_1(\tilde{f},1; \sqrt{N},N)+	K_0^{-\frac{1}{2}}.
		\end{multline}
		In order to establish \eqref{E:mainestimate2} we
		make a series of further reductions that will eventually allow us to apply Lemma~\ref{L:TKmultiplicative}. For every $p\equiv3\pmod 4,$ we have that $p\mid m^2+n^2$ implies that $p\mid m$ and $p\mid n$. Consequently, the contribution to the average of those $m,n\in [N]$ for which $(Qm+a)^2+(Qn+b)^2$ is divisible by some prime  $p\equiv 3\pmod 4$  is
		(note that $(Qm+a)^2+(Qn+b)^2$ is only divisible by primes $p>K_0$)
		$$
		\ll \frac{1}{N^2}\sum_{K_0< p\leq N} \left[\frac{N}{p}\right]^2\ll \frac{1}{K_0},
		$$
		which is acceptable.
		
		Next we show  that the contribution to the average  in  \eqref{E:mainestimate2} of those $m,n\in [N]$ for which $ (Qm+a)^2+(Qn+b)^2$ is divisible by $p^2$ for some  prime $p\equiv 1\pmod 4$  with $p>K_0$ (hence $p\nmid Q$) is also acceptable. Indeed, for
		fixed $n\in [N]$ such that  $p\nmid Qn+b$  there exist at most  $2[N/p^2]+2$ values of $m\in [N]$ such that $p^2\mid (Qm+a)^2+(Qn+b)^2.$ On the other hand,  if $p\mid Qn+b$ and $p\mid  (Qm+a)^2+(Qn+b)^2$,  then also $p\mid Qm+a$.
		Hence,  the contribution to the average  in \eqref{E:mainestimate2} of those  $m,n\in [N]$
		for which $ (Qm+a)^2+(Qn+b)^2$ is divisible by $p^2$ for some  prime $p\equiv 1\pmod 4$
		is bounded by (note again that $(Qm+a)^2+(Qn+b)^2$ is only divisible by primes $p> K_0$)
		$$
		\ll \frac{1}{N^2} \Big(\sum_{K_0< p\leq N}\Big(\left[\frac{N}{p^2}\right]+1\Big)N+\sum_{K_0< p\leq N}\left[\frac{N}{p}\right]^2\Big)\ll \frac{1}{K_0}+\frac{1}{\log{N}},
		$$
		where we used the prime number theorem to bound $\frac{1}{N}\sum_{K_0<p\leq N} 1$.
		%%Lastly, using  Lemma~\ref{L:wNQl} and the notation introduced there, we get %%that the contribution to the average  in  \eqref{E:mainestimate2} of the terms %%$m,n\in [N]$
		%%for which $ (Qm+1)^2+(Qn)^2$ is divisible by $p^2$ for some  $p\equiv %%1\pmod4$
		%%with  $\sqrt{N}<p\leq N,$ is bounded by
		%%	$$
		%%\frac{1}{N^2} \sum_{\sqrt{N}< p\leq N}\sum_{\substack{ m,n\in [N],   p^2\mid %%(Qm+1)^2+(Qn)^2}}1\leq \sum_{\sqrt{N}< p\leq N}\ \, w_{N,Q}(p^2)
		%%\ll \\ \sum_{\sqrt{N}< p\leq N} \frac{Q^2}{p^2}\ll \frac{Q^2}{\sqrt{N}},
		%%$$	
		%%		which is also acceptable.
		
		Combining the above reductions, we deduce that in order to establish the estimate  \eqref{E:mainestimate2} we may further assume that
		\begin{equation}\label{E:freduced}
			\tilde{f}(p^k)=1 \ \text{ for all } \ p\in \P,\,  k\ge 2, \ \text{  and }\  \tilde{f}(p^k)=1 \ \text{ for all } \ p\equiv 3\pmod 4, \, k\in \N.
		\end{equation}
		We are now in a situation where Lemma~\ref{L:TKmultiplicative} is applicable and gives that
		for all large enough $N$, depending only on $K_0$, if $\D_1(f,1;K_0,N)\leq 1$, we have
		(note that \eqref{E:freduced} implies that $\mathbb{D}_1(\tilde{f}, 1;K_0,N)=\mathbb{D}(\tilde{f}, 1;K_0,N)$)
		\begin{multline*}
			\E_{m,n\in [N]}|\tilde{f}((Qm+a)^2+(Qn+b)^2)- \exp(G_N(\tilde{f},K_0))|\ll \\
			(\D_1+\D_1^2)(\tilde{f},1; K_0,\sqrt{N})+ Q\cdot \D_1(\tilde{f},1; \sqrt{N},N)+K_0^{-\frac{1}{2}}.
		\end{multline*}
		Combining this bound with the bounds we got in order to arrive to this reduction, we get that \eqref{E:mainestimate2} is satisfied. This completes the proof.
	\end{proof}
	%%We are now ready to finish the proof of the main result of this section.

	\subsection{Proof of Proposition~\ref{P:concentration2quanti}}	
	We start with some reductions. Suppose that the statement holds when $\chi=1$ and $t=0$, we will show that it holds for arbitrary $\chi$ and $t$.
 	Let
		$\tilde{f}:=f \cdot \overline{\chi}\cdot n^{-it}$,  and apply the conclusion for $\chi=1, t=0$.
		We get the following bound for $\tilde{f}$:
	\begin{multline}\label{E:tildef}
		\E_{m,n\in [N]}\, \big|\tilde{f}\big((Qm+a)^2+(Qn+b)^2\big)-  \exp\big(G_N(\tilde{f},K_0)\big)\big|\ll \\
		(\D_1+\D_1^2)(\tilde{f},1; K_0,\sqrt{N})+Q^2\cdot\D_1(\tilde{f},1;N,3Q^2N^2)+	 Q\cdot \D_1(\tilde{f},1; \sqrt{N},N)+K_0^{-1/2}.
	\end{multline}
	Note that since $\chi$ is periodic with period $q$ and $q\mid Q$, we have $\chi((Qm+a)^2+(Qn+b)^2)=\chi(a^2+b^2)$ for every $m,n\in \N$. Furthermore,
	since  by assumption $(a^2+b^2, Q)=1$ and $q\mid Q$, we have $(a^2+b^2, q)=1$,  hence $|\chi(a^2+b^2)|=1$.
	%%\footnote{This is the second and last place where our %%assumption $(a^2+b^2,Q)=1$ is used.}
	Lastly, note that
	$\D_1(\tilde{f}, 1;x,y)=\D_1(f, \chi\cdot n^{it};x,y)$
    and
	$$
	G_N(\tilde{f},K_0)=2\sum_{\substack{ K_0< p\leq N,\\ p\equiv   1 \! \! \! \pmod{4}}}\, \frac{1}{p}\,(\tilde{f}(p) -1)=
	2\sum_{\substack{ K_0< p\leq N,\\ p\equiv   1 \! \! \! \pmod{4}}}\, \frac{1}{p}\,(f(p)\cdot \overline{\chi(p)}\cdot p^{-it} -1)=G_N(f,K_0).
	$$
	After inserting this information in \eqref{E:tildef}, we get that \eqref{E:generalf} is satisfied.
	
	So it suffices to show that if $Q= \prod_{p\leq K_0}p^{a_p}$ for some  $a_p\in \N$, then  if $N$ is large enough, depending only on  $Q$, and $\D_1(f, 1;K_0,N)\leq 1$, we have
	\begin{multline}\label{E:mainreduced}
		\E_{m,n\in [N]}\, \big|f\big((Qm+a)^2+(Qn+b)^2\big)- \exp\big(G_N(f,K_0)\big)\big|\ll\\
		(\D_1+\D_1^2)(f,1; K_0,\sqrt{N})+ Q^2\cdot \D_1(f,1;N,3Q^2N^2)+	 Q\cdot \D_1(f,1; \sqrt{N},N)+K_0^{-1/2}.
	\end{multline}

	For every $N\in \N$ we  decompose $f$ as  $f=f_{N,1} \cdot f_{N,2}$, where the multiplicative functions $f_{N,1},f_{N,2}\colon \N\to \U$  are defined on prime powers as follows
	\begin{align*}
		f_{N,1}(p^k):=& \begin{cases}f(p), \quad  &\text{if }\ k=1 \text{ and }   p> N,\,  p\equiv 1\! \! \! \pmod{4} \\
			1, \quad &\text{otherwise}  \end{cases}, \\
		f_{N,2}(p^k):=& \begin{cases}1, \quad  &\text{if }\ k=1 \text{ and }   p> N, \, p\equiv 1\! \! \! \pmod{4} \\
			f(p^k), \quad &\text{otherwise}  \end{cases}.
	\end{align*}

	We first study the function $f_{N,1}$. Following the notation of Lemma~\ref{L:wNQl} for $l,Q,N\in \N$ we let
	$$
	w_{N,Q}(l):=\frac{1}{N^2} \, \sum_{\substack{m,n\in [N],\\   l\mid  (Qm+a)^2+(Qn+b)^2} } 1.
	$$
	Lemma~\ref{L:wNQl}  implies that if $l$ is a sum of two squares, then
	\begin{equation}\label{E:wNQ}
		w_{N,Q}(l)\ll \frac{Q^2}{l}.
	\end{equation}
	Since for $N\gg Q$ we have $f_{N,1}((Qm+a)^2+(Qn+b)^2)-1\neq 0$ only if $(Qm+a)^2+(Qn+b)^2$ is divisible  by one or two primes $p>N$,\footnote{For $m,n\in [N]$ and $-Q\leq a,b\leq Q$  we have $(Qm+a)^2+(Qn+b)^2\ll Q^2N^2$, so if $(Qm+a)^2+(Qn+b)^2$ was divisible by three or more primes greater than $N$, we would have $N^3\ll Q^2 N^2$, which fails if $Q\ll N$.}
	we get
	\begin{multline}\label{E:f11}
		\E_{m,n\in[N]}
		|f_{N,1}\big((Qm+a)^2+(Qn+b)^2\big)-1|
		\leq \\
		\sum_{\substack{N<p\le
				3Q^2N^2,\\ p\equiv 1\! \! \!\pmod{4}}} |f(p)-1|\, w_{N,Q}(p)
		+ \sum_{\substack{N<p,q\le
				3Q^2N^2,\, p\neq q\\ p,q\equiv 1\! \! \!\pmod{4}}} |f(pq)-1|\,  w_{N,Q}(pq),
	\end{multline}
	where we used that $f_{N,1}(p)=f(p)$ for all $p>N$ and in the second sum we have ignored the contribution of the diagonal terms $p=q$ since, by construction, $f_{N,1}(p^2)=1$ for all primes $p$.
	Using  \eqref{E:wNQ} for $l:=p$, which is a sum of two squares since $p\equiv 1\pmod{4}$,  we estimate the first term as follows\footnote{Bounding $w_{N,Q}(p)$ using \eqref{E:wNPQ1}
		would lead to non-acceptable errors here, because the range of summation is much larger than $N$.}
	\begin{multline*}
		\sum_{\substack{N<p\le
				3Q^2N^2,\\ p\equiv 1\! \! \!\pmod{4}}} |f(p)-1|\, w_{N,Q}(p) \ll	Q^2\sum_{\substack{N<p\le
				3Q^2N^2,\\ p\equiv 1\! \! \!\pmod{4}}} \frac{|f(p)-1|}{p}\leq \\
		Q^2 \cdot \Big(\sum_{\substack{N<p\le
				3Q^2N^2,\\ p\equiv 1\! \! \!\pmod{4}}}
		\frac{|f(p)-1|^2}{p}\Big)^{\frac{1}{2}}\cdot
		\Big(\sum_{\substack{N<p\le
				3Q^2N^2,\\ p\equiv 1\! \! \!\pmod{4}}} \frac{1}{p} \Big)^{\frac{1}{2}} \ll
		Q^2\cdot \D_1(f,1;N,3Q^2N^2),
	\end{multline*}
	where we used  that $\sum_{N\le p\le 3Q^2N^2}\frac{1}{p}\ll 1$ for $N\geq Q$.
	Similarly, using  \eqref{E:wNQ} for $l:=pq$, which is a sum of two squares since $pq\equiv 1\pmod{4}$,  we estimate the second  term in \eqref{E:f11} as follows (note  that since $p\neq q$,  we have $f(pq)=f(p)f(q)$)
	\begin{multline*}
		\sum_{\substack{N<p,q\le
				3Q^2N^2,\, p\neq q, \\ p,q\equiv 1\! \! \!\pmod{4}}} |f(pq)-1|\, w_{N,Q}(pq) \ll	Q^2\sum_{\substack{N<p,q\le
				3Q^2N^2,\\ p,q\equiv 1\! \! \!\pmod{4}}}\frac{|f(p)-1|+|f(q)-1|}{pq}\leq \\
		2\, Q^2 \cdot \Big(\sum_{\substack{N<p,q\le
				3Q^2N^2,\\ p,q\equiv 1\! \! \!\pmod{4}}}
		\frac{|f(p)-1|^2}{pq}\Big)^{\frac{1}{2}}\cdot
		\Big(\sum_{\substack{N<p,q\le
				3Q^2N^2,\\ p,q\equiv 1\! \! \!\pmod{4}}} \frac{1}{pq} \Big)^{\frac{1}{2}} \ll
		Q^2\cdot \mathbb{D}_1(f,1;N,3Q^2N^2),
	\end{multline*}
	where we used  that $\sum_{N\le p\le 3Q^2N^2}\frac{1}{p}\ll 1$ for $N\geq Q$. 	
	Combining the above estimates and \eqref{E:f11}, we deduce that for $N\gg Q$ we have
	\begin{equation}\label{E:f1estimate}
		\mathbb{E}_{m,n\in [N]}\big|f_{N,1}\big((Qm+a)^2+(Qn+b)^2\big)-1|\ll  Q^2\cdot \mathbb{D}_1(f,1;N,3Q^2N^2).
	\end{equation}
	
	Next, we move to the function $f_2$. Since $f_2(p)=1 $ for all primes $p\geq N$,
	Lemma~\ref{L:TKmultiplicative2}  is applicable.  We get that if $N$ is large enough, depending on $K_0$, we have
	\begin{multline}\label{E:f2estimate}
		\mathbb{E}_{m,n\in [N]}|f_{N,2}\big((Qm+a)^2+(Qn+b)^2\big)- \exp(G_N(f,K_0))|
		\ll \\
		(\D_1+\D_1^2)(f,1; K_0,\sqrt{N})+ Q\cdot \D_1(f,1; \sqrt{N},N)+K_0^{-\frac{1}{2}},
	\end{multline}
	where we used that $f_{N,2}(p)=f(p)$ for all primes $p\equiv 1\pmod{4}$ with $p\leq N$, hence
	$G_N(f_{N,2},K_0)=G_N(f,K_0)$ and $\mathbb{D}_1(f_{N,2},1;K_0,N)=\mathbb{D}_1(f,1;K_0,N)$.

	Finally, we use the triangle inequality and combine \eqref{E:f1estimate} and  \eqref{E:f2estimate},
	to obtain that the left hand side  in \eqref{E:mainreduced} is bounded by
	\begin{multline*}	
		\mathbb{E}_{m,n\in [N]}\big(|f_{N,1}\big((Qm+a)^2+(Qn+b)^2\big)- 1|		
		+|f_{N,2}\big((Qm+a)^2+(Qn+b)^2\big)- \exp(G_N(f,K_0))|\big)\\
		\ll 	(\D_1+\D_1^2)(f,1; K_0,\sqrt{N})+ Q^2\cdot \mathbb{D}_1(f,1;N,3Q^2N^2)+ Q\cdot \D_1(f,1; \sqrt{N},N)+K_0^{-1/2}.
	\end{multline*}	
	Thus \eqref{E:mainreduced} holds, completing the proof.
	%%	\end{proof}

\section{Type II Pythagorean pairs and more}\label{S:PythPairs2}
\subsection{Proof of \cref{T:mainpairs2}}
As explained in Section~\ref{SS:Plan1ii},  in order to  complete the proof of  Theorem~\ref{T:mainpairs2} (and hence of  part~\eqref{I:MainPythPairs2} of Theorem~\ref{T:MainPythPairs}) it remains  to prove Proposition~\ref{P:aperiodic2},  Proposition~\ref{P:vanishing2}, and Lemma~\ref{L:mainpositive2}. We  do this in this section.

We repeat the statement of Proposition~\ref{P:aperiodic2} and explain how it can be derived from results in \cite{FH17}.
\begin{proposition}\label{P:aperiodic2'}
	Let $f\colon \N\to \U$ be an aperiodic completely multiplicative function. Then  for every $\delta>0, c\in \R$ and $\ell,\ell',Q\in \N$  we have
	\begin{equation}\label{E:wdmnB'}
		\lim_{N\to\infty} \E_{m,n\in [N]} \,
		\tilde{w}_{\delta,c}(m,n)\cdot  f(\ell((Qm+1)^2+(Qn)^2))\cdot \overline{f(\ell'(Qm+1)(Qn))}=0,
	\end{equation}
	where 	$\tilde{w}_{\delta,c}(m,n)$ is as in \eqref{E:weight2}.
\end{proposition}
\begin{proof}
	Recall that
	$$
	\tilde{w}_{\delta,c}(m,n):=F_\delta \Big(\log\frac{\ell (m^2+n^2)}{\ell'mn}-c\Big), \quad m,n\in\N,
	$$
	where  $F_\delta \colon \R\to [0,1]$ is the continuous function defined  in Lemma~\ref{L:Sdelta}.
	Arguing as in the proof of \cref{P:aperiodic1'}, we get  that it suffices to  verify  \eqref{E:wdmnB'}   when
	$\tilde{w}_{\delta,c}(m,n)$ is replaced by $(m^2+n^2)^{it}\cdot (mn)^{-it}$ for arbitrary $t\in \R$.  Furthermore, 	the limit remains unchanged if we replace $(m^2+n^2)^{it}\cdot (mn)^{-it}$
	with $((Qm+1)^2+(Qn)^2)^{it}\cdot ((Qm+1)(Qn))^{-it}$.
	Hence, it suffices to establish that for every $t\in \R$ we have
	\begin{equation}\label{E:fk02}
		\lim_{N\to\infty} \E_{m,n\in [N]} \,  f_t((Qm+1)^2+(Qn)^2)\cdot \overline{f_t((Qm+1)(Qn))}=0,
	\end{equation}
	where $f_t(n):=f(n)\cdot n^{it}$, $n\in \N$.
 Note that since the indicator function of an arithmetic progression is a linear combination of Dirichlet characters, in order to establish \eqref{E:fk02}, it suffices to show that  for all Dirichlet characters $\chi,\chi'$ we have
$$
\lim_{N\to\infty} \E_{m,n\in [N]} \,  \chi(m)\cdot \chi'(n)\cdot  f_t(m^2+n^2)\cdot \overline{f_t(mn)}=0,
$$
or, equivalently, that
\begin{equation}\label{E:fk02'}
\lim_{N\to\infty} \E_{m,n\in [N]} \,   f_t(m^2+n^2)\cdot (\overline{f_t}\cdot\chi) (m)\cdot  (\overline{f_t}\cdot\chi') (n)=0.
\end{equation}
 Since $f$ is aperiodic, so is $\overline{f_t}\cdot \chi$ (and $\overline{f_t}\cdot \chi'$).
	By  \cite[Theorem~9.7]{FH17} (applied to $Q(m,n):=m^2+n^2$),  we  deduce
	that 		\eqref{E:fk02'} holds, completing the proof.
\end{proof}

Recall that  in \eqref{E:PhiK} we defined the multiplicative  F\o lner sequence $(\Phi_K)$ by
$$
\Phi_K	:=\Big\{\prod_{p\leq K}p^{a_p}\colon K<a_p\leq 2K\Big\}, \quad K\in \N.
$$
%%We let
%%\begin{equation}\label{E:Phi}
%%	\Phi:=\bigcup_{K\in \N} \Phi_K.
%%\end{equation}
Note  that
%%the sets
%%$\Phi_K$, $K\in \N$, are disjoint and
every  $q\in \N$
divides all elements of $\Phi_K$ when $K\in \N$ is large enough depending on $q$.

The next  result  is a key ingredient in the proof of \cref{P:vanishing2} below.
\begin{lemma}\label{L:keyQQ'}
	Let $f\colon \N\to \S^1$ be a  completely multiplicative function such that $f\sim \chi\cdot n^{it}$ for some $t\in \R$ and  Dirichlet character $\chi$. Let also $\delta>0, c\in \R$ be fixed,  $\tilde{w}_{\delta,c}$ be  as in \eqref{E:weight2},  and  $(\Phi_K)$ be as in \eqref{E:PhiK}. For $Q,N\in \N$  we let
	\begin{equation}\label{E:ANK}
		L_{\delta,c, N}(f, Q):=\E_{m,n\in [N]} \, \tilde{w}_{\delta,c}(m,n)\cdot f(\ell((Qm+1)^2+(Qn)^2))\cdot 	\overline{f(\ell'(Qm+1)\, n)}
	\end{equation}
	and
	\begin{equation}\label{E:LNfQ}
		\tilde{L}_{\delta,c, N}(f, Q):= Q^{-it}\cdot L_{\delta, c,N}(f, Q).
	\end{equation}
	Then
	\begin{equation}\label{E:tildeANK}
		\lim_{K\to\infty} \limsup_{N\to\infty}	\max_{Q,Q'\in \Phi_K}|\tilde{L}_{\delta,c,N}(f,Q)-\tilde{L}_{\delta,c,N}(f,Q')|=0.
	\end{equation}
\end{lemma}
\begin{proof}
	For $K\in \N$, let $F_N(f,K)$ and $G_N(f,K)$ be defined as in \eqref{E:FNfQdef} and \eqref{E:GNfQdef} respectively.

	We apply the concentration inequalities of Proposition~\ref{P:concentration1} and Proposition~\ref{P:concentration2}.
	Since $f\sim \chi\cdot n^{it}$ for some $t\in \R$ and Dirichlet character $\chi$, we get that
	$$
	\lim_{K\to\infty}	\limsup_{N\to\infty}\max_{Q\in \Phi_K} \E_{m\in [N]}|f(Qm+1)- (Qm)^{it}\exp\big(F_{N}(f,K)\big)|=0
	$$
	and
	$$
	\lim_{K\to\infty}	\limsup_{N\to\infty}\max_{Q\in \Phi_K}	\E_{m,n\in [N]}\big| f\big((Qm+1)^2+(Qn)^2\big)-  Q^{2it} \cdot (m^2+n^2)^{it} \cdot  \exp\big(G_N(f,K)\big)\big|=0.
	$$
	We deduce that  if
	$$
	M_{\delta,c, N}( f):=f(\ell)\cdot \overline{f(\ell')}\cdot \E_{m,n\in [N]}\, \tilde{w}_{\delta,c}(m,n)\cdot  (m^2+n^2)^{it}\cdot  m^{-it}\cdot \overline{f(n)},
	$$
	then
	$$
	\lim_{K\to\infty}	\limsup_{N\to\infty}\max_{Q\in \Phi_K} |\tilde{L}_{\delta, c,N}(f,Q)- M_{\delta,c,N}(f)\cdot  \exp(G_N(f,K))\cdot \overline{\exp(F_N(f,K))}|=0.
	$$
	Using this identity and the triangle inequality  we deduce that   \eqref{E:tildeANK} holds.
\end{proof}
Recall that $\CM_p$ and $\CA$ were defined in \eqref{E:pretentious} and \eqref{E:CA} respectively.
The next result follows easily from Lemma~\ref{L:Fol0} and the continuity of finite Borel measures.
\begin{lemma}\label{L:averagesmall}
	Let $\sigma$ be a Borel probability  measure on $\CM_p$.  Then for every $\varepsilon>0$  there exist a Borel subset $\CM_\varepsilon$ of $\CM_p\setminus \CA$   and  $K_0\in \N$,  such that
	\begin{equation}\label{E:MM'e}
		\sigma((\CM_p \setminus \CA) \setminus \CM_\varepsilon)\leq \varepsilon
	\end{equation}
	and
	\begin{equation}\label{E:fQsmall}
		\sup_{f\in \CM_\varepsilon} |\E_{Q\in \Phi_K}\, f(Q)\cdot Q^{-it_f}|\leq \varepsilon \, \text{ for all }\,  K\geq K_0,
	\end{equation}
	where $t_f$ is the unique real for which $f\sim \chi\cdot n^{it_f}$ for some Dirichlet character $\chi$.
\end{lemma}
\begin{remark}
	The important point is that $K_0$ does not depend on $f$ as long as $f\in \CM_\varepsilon$.
\end{remark}
\begin{proof}
	Let  $\varepsilon>0$. For $m\in \N$, we let
	$$
	\CM_{\varepsilon, m}:=\{f\in \CM_p\setminus \CA \colon |\E_{Q\in \Phi_K}\, f(Q) \cdot Q^{-it_f}|\leq \varepsilon \, \text{ for all }\,  K\geq m\}.
	$$
	Note that  by Lemma~\ref{L:Borel} the map $f\mapsto t_f$ from $\CM_p$ to $\R$ is Borel, hence  for every $\varepsilon>0$ the sets 	$\CM_{\varepsilon, m}$ form an increasing family of  Borel sets.
	Since  for $f\notin \CA$ we have   $f\cdot n^{-it_f}\neq 1$,  we get by  Lemma~\ref{L:Fol0}  that for every $f\in \CM_p\setminus \CA$ we have
	$$
	\lim_{K\to\infty} \E_{Q\in \Phi_K}\, f(Q)\cdot Q^{-it_f}=0.
	$$
	Hence,
	$$
	(\CM_p\setminus \CA):=\bigcup_{m\in\N} 	\CM_{\varepsilon, m}.
	$$
	It follows that there exists $m_0\in \N$ such that
	$$
	\sigma(	(\CM_p\setminus \CA)\setminus 	\CM_{\varepsilon, m_0})\leq \varepsilon.
	$$
	Renaming 	$\CM_{\varepsilon, m_0}$	 as $\CM_{\varepsilon}$ and letting $K_0:=m_0$, gives the asserted statement.
\end{proof}
Using  the previous two results we are going to prove Proposition~\ref{P:vanishing2}, which we formulate again for convenience.

\begin{named}{\cref{P:vanishing2}}{}
    Let $(\Phi_K)$, $\CA$, $B_{\delta,c}(f,Q; m,n)$  be defined by  \eqref{E:PhiK},  \eqref{E:CA},  \eqref{E:BdfQmn}, respectively, and $\delta>0, c\in \R$.  Let also $\sigma$ be a Borel probability measure on $\CM_p$. Then	
	%%for every $\varepsilon>0$
	%%there exists $T_0>0$, depending only on $\varepsilon$ and $\sigma$,  such that
	%%  for every $T\geq T_0$ we have
	$$
	\lim_{K\to\infty} 	\limsup_{N\to \infty}\Big|\E_{Q\in \Phi_{K}}\,  \E_{m,n\in[N]}	\int_{\CM_p \setminus \CA} \, B_{\delta,c}(f,Q; m,n)\, d\sigma(f)\Big|=0.
	$$	
\end{named}

\begin{proof}
	Let $\delta,c, \varepsilon>0$.
	%% Using the continuity of the Borel measure $\sigma$ we get that %%there exists $T_0>0$ such that
	%%\begin{equation}\label{E:TT0}
	%%	\sigma(\CA\setminus \CA_{T_0})\leq \varepsilon/4.
	%%	\end{equation}
By Lemma~\ref{L:averagesmall}  there exists $K_0=K_0(\sigma)\in \N$ and a Borel subset
$\CM_{\varepsilon}$  of $\CM\setminus \CA$, such that
\begin{equation}\label{E:MM'e'}
	\sigma((\CM_p \setminus \CA) \setminus \CM_\varepsilon)\leq \varepsilon/4
\end{equation}
and
\begin{equation}\label{E:fQsmall'}
	\sup_{f\in \CM_\varepsilon} |\E_{Q\in \Phi_K}\, f(Q)\cdot Q^{-it_f}|\leq \varepsilon/2 \, \text{ for all }\,  K\geq K_0.
\end{equation}
%%It follows by \eqref{E:TT0} and \eqref{E:MM'e'}  that
%%	\begin{equation}\label{E:TT0'}
	%%	\sigma((\CM_p \setminus \CA_{T_0}) \setminus \CM_\varepsilon)\leq \varepsilon/2
	%%\end{equation}
	Because of \eqref{E:MM'e'},  and since $|B_{\delta,c}(f,Q; m,n)|\leq 1$,
	it suffices to show that
	\begin{equation}\label{E:MAT0'}
		\limsup_{K\to\infty} 	\limsup_{N\to \infty}\Big|\E_{Q\in \Phi_{K}}\,  \E_{m,n\in[N]}	\int_{\CM_\varepsilon} \, B_{\delta,c}(f,Q; m,n)\, d\sigma(f)\Big|\leq \varepsilon.
	\end{equation}
	%%\eqref{E:MAT0} holds with $\CM_\varepsilon$ in place of $\CM_p \setminus \CA$.
	As in Lemma~\ref{L:keyQQ'}, for  $Q,N\in \N$ we let
	$$
	\tilde{L}_{\delta,c,N}(f,Q):= f(Q)\cdot Q^{-it_f}\cdot \E_{m,n\in[N]}\, B_{\delta,c}(f,Q; m,n).
	$$
	We also let for $Q,N\in \N$
	\begin{equation}\label{E:IQN}
		I(Q,N):=	\E_{m,n\in [N]}\int_{\CM_{\varepsilon}} \, B_{\delta,c}(f,Q; m,n)\, d\sigma(f)=
		\int_{\CM_{\varepsilon}} \overline{f(Q)\cdot Q^{-it_f}} \cdot  \tilde{L}_{\delta,c,N}(f,Q)\, d\sigma(f).
	\end{equation}
	Finally, for  $K\in \N$, we let $Q_K$ be an arbitrary element of $\Phi_K$, and define
	\begin{equation}\label{E:IQN1}
		I_1(Q,N):=	\int_{\CM_{\varepsilon}} \overline{f(Q)\cdot  Q^{-it_f}}\cdot  \tilde{L}_{\delta,c,N}(f,Q_K)\, d\sigma(f), \quad Q\in \Phi_K, N\in \N.
	\end{equation}
	Recall that by part~\eqref{I:Borel2} of Lemma~\ref{L:Borel} the map $f\mapsto t_f$ from $\CM_p$ to $\R$ is Borel, so the  integral defining $I_1(Q,N)$ is well defined.
	Using \eqref{E:IQN} and  \eqref{E:IQN1} we get that 
	$$
	\max_{Q\in \Phi_K}|I(Q,N)-I_1(Q,N)|\leq	\int_{\CM_{\varepsilon}} \max_{Q\in \Phi_K}| \tilde{L}_{\delta,c,N}(f,Q)- \tilde{L}_{\delta,c,N}(f,Q_K)\, d\sigma(f)|, \quad K\in \N.
	$$
	We deduce from this, equation \eqref{E:tildeANK} of Lemma~\ref{L:keyQQ'}, and the bounded convergence theorem,  that
	$$
	\lim_{K\to\infty} \limsup_{N\to\infty}	\max_{Q\in \Phi_K}|I(Q,N)-I_1(Q,N)|=0.
	$$
	It follows from the above facts that in order 	to show that \eqref{E:MAT0'} holds, it suffices to show that
	\begin{equation}\label{E:I1QNe}
		\limsup_{K\to\infty} \limsup_{N\to\infty} |\E_{Q\in \Phi_K} I_1(Q,N)|\leq \varepsilon.
	\end{equation}
	Using the definition of $I_1(Q,N)$ in \eqref{E:IQN1} and the estimate \eqref{E:fQsmall'}, we get that 	 for every $K\geq K_0$ we have
	$$
	\sup_{N\in\N}|\E_{Q\in \Phi_{K}}\, I_1(Q,N)|\leq 	\sup_{f\in \CM_{\varepsilon}} |\E_{Q\in \Phi_{K}}\, f(Q)\cdot Q^{-it_f}|\leq \varepsilon.
	$$
	Hence,
	$$
	\limsup_{K\to\infty} \limsup_{N\to\infty} |\E_{Q\in \Phi_{K}}\, I_1(Q,N)|\leq \varepsilon,
	$$
	establishing \eqref{E:I1QNe} and completing the proof.
\end{proof}
Finally we restate and prove Lemma~\ref{L:mainpositive2}. 
\begin{named}{\cref{L:mainpositive2}}{}
	Let $\sigma$ be a bounded Borel  measure on $\CM_p$ such that $\sigma(\{1\})>0$ and $\CA$ as in \eqref{E:CA}.
	Then there exist $\delta_0,\rho_0>0$, $c_0\geq 0$, depending only on $\sigma$ (for fixed $\ell,\ell'$),  such that
	\begin{equation}\label{E:mainpositive2'}
		\liminf_{N\to\infty}\inf_{Q\in \N}	\Re\Big( \E_{m,n\in[N]}\int_{\CA} B_{\delta_0,c_0}(f,Q;m,n)\, d\sigma(f)\Big)\geq\rho_0.
	\end{equation}
\end{named}
\begin{proof}
If $2\ell\leq \ell'$ (which includes the Pythagorean pair $m^2+n^2, 2mn$), then, using the positivity property  of the weight $\tilde{w}_{\delta,0}(m,n)$ in Lemma~\ref{L:Sdelta}, the proof is identical to the
 one used to establish Lemma~\ref{L:mainpositive1}. In the general case, an additional maneuver is needed. We use the weight $\tilde{w}_{\delta,c}(m,n)$ in Lemma~\ref{L:Sdelta} for some
 $c\geq L:= \log\frac{2\ell}{\ell'}$ that will be chosen appropriately.  Let
 $$
 I(c):=\int_\CA e^{ic t_f }d\sigma(f).
 $$
 Following the argument in Lemma~\ref{L:mainpositive1}, we obtain the asserted positivity  provided that
 $\Re(I(c))>0$. We use an averaging  argument to prove this. For $M>0$, we have
 $$
 \frac{1}{M} \int_L^{L+M}I(c)\, dc=\sigma(\{1\}) + \int_{\CA\setminus \{1\}} e^{iLt_f} \, \frac{e^{iMt_f}-1}{iMt_f}\, d\sigma(f).
 $$
 By the bounded convergence theorem, we get that  the second integral converges to $0$ as $M\to \infty$. Since $\sigma(\{1\}) >0$, we deduce that there exists $c>L$ such that  $\Re(I(c))>0$, completing the proof.
\end{proof}

\subsection{Proof of \cref{T:DLMS}}\label{SS:DLMS}
We sketch the proof of \cref{T:DLMS}.
Following the  reduction in Section~\ref{SS:ReductionMultiplicative} we need to show that
under the assumptions of Theorem~\ref{T:MainPythPairs} we have
%%\begin{equation}\label{E:positiveM}
%%	\limsup_{N\to\infty} \E_{m,n\in[N]}\int_{\CM} f(n(n+1))\cdot \overline{f(m^2)}\, d\sigma(f)>0.
%%\end{equation}
%%For reasons that will become clear later it will be more convenient for us to work with %%logarithmic averages.
\begin{equation}\label{E:positiveM'}
	\liminf_{N\to\infty} \lE_{m,n\in[N]}\int_{\CM} f(n(n+1))\cdot \overline{f(m^2)}\, d\sigma(f)>0.
\end{equation}
To prove this we   follow the argument used in the proof of  part~\eqref{I:MainPythPairs2} of Theorem~\ref{T:MainPythPairs}.\footnote{We follow part~\eqref{I:MainPythPairs2}  of Theorem~\ref{T:MainPythPairs} and not part~\eqref{I:MainPythPairs1} because we do not know that the limit of the averages $\lE_{m,n\in[N]}\, f(n(n+1))\cdot \overline{f(m^2)}$ exists for every $f\in \CM$.}
We will  restrict our average to the grid $\{(Qn,m)\colon m,n\in \N\}$, this is why
for $f\in\CM$ and $Q,m,n\in\N$, we let
$$
B(f,Q;m,n):=  f((Qn)(Qn+1))\cdot \overline{f(m^2)}.
$$
(For reasons that will become clear shortly, in this case we do not have to introduce any kind of weight $w_\delta$.)

%%nd it suffices to show that for some $Q\in \N$ we have (we'll show that the limit exist)
%%\begin{equation}\label{E:positiveM'}
%%	\lim_{N\to\infty} \E_{m,n\in[N]}\int_{\CM} f((Qn)(Qn+1))\cdot \overline{f((Qm)^2)}\, %%d\sigma(f)>0.
%%\end{equation}

We first claim that if $f\in \CM$ is aperiodic, then for every $Q\in \N$ we have
$$
\lim_{N\to\infty}\lE_{m,n\in[N]}\, 	B(f,Q;m,n)=0.
$$
(This corresponds to \cref{P:aperiodic2}.)
Since $f$ is completely multiplicative, it suffices to show that
%%\begin{equation}\label{E:aperiodicM}
%%\lim_{N\to\infty}\E_{m,n\in[N]} f(n)\cdot f(Qn+1))\cdot \overline{f^2(m)}=0.
%%\end{equation}
%%Since $f$ is completely multiplicative it suffices to show that
\begin{equation}\label{E:aperiodicM'}
	\lim_{N\to\infty}\lE_{n\in [N]}\,  f(n)\cdot f(Qn+1)=0 \quad \text{or} \quad 	\lim_{N\to\infty}\lE_{m\in [N]} \, f^2(m)=0.
\end{equation}
Suppose  that $f^2$ does not have logarithmic mean value $0$. Then by a consequence of a result of Hal\'asz~\cite{Hal68}, we have $f^2\sim 1$.\footnote{Hal\'asz's theorem  gives that $f^2\sim n^{it}$ for some $t\in \R$, but for logarithmic averages we have that if $g\sim n^{it}$ for some $t\neq 0$,  then $g$ has mean $0$.}  Combining  this with the following
consequence of
a result of Tao in \cite{Tao15},
we deduce that  \eqref{E:aperiodicM'} holds.
%%(we use logarithmic averages in order to have this result available).
\begin{lemma}
	Suppose that $f\in \CM$ is aperiodic and satisfies $f^2\sim 1$.
	%%$f^2\sim n^{it}$ for some $t\in \R$.
	Then for every $Q\in \N$ we have
	\begin{equation}\label{E:Tao}
		\lim_{N\to\infty}\lE_{n\in [N]} \, f(n)\cdot f(Qn+1)=0.
	\end{equation}
\end{lemma}
\begin{proof}
	We say that $f\in \CM$ is strongly aperiodic if for every Dirichlet character $\chi$ and $A\geq 1$  we have
	$\lim_{N\to\infty} \min_{|t|\leq AN}\D(f,\chi\cdot n^{it}; 1,N)=+\infty$. It was shown in \cite[Corollary~1.5]{Tao15} that if $f$ is strongly aperiodic, then \eqref{E:Tao} holds for
	every $Q\in \N$. Thus,  it  remains  to show that if $f$ is aperiodic and $f^2\sim 1$, then  $f$ is strongly aperiodic.
	%%To this end, note that since aperiodicity and strong aperiodicity are preserved under  %%multiplication by Archimedean characters we can assume that $f^2\sim 1$ and $f$ is strongly %%aperiodic. That these two assumptions imply strong aperiodicity
	This can be shown exactly as in the proof of \cite[Proposition~6.1]{F16}; the assumption
	$f^2\sim 1$ in our setting replaces the  assumption $f^k=1$ for some $k\in \N$ that was used in \cite{F16}.
\end{proof}

Using the previous claim and the bounded convergence theorem, we get that
it suffices to establish \eqref{E:positiveM'} when the range of integration $\CM$ is replaced by the subset  $\CM_p$ of pretentious multiplicative functions.

Next we claim that  if $(\Phi_K)$ is as in  \eqref{E:PhiK} and  $\sigma$ is  a Borel probability measure on $\CM_p$, then	
%%(the averages below over $m,n$ are not logarithmic)
%%for every $\varepsilon>0$
%%there exists $T_0>0$, depending only on $\varepsilon$ and $\sigma$,  such that
%%  for every $T\geq T_0$ we have
\begin{equation}\label{E:Bnodelta}
	\lim_{K\to\infty} 	\limsup_{N\to \infty}\Big|\E_{Q\in \Phi_{K}}\,  \lE_{m,n\in[N]}	\int_{\CM_p \setminus \{1\}} \, B(f,Q; m,n)\, d\sigma(f)\Big|=0.
\end{equation}
(This corresponds to \cref{P:vanishing2}. Note that $\CA$ can be replaced by $\{1\}$ in this case, which is the reason why introducing a weight  is not needed for this argument.)
To prove this, we argue as in the proof of \cref{P:vanishing2}.  If $f\sim \chi \cdot n^{it_f}$ for some $t_f\in \R$ and Dirichlet character $\chi$, for $Q,N\in \N$,  we let
$$
\tilde{L}_{ N}(f, Q):= Q^{-it_f}\cdot \lE_{m,n\in [N]}  \, f(n(Qn+1))\cdot \overline{f(m^2)}
$$
and  show that
$$
\lim_{K\to\infty} \limsup_{N\to\infty}	\max_{Q,Q'\in \Phi_K}|\tilde{L}_{N}(f,Q)-\tilde{L}_{N}(f,Q')|=0.
$$
We do this exactly as in the proof of \cref{L:keyQQ'}, using  in this case the concentration estimate of \cref{P:concentration1} for logarithmic averages (see the third remark following \cref{P:concentration1}). Then \eqref{E:Bnodelta} follows exactly as in the proof of \cref{P:vanishing2}. The reason why we only have to exclude the multiplicative function $\{1\}$ in the integral in \eqref{E:Bnodelta} (versus the set $\CA$ of all Archimedean characters), is because in  our current setting we have
$$
\E_{m,n\in[N]} \, B(f,Q;m,n)= f(Q)\cdot Q^{it_f}\cdot \tilde{L}_N(f,Q),
$$
and $Q\mapsto f(Q)\cdot Q^{it_f}$ is the trivial multiplicative function only when $f=1$.
Note also that the variant of \cref{L:mainpositive2} is trivial in our case, since $\CA$ is replaced by $\{1\}$.
With the above information we can  complete the proof of \eqref{E:positiveM'}  exactly as we did at the end of Section~\ref{SS:Plan1ii}.

\section{Pythagorean triples on level sets - Reduction to the pretentious case}\label{S:triplesreduction}
First, let us recall a convention made in Section~\ref{SS:Plan2}, which we will continue to follow in this and the next section. When we write $\E^*_{k\in \N}$, we mean the limit $\lim_{K\to\infty} \E_{k\in \Phi_K}$, where $(\Phi_K)$ is a multiplicative F\o lner sequence chosen so that all the limits in the following statements exist. Since it will always be the case in our arguments that only a countable collection of limits needs to be considered, such a F\o lner sequence can be taken as a subsequence of any given
multiplicative F\o lner sequence.

As explained in Section~\ref{SS:Plan2}, the proof of	Theorem~\ref{T:triplesrestated}
splits in two parts,  Propositions~\ref{P:reduction} and \ref{P:pretentiousfinite}.
%%{\bf If we only deal with finite valued cmfs I can rewrite the argument in this case only, this will allow some small simplifications. }
Our goal in this section is to establish the first part, which we now state in a more general form (we do not assume that $f$ takes finitely many values).
\begin{proposition}\label{P:reduction'}
	%%	Let $\ell_1,\ell_2,\ell_3\in \N$.
	Suppose that for every  completely multiplicative function  $h\colon \N\to \S^1$, with $h\sim n^{it}$ for some $t\in \R$, modified Dirichlet character $\tilde{\chi}\colon \N\to \S^1$, and open arc $I$ on $\S^1$ around $1$,  we have
	$$
	\liminf_{N\to \infty} \E_{m,n\in [N],m>n}\, \E^*_{k\in \N}\ A(k\, (m^2-n^2))\cdot A(k\cdot 2 mn)
	\cdot  A(k\,(m^2+n^2))>0,
	$$
	where
	$$
	A(n):=F(h(n))\cdot F(\tilde{\chi}(n)), \quad n\in \N, \quad F:={\bf 1}_I.
	$$
	Then  for every completely multiplicative function $f\colon \N\to \S^1$ and open arc $I$ around $1$,  we have
	$$
	\liminf_{N\to \infty} \E_{m,n\in [N],m>n}\, \E^*_{k\in \N} \, F(f(k\, (m^2-n^2)))\cdot F(f(k\, 2 mn))
	\cdot  F(f(k\,  (m^2+n^2)))>0,
	$$
	where $F$ is as before.  	Furthermore, if our assumptions hold for all finite-valued completely multiplicative functions $h$, then the conclusion holds for all finite-valued completely multiplicative functions $f$.
\end{proposition}

\subsection{Preparation}
Recall that we write  $f\sim g$ if $\D(f,g)<+\infty$ where $\D(f,g)$ is as in \eqref{E:Dfg}.
\begin{lemma}\label{L:roots}
	Let $f\colon \N\to \S^1$ be  a completely multiplicative function such that $f\sim n^{it}$ for some $t\in \R$. Then for every $d\in \N$, there exists a completely multiplicative function $g\colon \N\to \S^1$, such that $g\sim n^{it/d}$ and  $g^d=f$. Furthermore, if $f$ takes finitely many values, then so does $g$.
\end{lemma}
\begin{proof}
	Suppose first that $f\sim 1$. Then $f(p)=e(\theta_p)$, $p\in \P$, for some $\theta_p\in [-1/2,1/2)$ with  $\sum_{p\in\P}\frac{1-\cos(\theta_p)}{p}<+\infty$. Hence, $\sum_{p\in\P}\frac{\theta_p^2}{p}<+\infty$.  We define the completely multiplicative function $g\colon \N\to \S^1$ by
	$$
	g(p):=e(\tilde{\theta}_p), \  \text{ where } \, \tilde{\theta}_p:=\theta_p/d, \ p\in \P.
	$$
	We have  $g^d=f$. Also   $\sum_{p\in\P}\frac{\tilde{\theta}_p^2}{p}<+\infty$, hence $g\sim 1$.
	
	Now suppose that  $f\sim n^{it}$, and let $d\in \N$.  Then $f\cdot n^{-it}\sim 1$, and the previous argument gives that there exists $h\colon \N\to \S^1$ with $h\sim 1$ such that
	$h^d=f\cdot n^{-it}$. Let $g:=h\cdot n^{it/d}$. Then $g^d=f$ and $g\sim n^{it/d}$.	
\end{proof}
A similar statement is not always true when $f\sim \chi$ where $\chi$ is a Dirichlet character (not even when $f=\chi$).

We remind the reader that modified Dirichlet characters $\tilde{\chi}$ were defined in Section~\ref{SS:multiplicative}.
If a completely multiplicative function  $f\colon \N\to \S^1$ is such that $f^l$ is aperiodic for every $l\in \N$, then things are easier for us. If this is not the case (for example, it is never the case when $f$ is finite-valued), then the next lemma gives a useful decomposition to work with.
\begin{lemma}\label{L:decomposition}
	Let $f\colon \N\to \S^1$ be an aperiodic completely multiplicative function such that $f^d$ is  pretentious for some $d\in \N$, and suppose that $d\geq 2$ is the smallest such $d$.  Then there exist  completely multiplicative functions $g,h\colon \N\to \S^1$ and a Dirichlet character $\chi$, such that
	\begin{enumerate}
		\item $f=g\cdot h$
		
		\item $g,\ldots, g^{d-1}$  are aperiodic and $g^d=\tilde{\chi}$.
		
		\item $h\sim n^{it}$ for some $t\in \R$.
	\end{enumerate}
	Furthermore, if $f$ takes finitely many values, then so does $h$ and $h\sim 1$.
\end{lemma}
\begin{proof}
	By our assumption, we have that $f,\ldots, f^{d-1}$ are aperiodic and $f^d\sim \chi\cdot  n^{it}$ for some $t\in \R$ and Dirichlet character $\chi$. Then $f^d\cdot \overline{\tilde{\chi}}\sim n^{it}$, and  Lemma~\ref{L:roots} gives that there exists a completely multiplicative function $h\colon \N\to \S^1$ such that
	$$
	h\sim n^{it/d} \quad \text{and} \quad
	h^d=f^d\cdot \overline{\tilde{\chi}}.
	$$
	Let $g:=f\cdot \overline{h}$. Then obviously $f=g\cdot h$. Also, for $j=1,\ldots, d-1$ we have
	$
	g^j=f^j\cdot h^j
	$
	is aperiodic, since by assumption $f^j$ is aperiodic and $h^j$ is pretentious. Moreover,
	$$
	g^d=f^d\cdot \overline{h}^d=\tilde{\chi}.
	$$
	
	Lastly, suppose that $f$ takes finitely many values. Since $g$ also takes finitely many values, and $h:=f\cdot \overline{g}$, we have that $h$ takes finitely many values. Also, since $h$ takes finitely many values and $h\sim n^{it}$ for some $t\in \R$, we have that $t=0$.
	This completes the proof.
\end{proof}
Since $\chi$ is a Dirichlet character, there exists $r\in \N$ such that $\tilde{\chi}^r=1$.
We gather some facts about $g$ that we shall use in the proof of \cref{P:reduction'}:
\begin{itemize}
	\item $g^{rd}=\tilde{\chi}^r=1$, hence $g$ takes values in $(rd)$-roots of unity and
	the sequence $(g^j)_{j\in \N}$ is periodic with period $rd$.
	\smallskip
	
	\item
	$g^{d}=\tilde{\chi}$,   $g^{2d}=\tilde{\chi}^2$, $\ldots$,  $g^{(r-1)d}=\tilde{\chi}^{r-1}$, $g^{rd}=1$.
	
	\smallskip
	
	\item $	g^j$    is aperiodic if   $ j\not\equiv 0 \! \pmod{d}$.
\end{itemize}

\subsection{Proof of Proposition~\ref{P:reduction'}}  In this subsection we prove
Proposition~\ref{P:reduction'}. For convenience we use the following notation.
\begin{definition}	If $I$ is a circular arc around $1$ and $d\in \N$, we let
	$$
	I/d:=\{e(t/d)\colon e(t)\in I, t\in [-1/2,1/2)\}.
	$$
\end{definition}
Let $f\colon \N\to \S^1$ be a completely multiplicative function and  $I$ be an open arc around $1$.
Let also $F\colon \S^1\to [0,1]$ be a continuous function such that
$$
{\bf 1}_{I/4}\leq F\leq {\bf 1}_{I/2}.
$$
It suffices to show that under the assumption of Proposition~\ref{P:reduction'} we have
\begin{equation}\label{E:Fpositive}
	\liminf_{N\to \infty} \E_{m,n\in [N],m>n}\, \E^*_{k\in \N}\ F(f(k(m^2-n^2)))\cdot F(f(k\, 2mn))
	\cdot  F(f(k(m^2+n^2)))>0.
\end{equation}

We consider three cases.

\smallskip

{\bf Case 1.} If $f$ is pretentious, then $f= h\cdot \tilde{\chi}$, where $h\sim n^{it}$ for some $t\in \R$, and $\tilde{\chi}$ is a modified Dirichlet character, and the conclusion   follows from our assumption.

\smallskip

{\bf Case 2.} Suppose that   $f$ is aperiodic and $f^d$ is pretentious for some $d\geq 2$. We use Lemma~\ref{L:decomposition} to get a decomposition $f=gh$, where
$g$ takes values on $(rd)$-roots of unity for some $r\in \N$,  $g,\ldots, g^{d-1}$  are aperiodic and $g^d=\tilde{\chi}$ for some modified Dirichlet character $\tilde{\chi}$,
and $h\sim n^{it}$ for some $t\in \R$.
%% $g,h$ satisfy the properties mentioned there, and we use the notation introduced there and the %%remarks following the lemma.
Note first that in order to establish \eqref{E:Fpositive} it suffices to show that
\begin{equation}\label{E:Fpositivea}
	\liminf_{N\to \infty} \E_{m,n\in [N], m>n}\, \E^*_{k\in \N}\
	c_{k,m,n} \cdot
	F(g(k(m^2-n^2)))\cdot F(g(k\, 2mn))
	\cdot  F(g(k(m^2+n^2)))>0,
\end{equation}
where
\begin{equation}\label{E:ckmn}
	c_{k,m,n}:=
	F(h(k(m^2-n^2)))\cdot F(h(k\, 2mn))
	\cdot  F(h(k(m^2+n^2))), \quad k,m,n\in \N.
\end{equation}
This is so, since if $g(n),h(n)\in I/2$, then $f(n)=g(n)\cdot h(n)\in I$.

\smallskip

{\bf Main Claim.} {\em  If  for  $G:={\bf 1 }_{\{1\}}$ and $c_{k,m,n}$ as in \eqref{E:ckmn} we have
	\begin{equation}\label{E:chiFpositive}
		\liminf_{N\to \infty} \E_{m,n\in [N], m>n}\, \E^*_{k\in \N}\
		c_{k,m,n} \cdot
		G(\tilde{\chi}(k(m^2-n^2)))\cdot G(\tilde{\chi}(k\, 2mn))
		\cdot  G(\tilde{\chi}(k(m^2+n^2)))>0,
	\end{equation}
	then \eqref{E:Fpositivea} holds.}

\smallskip

Note that  \eqref{E:chiFpositive} is satisfied from the hypothesis of Proposition~\ref{P:reduction'}. So to finish the proof of Proposition~\ref{P:reduction'} in Case 2, it  remains to verify  the above claim.

We start with a simple identity. Since $g$ takes values in $rd$ roots of unity, we have
$$
{\bf 1}_{g=1}=\E_{0\leq j <rd}\, g^j.
$$
Since  $F\geq {\bf 1}_{\{1\}}$,
it suffices to verify \eqref{E:Fpositivea} with
$\sum_{j=0}^{rd-1}g^j$ in place of $F\circ g$.
Let
$$
J:=\{0\leq j<rd \colon  j\not\equiv 0 \! \pmod{d}\}.
$$
Recall that $g^j$ is aperiodic for $j\in J$.
Also $g^d=\tilde{\chi}$ and $\tilde{\chi}$ takes values on $r$-th roots of unity, hence
$$
\sum_{j=0}^{rd-1}g^j=\sum_{j=0}^{r-1}\tilde{\chi}^j +\sum_{j\in J }g^j=r \cdot {\bf 1}_{\tilde{\chi}=1}+ \sum_{j\in J }g^j.
$$
Hence, in order to verify   \eqref{E:Fpositivea}, it suffices to show that
\begin{equation}\label{E:Hpositive}
	\liminf_{N\to \infty} \E_{m,n\in [N], m>n}\, \E^*_{k\in \N}\
	c_{k,m,n} \cdot
	H(k(m^2-n^2))\cdot H(k\, 2mn)
	\cdot  H(k(m^2+n^2))>0,
\end{equation}
where
$$
H:=r \cdot {\bf 1}_{\tilde{\chi}=1}+ \sum_{j\in J}g^j.
$$
%%Since  $g^j$    is aperiodic if   $ j\not\equiv 0 \! \pmod{d}$
After expanding the product we get a finite sum of expressions of the form
\begin{equation}\label{E:Hi}
	\liminf_{N\to \infty} \E_{m,n\in [N],m>n}\, \E^*_{k\in \N}\
	c_{k,m,n} \cdot
	H_1(k(m^2-n^2))\cdot H_2(k\, 2mn)
	\cdot  H_3(k(m^2+n^2)),
\end{equation}
where  each  $H_1,H_2,H_3$ is either of the form $r\cdot {\bf 1}_{\tilde{\chi}=1}$, or of
the  form  $g^j$    for some $j\in J$.

With this in mind, we see that the positivity property \eqref{E:Hpositive} would follow once we establish  the following three claims:
\begin{enumerate}
	\item \label{I:Hi1} If $H_1=H_2=H_3=r\cdot  {\bf 1}_{\tilde{\chi}=1}$, then the limit in \eqref{E:Hi} is positive.
	
	\item\label{I:Hi2} If $H_1=H_2=r\cdot  {\bf 1}_{\tilde{\chi}=1}$ and $H_3=g^j$ for some $j\in J$,
	then the limit in \eqref{E:Hi} is $0$.
	
	\item \label{I:Hi3} If $H_1=g^j$ or $H_2=g^j$ for some $j\in J$,
	then the limit in \eqref{E:Hi} is $0$.
\end{enumerate}
(We do not combine the last two cases because the argument we use is different.)

We prove \eqref{I:Hi1}. This  follows immediately from the assumption \eqref{E:chiFpositive} of the Main Claim.

We prove \eqref{I:Hi2}. We will show that for every $m,n\in \N$ with $m>n$  we have
$$
\E^*_{k\in \N}\
c_{k,m,n} \cdot
H_1(k(m^2-n^2))\cdot H_2(k\, mn)
\cdot  H_3(k(m^2+n^2))=0.
$$
Using the definition of $c_{k,m,n}$ in \eqref{E:ckmn} and  uniform approximation of $F$, it suffices to show that  for every $m,n\in \N$  with $m>n$ we have
$$
\E^*_{k\in \N}\
H_1'(k(m^2-n^2))\cdot H_2'(k\, mn)
\cdot  H_3'(k(m^2+n^2))=0,
$$
where $H_1':=\tilde{\chi}^{j_1}\cdot h^{j_2}$, $H_2':=\tilde{\chi}^{j_3}\cdot h^{j_4}$, and  $H_3':= g^{j_5}\cdot h^{j_6}$, for some $j_1,j_2,j_3,j_4,j_6\in \Z$ and $j_5:=j\in J$.
Factoring out the  multiplicative average $\E^*_{k\in \N}$, we get that it suffices to show
that
$$
\E^*_{k\in \N}\ \tilde H(k)=0 \quad \text{where} \quad \tilde H:=\tilde{\chi}^{j_1+j_3}\cdot h^{j_2+j_4+j_6} \cdot g^{j_5}.
$$
Since $g^{j_5}$ is aperiodic and $\tilde{\chi}^{j_1+j_3}\cdot h^{j_2+j_4+j_6}$ is pretentious, we get that $\tilde H\neq 1$, hence $ \E^*_{k\in \N}\ \tilde H(k)=0$.

We prove \eqref{I:Hi3}. Suppose that $H_1=g^{j_1}$ for some $j_1\in J$, the argument is similar for $j_2$. Using the definition of $c_{k,m,n}$ from \eqref{E:ckmn} and uniform approximation of $F$, it suffices to show that
$$
\lim_{N\to \infty} \E_{m,n\in [N], m>n}\, \E^*_{k\in \N}\
H_1'(k(m^2-n^2))\cdot H_2'(k\, mn)
\cdot  H_3'(k(m^2+n^2))=0,
$$
where $H_1':=g^{j_1}\cdot h^{j_2}$, $H_2':=\tilde{\chi}^{j_3}\cdot h^{j_4}$ or $g^{j_5}\cdot h^{j_6}$, $H_3':=\tilde{\chi}\cdot h^{j_7}$ or $g^{j_8}\cdot h^{j_9}$, for some $j_2,\ldots, j_9\in \Z$.
%\jpm{Again I think $H_2'$ and $H_3'$ should be a bit different. {\bf Nikos: Yes, made the change, please doublecheck.}}
Factoring out the multiplicative average
$ \E^*_{k\in \N}\ (H_1'\cdot H_2'\cdot H_3')(k)$, we get that it suffices to show that
\begin{equation}\label{E:Hi3}
	\lim_{N\to \infty} \E_{m,n\in [N], m>n}\,
	H_1'((m^2-n^2))\cdot H_2'( mn)
	\cdot  H_3'(m^2+n^2)=0,
\end{equation}
where $H_1'$ is an aperiodic completely multiplicative function (since $g^{j_1}$ is aperiodic and $h^{j_2}$ is pretentious), and $H_2'$, $H_3'$ are completely multiplicative functions. The hypothesis of  Proposition~\ref{P:aperiodic}  is satisfied and we deduce that \eqref{E:Hi3} holds.

This finishes the proof of the Main Claim and the proof of Case 2.

\smallskip

{\bf Case 3.} Suppose that  $f^l$ is aperiodic for every $l\in \N$.
In this case we claim that the following identity holds
$$
\liminf_{N\to \infty} \E_{m,n\in [N], m>n}\, \E^*_{k\in \N}\,   F(f(k(m^2-n^2)))\cdot F(f(k\, mn))
\cdot  F(f(k(m^2+n^2)))= \Big(\int F\, dm_{\S^1}\Big)^3.
$$
If we prove this, then \eqref{E:Fpositive} holds,  since $\int F\, dm_{\S^1}\geq m_{\S^1}(I/4)>0$.

Using uniform approximation of $F$,  it suffices to show that
$$
\lim_{N\to \infty} \E_{m,n\in [N],m>n}\, \E^*_{k\in \N} \, F_1(k(m^2-n^2))\cdot F_2(k\, mn)
\cdot  F_3(k(m^2+n^2))=0
$$
when for $i=1,2,3$ we have $F_i=f^{j_i}$, $j_i\in \Z$, and at least one of the $j_1,j_2,j_3$ is non-zero.

We consider two cases.
Suppose first that  $j_1=j_2=0$. Then $j_3\neq 0$. After factoring out the multiplicative average $\E^*_{k\in \N}$ it suffices to show that
$$
\E^*_{k\in \N}   \, f^{j_3}(k)=0.
$$
This is the case since $f^{j_3}$ is a non-trivial completely multiplicative function.

Suppose now that  $j_1\neq 0$, the argument is similar if $j_2\neq 0$.
After factoring out the multiplicative average $\E^*_{k\in \N}$ it suffices to show that
\begin{equation}\label{E:F123}
	\lim_{N\to \infty} \E_{m,n\in [N], m>n}\, F_1((m^2-n^2))\cdot F_2( 2mn)
	\cdot  F_3(m^2+n^2)=0.
\end{equation}
By our assumption, $F_1=f^{j_1}$ is  aperiodic. Note also that  all $F_1,F_2,F_3$ are completely multiplicative function. The asserted identity then follows again  from Proposition~\ref{P:aperiodic}.\footnote{It is crucial for this part of the argument that we avoided working with an aperiodicity assumption on $F_3$, since such an assumption does not imply that \eqref{E:F123} holds (but it does hold if $F_1$ or $F_2$ are aperiodic completely multiplicative functions).}

\section{Pythagorean triples on level sets - The pretentious case}\label{S:Triples}
Our goal in this section is to prove Proposition~\ref{P:pretentiousfinite}, which combined with Proposition~\ref{P:reduction'} (\cref{P:reduction} is a direct consequence) implies Theorem~\ref{T:Triples}.
We first restate	 Proposition~\ref{P:pretentiousfinite} in a slightly more convenient form.
Let $f\colon \N\to \S^1$ be a pretentious completely multiplicative function  taking finitely many values. Then for some $d\in \N$ it takes values on $d$-th roots of unity.  We can assume that $d$ is minimal with this property, in which case we have $f^j\neq 1$ for
$j=1,\ldots, d-1$. In this case we will show the following.
\begin{proposition}\label{P:pretentiousfinite'}
	Let $d\in \N$ and    $f\colon \N\to \S^1$ be pretentious multiplicative function taking values on $d$-th roots of unity and $\tilde{\chi}\colon \N\to \S^1$ be a modified Dirichlet character.  Then
	$$
	\liminf_{N\to \infty} \E_{m,n\in [N], m>n}\, \E^*_{k\in \N}\ A(k\, (m^2-n^2))\cdot A(k\cdot 2mn)
	\cdot  A(k\, (m^2+n^2))>0,
	$$
	where
	\begin{equation}\label{E:An}
		A(n):=F(f(n))\cdot F(\tilde{\chi}(n)), \quad n\in \N, \quad  F:={\bf 1}_{\{1\}}.
	\end{equation}
\end{proposition}
\begin{remark}
	%%$\bullet$
	Note that in the argument that follows we only deal with countably many choices of multiplicative functions  and other choices of parameters, so we can choose a subsequence of positive integers $(N_l)$ along which all the limits (as $l\to\infty$) that appear below exist. We make this implicit assumption throughout. 		
	%%	$\bullet$ The multiplicative average $\E^*_{k\in \N}$ does not help in the finite %%valued case, but we have to introduce it  in order to be able to combine this result %%with Proposition~\ref{P:reduction}.
\end{remark}
Before giving the proof of Proposition~\ref{P:pretentiousfinite'}, we
show how the concentration estimates of Corollary~\ref{C:concentration2} follow from Propositions~\ref{P:concentration1} and \ref{P:concentration2}.
\subsection{Proof of Corollary~\ref{C:concentration2}}
We will deduce part~\eqref{I:1} from  Proposition~\ref{P:concentration1}. In a similar fashion
we can deduce part~\eqref{I:2} from  Proposition~\ref{P:concentration2}.

Let $\varepsilon>0$ and $\varepsilon<1$. Since $f$ is a finite-valued pretentious multiplicative function we have by Lemma~\ref{L:convergesabsolutely}
that $f\sim \chi$ for some Dirichlet character $\chi$ with period $q$ and
$$
\sum_{p\in  \P}\, \frac{1}{p}\,|1-f(p)\cdot \overline{\chi(p)}|<\infty.
$$
Hence, there exists $K_0\in \N$ such that
$$
\sum_{p\geq K_0}\, \frac{1}{p}\,|1- f(p)\cdot \overline{\chi(p)}| +K_0^{-1/2}\leq \varepsilon.
$$
This implies that
$$
\mathbb{D}(f,\chi;K_0,\infty)\leq \varepsilon
\quad \text{and} \quad  \big|\exp\big(F_N(f,K_0)\big)-1\big|\ll \varepsilon,
$$
where
$F_N(f,K_0)=\sum_{K_0< p\leq N} \frac{1}{p}\,\big(f(p)\cdot \overline{\chi(p)} -1\big)$.

We let $Q_0:=q\cdot \prod_{p\leq K_0}p$. If $Q\in \N$ is such that $Q_0\mid Q$, then using the second remark following Proposition~\ref{P:concentration1}   with $t=0$, we get that
$$
\limsup_{N\to\infty} \E_{n\in [N]}|f(Qn+1)-  \exp\big(F_N(f,K_0)\big)|\ll \varepsilon.
$$
Since $|\exp\big(F_N(f,K_0)\big)-1|\ll \varepsilon$,
this completes the proof.
\subsection{Proof of Proposition~\ref{P:pretentiousfinite'}}
%%Again, to ease notation we give the arguiment when $\ell_1=\ell_2=\ell_3=1$, the %%argument in the general case is similar.
Recall that  $A(n)$ is given by \eqref{E:An}.
Since $A(n)\geq 0$ for every $n\in \N$, 	it suffices to show that there exist $Q\in \N$ and $N_l\to \infty$ (which can be taken to be a subsequence of any given $M_l\to\infty$) such that  all limits appearing below as $l\to \infty$ exist and
\begin{multline*}
	\lim_{l\to \infty} \E_{m,n\in [N_l], m>n}\, \E^*_{k\in \N} \, A(k((Qm+1)^2-(Qn)^2)\cdot
	A(k\, 2(Qm+1)(Qn))
	\cdot \\  A(k((Qm+1)^2+(Qn)^2))>0.
\end{multline*}

Since $f$ takes values on $d$-th roots of unity  and $\tilde{\chi}$ takes values on $d'$-th roots of unity for some $d,d'\in \N$, we have
\begin{equation}\label{E:Ff}
	F(f)= {\bf 1}_{f=1}=\E_{0\leq j <d}\, f^j, \quad 	F(\tilde{\chi})= {\bf 1}_{\tilde{\chi}=1}=\E_{0\leq j <d'}\, \tilde{\chi}^j.
\end{equation}
Let $m,n \in \N$ with $m>n$ be fixed. In order to compute
$$
\E^*_{k\in \N} \, A(k(m^2-n^2))\cdot A
(k\, 2mn)
\cdot  A(k(m^2+n^2)),
$$
we use \eqref{E:Ff},  expand, and use that  by Lemma~\ref{L:Fol0} we have $ \E^*_{k\in \N} \ g(k)=0$ for all completely multiplicative functions
$g\colon \N\to \U$ with $g\neq 1$ (in particular this holds if $g:=f^k\cdot \tilde{\chi}^{k'}\neq 1$). We see that the previous expression is equal to $1/(dd')^3$ times
$$
\sum_{k_i,k_i'\in \CK} \,   (f^{k_1}\cdot \tilde{\chi}^{k'_1})(m^2-n^2)\cdot (f^{k_2}\cdot \tilde{\chi}^{k'_2})(2mn)\cdot (f^{k_3}\cdot \tilde{\chi}^{k'_3})(m^2+n^2),
$$
where
$$
\CK:=\{0\leq k_1,k_2,k_3<d,\, 0\leq k'_1,k'_2,k'_3< d' \colon f^{k_1+k_2+k_3}\cdot \tilde{\chi}^{k_1'+k_2'+k_3'}=1\}.
$$
In what follows, we implicitly assume that all $k_i,k_i'$ belong to $\CK$.
%%\footnote{Since we only deal with  the finite valued case here, we will not use the aditional properties we get because $k_i,k_i'$ belong in $\CK$.
	%%	These are only useful in the general case where oscillatory factros exist and we want to cancel them.}

Let $q$ be the period of $\chi$, then $\tilde{\chi}(qn+1)=1$ for every $n\in \N$.
Taking the previous facts in mind, we see that  in order to establish the needed positivity it suffices to show that there exists $Q\in \N$ such that  $q\mid Q$ and
\begin{equation}\label{E:k123}
	L(Q):=  \sum_{k_i,k_i'\in \CK}  \,    \Re(L_{k_1,k_2,k_3,k_2'}(Q))>0,
\end{equation}
where
\begin{multline}\label{E:Lk123}
	L_{k_1,k_2,k_3,k_2'}(Q):=\lim_{l\to\infty} \E_{m,n\in [N_l], m>n}  \,   f^{k_1}((Qm+1)^2-(Qn)^2)\cdot f^{k_2}(2(Qm+1)(Qn))\cdot \\
	f^{k_3}((Qm+1)^2+(Qn)^2)\cdot \tilde{\chi}^{k_2'}(2(Qn)).
\end{multline}
(We used that $\tilde{\chi}(j)=1$ for $j\in Q\Z+1$.)

%%	Let also
%%\begin{equation}\label{E:k123*}
%%		L^*(Q):=  \sum_{k_1,k_2,k_3,k'_2 \,  \text{ not all 0}} \,     \Re\big(L_{k_1,k_2,k_3,k_2'}(Q)\big).
%%	\end{equation}
%%	Since  $L_{0,0,0,0}(Q)=1$ for every $Q\in \N$, it suffices to show that there exists $Q\in \N$  %%such that
%%	$L^*(Q)>-1$. In fact, we will show that for every $\varepsilon>0$ there exists %%$Q=Q(f,\tilde{\chi},\varepsilon)\in \N$ such that $L^*(Q)\geq -\varepsilon$.
%%	This will be an immediate consequence of the following two claims.

\medskip

{\bf Claim 1 ($f^{k_2}\cdot \tilde{\chi}^{k_2'}=1$).}
{\em For every  $\varepsilon>0$  there exists $Q_0=Q_0(f,\tilde{\chi},\varepsilon)\in \N$  with $q\mid Q_0$, such that the following holds: If $Q\in \N$ satisfies $Q_0\mid Q$, then   for all $k_1,k_2,k_3,k_2'\in \CK$ with $ f^{k_2}\cdot \tilde{\chi}^{k_2'}=1$ we have
\begin{equation}\label{E:Lkk}
	\Re(L_{k_1,k_2,k_3,k_2'}(Q))\geq 1-\varepsilon.
\end{equation}
As a consequence, there exists $Q_0:=Q_0(f,\tilde{\chi} )$, such that if $Q\in \N$ satisfies $Q_0\mid Q$, then }
\begin{equation}\label{E:geq1}
\sum_{k_1,k_2,k_3,k_2'\colon f^{k_2}\cdot \tilde{\chi}^{k_2'}= 1}
\Re(L_{k_1,k_2,k_3,k_2'}(Q))\geq 1.
\end{equation}

We prove the claim. Let $\varepsilon>0$. Note first that since $f^{k_2}\cdot \tilde{\chi}^{k_2'}=1$, we get using \eqref{E:Lk123} that
\begin{multline}\label{E:Lk123'}
L_{k_1,k_2,k_3,k_2'}(Q):=\lim_{l\to\infty} \E_{m,n\in  [N_l], m>n}  \,   f^{k_1}(Q(m-n)+1)\cdot f^{k_1}(Q(m+n)+1) \\
f^{k_2}(Qm+1)\cdot f^{k_3}((Qm+1)^2+(Qn)^2).
\end{multline}
Using this identity,  Corollary~\ref{C:concentration2}, and Lemma~\ref{L:lN},  we deduce that    there exists $Q_0=Q_0(f,\varepsilon)$, with $q\mid Q_0$,
such that if
$Q\in \N$  satisfies  $Q_0\mid Q$, then   for all $k_1,k_2,k_3,k_2'\in \CK$ such that $ f^{k_2}\cdot \tilde{\chi}^{k_2'}=1$ we have
$$
| L_{k_1,k_2,k_3,k_2'}(Q)-1|\ll_d \varepsilon.
%% \quad \text{as long as } f^{k_2}\cdot \tilde{\chi}^{k_2'}=1.
$$
This proves \eqref{E:Lkk}. 	 Since $L_{0,0,0,0}=1$, using \eqref{E:Lkk} for $\varepsilon=1/2$, we deduce \eqref{E:geq1}.
This completes the proof of  Claim~1.

\medskip

{\bf Claim 2 ($f^{k_2}\cdot \tilde{\chi}^{k_2'}\neq 1$).}
%%	{\em Suppose that $f^{k_2}\cdot \tilde{\chi}^{k_2'}\neq 1$.  %%Then
{\em Let $Q_0\in \N$ be such that \eqref{E:geq1} holds for every $Q\in \N$ such that $Q_0\mid Q$. 	Then for every  $\varepsilon>0$
	there exists $Q_1=Q_1(f,\tilde{\chi},\varepsilon)\in \N$  such that $Q_0\mid Q_1$  (hence \eqref{E:Lkk} holds for $Q=Q_1$) and}
\begin{equation}\label{E:geqe}
	\sum_{k_1,k_2,k_3,k_2'\colon f^{k_2}\cdot \tilde{\chi}^{k_2'}\neq 1}
	\Re(L_{k_1,k_2,k_3,k_2'}(Q_1))\geq - \varepsilon.
	%%, \quad
	%%	\text{as long as }  f^{k_2}\cdot \tilde{\chi}^{k_2'}\neq1.
\end{equation}

We prove the claim. Let $\varepsilon>0$. It suffices to show
that
\begin{equation}\label{E:Lk120}
	\lim_{K\to\infty}\E_{Q\in \Phi_K} \,  L_{k_1,k_2,k_3,k_2'}(Q)=0, \quad
	\text{as long as }  f^{k_2}\cdot \tilde{\chi}^{k_2'}\neq1.
\end{equation}
Note that
$$
L_{k_1,k_2,k_3,k_2'}(Q):=(f^{k_2}\cdot \tilde{\chi}^{k_2'}) (2Q)\cdot 	L'_{k_1,k_2,k_3,k_2'}(Q),
$$
where
\begin{multline}\label{E:Lk123''}
	L'_{k_1,k_2,k_3,k_2'}(Q):= \lim_{l\to\infty} \E_{m,n\in  [N_l], m>n}  \,   f^{k_1}(Q(m-n)+1)\cdot f^{k_1}(Q(m+n)+1)\\
	f^{k_2}(Qm+1) \cdot f^{k_3}((Qm+1)^2+(Qn)^2)\cdot
	f^{k_2}(n)\cdot  \tilde{\chi}^{k'_2}(n).
\end{multline}

We prove \eqref{E:Lk120}. Let $\varepsilon'>0$.
Using  \eqref{E:Lk123''},  Corollary~\ref{C:concentration2}, and Lemma~\ref{L:lN}, we get that  there exists $ Q_2=Q_2(f,\varepsilon')$ such that the following holds: If $Q\in \N$ satisfies $Q_2\mid Q$, then, for all
$k_1,k_2,k_3,k_2'\in\CK$ we have
\begin{equation}\label{E:CLk123}
	| L'_{k_1,k_2,k_3,k_2'}(Q)-c_{k_2,k'_2}|\ll \varepsilon',
\end{equation}
where
$$
c_{k_2,k'_2}:=\lim_{l\to\infty}\E_{n\in [N_l]} \, f^{k_2}(n)\cdot  \tilde{\chi}^{k'_2}(n).\footnote{The limit exists since $f^{k_2}\cdot \tilde{\chi}^{k'_2}$ is finite-valued, but we do not have to use this.}
$$
Hence, by \eqref{E:Lk123}, \eqref{E:Lk123''}, and \eqref{E:CLk123}, we have
\begin{equation}\label{E:CLk123'}
	| L_{k_1,k_2,k_3,k_2'}(Q)-c_{k_2}\cdot (f^{k_2}\cdot \tilde{\chi}^{k_2'})(2Q)|\ll \varepsilon'\quad \text{for all } Q \text{ with } Q_2\mid Q.
\end{equation}
Since by assumption $f^{k_2}\cdot \tilde{\chi}^{k_2'}\neq 1$, we have
$$
\lim_{K\to\infty}  \E_{Q\in \Phi_K} \, (f^{k_2}\cdot \tilde{\chi}^{k_2'})(Q)=0.
$$
Combining this with \eqref{E:CLk123'},  we get that \eqref{E:Lk120} holds.  This proves  Claim~2.

Putting together the two claims,  in particular the estimates \eqref{E:geq1} and \eqref{E:geqe}, we deduce that for every $\varepsilon>0$  there exists $Q_1=Q_1(f, \tilde{\chi}, \varepsilon)\in \N$ with $q\mid Q_1$, such that $L(Q_1)\geq 1-\varepsilon$, hence \eqref{E:k123} holds for $Q=Q_1$.  This completes the proof.

\subsection{More general equations}\label{SS:abc}
Our methods allow us to extend Theorem~\ref{T:Triples} and cover
more general equations of the form
\begin{equation}\label{E:abcz}
	ax^2+by^2=cz^2
\end{equation}
where $a,b,c\in \N$ are  squares satisfying Rado's condition, i.e., we  have either  $a=c$, or $b=c$, or $a+b=c$.  We summarize the key differences in the argument.

Suppose first that  $a=c$ (the case $b=c$ is similar).   Then, as in Section~\ref{SS:parametric}, we get parametrizations of \eqref{E:abcz}
of the form
$$
x=\ell_1\, (m^2-n^2), \quad  y=\ell_2\, mn,  \quad z=\ell_3\, (m^2+n^2),
$$  for some $\ell_1,\ell_2,\ell_3\in \N$, and our hypothesis $a=c$ implies that  we can take $\ell_1=\ell_3$.
This fact is then used to handle Claim~1 in the proof of Proposition~\ref{P:pretentiousfinite'}, and the rest of the argument remains unchanged.  To see how Claim~1 is handled, note that in our setting the expressions 	$L_{k_1,k_2,k_3,k_2'}(Q)$ in \eqref{E:Lk123} take the form
\begin{multline}\label{E:Lk123abc}
	L_{k_1,k_2,k_3,k_2'}(Q):=c_{k_1,k_2,k_3}\cdot \lim_{l\to\infty} \E_{m,n\in [N_l], m>n}  \,   f^{k_1}((Qm+1)^2-(Qn)^2)\cdot f^{k_2}(2(Qm+1)(Qn))\cdot \\
	f^{k_3}((Qm+1)^2+(Qn)^2)\cdot \tilde{\chi}^{k_2'}(2(Qn)),
\end{multline}
where
$$
c_{k_1,k_2,k_3}:=(f^{k_1}\cdot \tilde{\chi}^{k_1'})(\ell_1)\cdot (f^{k_2}\cdot \tilde{\chi}^{k_2'})(\ell_2)\cdot  (f^{k_3}\cdot \tilde{\chi}^{k_3'})(\ell_3).
$$
Using additionally that $\ell_1=\ell_3$ and that  $f^{k_2}\cdot \tilde{\chi}^{k_2'}=1$,  $f^{k_1+k_2+k_3}\cdot \tilde{\chi}^{k_1'+k_2'+k_3'}=1$, which are standing assumptions in Claim 1, we deduce
that $c_{k_1,k_2,k_3}=1$. With this information at hand, the proof of Claim~1 in our setting
is exactly the same as in the case of Pythagorean triples.

Now suppose  that $a+b=c$, in which case the argument is  a bit different and somewhat simpler. As shown  in Step~2 of \cite[Appendix~C]{FH17}, we can obtain parametrizations of \eqref{E:abcz}  of the form
$$
x=k\, (m+\ell_1n)\cdot (m+\ell_2n), \quad y=k\, (m+\ell_3n)\cdot (m+\ell_4n), \quad z=k\, (m^2+(\ell_5n)^2),
$$
for suitable  $\ell_1,\ell_2,\ell_3,\ell_4, \ell_5\in \N$ that satisfy  $\ell_1\neq \ell_2$, $\ell_3\neq \ell_4$, and $\{\ell_1,\ell_2\}\neq \{\ell_3,\ell_4\}$. Note that our assumption $a+b=c$ was used to ensure that the coefficient of $m$ is $1$ in all linear forms.
We average on the grid $\{(Qm+1,Qn)\colon m,n\in \N\}$.
We will demonstrate how Claims~1 and 2 in the proof of Proposition~\ref{P:pretentiousfinite'} can be established within our framework. The remainder of the argument remains unaltered.
In our context, the expressions 	$L_{k_1,k_2,k_3,k_2'}(Q)$ in \eqref{E:Lk123} take the form
\begin{multline}\label{E:Lk123abc'}
	L_{k_1,k_2,k_3,k_2'}(Q):= \lim_{l\to\infty} \E_{m,n\in [N_l], m>n}  \,   f^{k_1}\big((Q(m+\ell_1n)+1) (Q(m+\ell_2n)+1)\big)\cdot \\
	f^{k_2}\big((Q(m+\ell_3n)+1) (Q(m+\ell_4n)+1)\big)\cdot 	f^{k_3}((Qm+1)^2+(Q\ell_5n)^2).
\end{multline}
Using the concentration estimates of \cref{C:concentration2},  we can  see   that Claim~1
holds without  assuming that $ f^{k_2}\cdot \tilde{\chi}^{k_2'}=1$. Therefore, in our setting, Claim~2 in the proof of Proposition~\ref{P:pretentiousfinite'} is already addressed by this case and  requires no further explanation.

\end{document}